\documentclass[11pt,reqno]{amsart}

\usepackage{amsmath,amsthm,amscd,amssymb,amsfonts}
\usepackage[margin=1.14 in]{geometry}
\usepackage{subfig}
\usepackage{graphicx}
\usepackage{hyperref}
\usepackage{enumitem}
\usepackage[labelfont=bf]{caption}
\allowdisplaybreaks
\usepackage{color}
\usepackage{mathtools}
\usepackage{bbm}
\usepackage{xcolor}
\usepackage{tikz,tikz-cd}

\allowdisplaybreaks

\theoremstyle{plain}

\newtheorem{theorem}{Theorem}[section]
\newtheorem*{theorem*}{Theorem}
\newtheorem{lemma}[theorem]{Lemma}
\newtheorem{corollary}[theorem]{Corollary}
\newtheorem{proposition}[theorem]{Proposition}

\theoremstyle{definition}
\newtheorem{definition}[theorem]{Definition}

\theoremstyle{remark}
\newtheorem{remark}[theorem]{Remark}

\numberwithin{equation}{section}

\newcommand{\bA}{\mathbb{A}}

\newcommand{\bC}{\mathbb{C}}

\newcommand{\bE}{\mathbb{E}}

\newcommand{\bP}{\mathbb{P}}
\newcommand{\bR}{\mathbb{R}}

\newcommand{\bW}{\mathbb{W}}

\newcommand{\cA}{\mathcal{A}}
\newcommand{\cB}{\mathcal{B}}
\newcommand{\cC}{\mathcal{C}}
\newcommand{\cD}{\mathcal{D}}

\newcommand{\cF}{\mathcal{F}}

\newcommand{\cL}{\mathcal{L}}

\newcommand{\cU}{\mathcal{U}}

\newcommand{\NCPd}{{\rm NCP}_d }

\newcommand{\tr}{{\rm tr}\,}
\newcommand{\trn}{{\rm tr}_n\,}

\DeclareMathOperator{\sa}{sa}
\DeclareMathOperator*{\median}{\ast} 

\DeclarePairedDelimiter{\ip}{\langle}{\rangle}
\DeclarePairedDelimiter{\norm}{\lVert}{\rVert}

\title{Viscosity Solutions in Non-commutative Variables}

\author{Wilfrid Gangbo$^1$}
\address{$^1$Department of Mathmatics, University of California, Los Angeles}
\email{\href{mailto:wgangbo@math.ucla.edu}{wgangbo@math.ucla.edu}}

\author{David Jekel$^2$}
\address{$^2$Department of Mathematical Sciences, University of Copenhagen}
\email{\href{mailto:daj@math.ku.dk}{daj@math.ku.dk}}

\author{Kyeongsik Nam$^3$}
\address{$^3$Department of Mathematical Sciences, Korea Advanced Institute of Science and Technology}
\email{\href{mailto:ksnam@kaist.ac.kr}{ksnam@kaist.ac.kr}}

\author{Aaron Z. Palmer$^4$}
\address{$^4$Department of Mathmatics, University of California, Los Angeles}
\email{\href{mailto:azp6@cornell.edu}{azp6@cornell.edu}}

\date{\today}

\begin{document}
	
	\begin{abstract}
		Motivated by parallels between mean field games and random matrix theory, we develop stochastic optimal control problems and viscosity solutions to Hamilton-Jacobi equations in the setting of non-commutative variables.  Rather than real vectors, the inputs to the equation are tuples of self-adjoint operators from a tracial von Neumann algebra.  The individual noise from mean field games is replaced by a free semi-circular Brownian motion, which describes the large-$n$ limit of Brownian motion on the space of self-adjoint matrices.  We introduce a classical common noise from mean field games into the non-commutative setting as well, allowing the problems to combine both classical and non-commutative randomness.
	\end{abstract}


	\maketitle
	
	\section{Introduction}
	
	\subsection{Context and motivation}
	
	We aim to develop free non-commutative analogs of mean fields games, stochastic control, and Hamilton--Jacobi--Bellman equations.  Random matrix theory has a long history since its appearance in 1928, where empirical covariance matrices of measured data naturally form a random matrix ensemble and the eigenvalues play a crucial role in principal component analysis \cite{wishart1928generalised}.  Wigner studied the spectral distributions of random matrices with independent entries (and especially the \emph{Gaussian unitary ensemble} or GUE) motivated by the observed similarity of random spectral distributions to energy levels of atomic nuclei \cite{wigner1967random}. {Mean field games, introduced by Caines--Malham{\'e}--Huang \cite{huang2006large} and Lasry--Lions \cite{lasry2007mean}, describe competitive interactions between a large number of agents through the continuum approximation using their bulk density;  see also Gomes--Sa\'{u}de \cite{gomes2014meanfield} for a survey of mean field game models and Carmona-Delarue \cite{carmona2018probabilistic} for a more exhaustive coverage. The study of mean field games deploys and advances methods from statistical physics, and is motivated by applications to modeling social behavior, economics, and finance.  An influential mathematical development of Cardaliaguet--Delarue--Lasry--Lions \cite{cardaliaguet2019master} was an infinite-dimensional PDE approach to the convergence of the $n$-player games to the mean field game limit (see also \cite{gangboMesz,GangboSwiech2015}). The infinite-dimensional mean field game PDE has also been developed using analytic solutions by Gangbo--M{\'e}sz{\'a}ros--Mou--Zhang \cite{gangbo2022mean} and  Bansil--M{\'e}sz{\'a}ros--Mou \cite{BansilMeszM} (see also \cite{JakobsenRutkowski}). 
		
		Connections between mean field games and random matrix theory have been apparent for some time.  For instance, recent studies in game theory were motivated by the empirical observations that the distinctive behavioral patterns of bus drivers in some cities relate to previously observed experiments in quantum physics. The dynamics of these bus patterns correspond to paths of eigenvalues for $n\times n$ self-adjoint random matrices.  In the same vein, the Dyson game \cite{carmona2020dyson} described a mean field games approach matrix Brownian motion where the positions of $n$ players correspond to the eigenvalues of the matrix.  Hence, the empirical distribution of the player positions represents the empirical spectral distribution of the matrix, i.e.\ the measure with mass $1/n$ at each of the eigenvalues.  The Nash--optimal trajectories of the game are given by Dyson Brownian motions, first introduced by Freeman Dyson \cite{dyson1962brownian}. The studies in \cite{carmona2020dyson} and \cite{dyson1962brownian} connect to a rich area of dynamical systems such as the theory of  Calogero--Moser--Sutherland models which involve completely integrable Hamiltonians; see Menon \cite{menon2017complex} for further discussion.  The Dyson Brownian motion also has a natural Riemannian interpretation using motion by mean curvature which was discovered by Huang, Inauen, and Menon \cite{Menon.et.al}.  The connections with mean field games and more generally PDE theory led to a new approach for analyzing the spectrum of single-matrix models developed by Bertucci--Debah--Lasry--Lions \cite{bertucci2022spectral}.
		
		We want to explore the connection between mean field games and random matrix theory in the setting of multimatrix models with several self-adjoint matrices that do not necessarily commute with each other.  Unlike a single self-adjoint matrix, the empirical distribution of $d$-tuple of self-adjoint matrices cannot simply be described by the positions of $n$ eigenvalues or any classical measure on $\mathbb R^d$; it is fundamentally non-commutative, so that one is forced to deal directly with the non-commutative moments $\tr(X_{i_1} \dots X_{i_k})$ or to formulate some analog of smooth test functions for non-commuting variables.  The large-$n$ limiting behavior of such multimatrix models is the domain of free probability theory, developed in large part by Voiculescu, where the role of probabilistic independence is played by free independence modeled on the behavior of free products of groups \cite{voiculescu1991limit,voiculescu1998strengthened}.  The large-$n$ limits are modeled by tuples $(X_1,\dots,X_d)$ of non-commuting random variables, which are operators in a von Neumann algebra, i.e.\ a non-commutative measure space.
		
		The difference between the classical and non-commutative settings can be described in terms of the choice of observables and the symmetry group acting on the system as follows. In the classical case, the observables correspond to commuting (diagonal) matrices with the symmetry of the permutation group representing interchangeability. In the non-commutative case, the observables are self-adjoint matrices or operators (as in quantum mechanics), and the symmetry group is the unitary group representing invariance under a change of coordinate basis.  In the von Neumann algebraic framework, the symmetries can be understood as trace-preserving automorphisms and, more generally, inclusions of von Neumann algebras. In our framework, this arises from the assumption that problem data is given by `tracial $\mathrm{W}^*$-functions,' along with structural assumptions about the model noise, which we discuss later.
		
		Despite the inherent non-commutative nature of multimatrix models that makes them less amenable to classical techniques, there is ample motivation to relate multimatrix models with mean field games, especially the role of non-commutative stochastic analysis and stochastic optimization problems in free probability.   Biane--Capitaine--Guionnet \cite{biane1998stochastic}  used a stochastic optimization problem to study large deviations theory for matrix Brownian motion, and hence also free entropy.  This analysis was carried further by Dabrowski in the study of entropy for free Gibbs laws \cite{Dabrowski2017Laplace}, and the related Hamilton--Jacobi--Bellman equation appeared in \cite{jekel2020elementary}.  These works occur in the larger context of free stochastic analysis \cite{biane1998stochastic,guionnet2009free,dabrowski2010path,jekel2023martingale}, free entropy and information theory \cite{voiculescu1993analogues1,voiculescu1994analogues2,Voiculescu1998analogues5,voiculescu1999analogues6,hiai2004free,hiai2006free}, and free optimal transport \cite{biane2001free,guionnet2014freemonotone,jekel2022tracial,gangbo2022duality,jekel2024optimal} (and the bibliography listed here is far from exhaustive).  We also briefly note the emerging research in matrix-valued optimal transport \cite{chen2017matrix, chen2020matrix} as well as entropic optimal transport for density operators in the quantum setting, e.g.\ \cite{carlen2017gradient,depalma2021quantumchannels,feliciangeli2023non}; although this is distinct from the free probability setting studied here, the similarity of approaches in \cite{duvenhage2022quadratic,gangbo2022duality} suggests potential connections.
		
		This work will formulate mean-field-like stochastic control problems in a free probability setting motivated by the large-$n$ behavior of random matrix theory, and study the resulting Hamilton--Jacobi equations in tracial von Neumann algebras; these can be understood as a non-local PDE in the single matrix case and the multimatrix case they are evolution equations in Hilbert space.  In a subsequent paper, we plan to show that under appropriate convexity assumptions our free model accurately describes the large-$n$ limiting behavior of the associated random matrix models.
		
		Our model will incorporate, for the first time in the free probability setting, both \emph{common noise} and \emph{free individual noise}. To motivate this feature, we consider the analogous models of mean field games that incorporate common noise and classical individual noise. The classical individual noise is modeled by independent Brownian motions, which lead to a second order Laplacian term in the evolution of the player densities. The individual noise captures variation and uncertainty in the individual player dynamics.  The common noise is modeled by a single Brownian motion that is shared in the dynamics of all the players, which results in a stochastic term in the partial differential equation for the player densities.  The common noise captures environmental uncertainty and exogenous variations in the system.  Both these sources of noise are important for financial applications as well as population modeling.  The individual noise is also motivated from physics by the need to capture unstructured interactions between particles in the nucleus of an atom.  Whether sources of non-commutative randomness that arise in games or other systems would also be captured by similar freely independent noise is an intriguing aspect of the problem for further study into applications. For background on the role of individual and common noise in mean field games, we refer the reader to Carmona-Delarue \cite{carmona2018probabilistic}; Chapter 1.4 for economic applications and Chapter 2.1 for a typical mathematical set-up. 
		
		The study of Hamilton--Jacobi equations on infinite dimensional spaces started several decades ago, for instance in Hilbert space settings  \cite{Crandall-Lions85,Crandall-Lions86,Crandall-Lions86b}. This was followed by studies on the Wasserstein space, whose differential structure has been proven to coincide with that of a quotient Hilbert space \cite{gangbo2019}.
		The study in the current  manuscript is the non commutative analog of the recent developments which appeared in \cite{AmbrosioF2014,BayraktarEkrenZhang,cosso2024master,GangboMayorgaS,GangboSwiech,gangboNguyenT}, dealing with either no noise, the common noise or the individual noise. In this study, we trade the Wasserstein space with the set of non commutative laws.

		\subsection{Control problems in the multimatrix and von Neumann algebraic settings}
		
		For concreteness, we first describe the stochastic control problems we are interested in the setting of $n \times n$ matrix tuples before moving on to the general von Neumann algebraic setting.  The state space will be $M_n(\bC)_{\sa}^d$, the set of $d$-tuples of self-adjoint $n \times n$ complex matrices.  In the mean field games analogy, due to the non-commutative nature of the problem, one cannot isolate individual players (just as a quantum graph does not have a distinct set of vertices), but conceptually we may think of the multimatrix as describing ``$M_n(\bC)$-many points'' in $\bR^d$.  We are interested in the large $n$ behavior of stochastic control problems on $M_n(\bC)_{\sa}^d$, i.e.\ we take $n \to \infty$ with $d$ fixed.
		
		The free individual noise in the multimatrix setting comes from a Brownian motion $\widehat{W}_t^n$ on the space of $d$-tuples of self-adjoint matrices.  For each $j = 1$, \dots, $d$, the $\widehat{W}_t^{n,j}$ is a random Gaussian self-adjoint matrix whose entries are independent up to Hermitian symmetry.  The distribution of $t^{-1/2} \widehat{W}_t^{n,j}$ is the \emph{Gaussian unitary ensemble}, so called because its distribution is invariant under unitary conjugation.  Hence, we refer to $W_t^{n,j}$ and $W_t^n$ as a \emph{GUE$(n)$ Brownian motion}.  The large-$n$ limiting behavior of this process is described by a free Brownian motion (see e.g.\ \cite{biane1997freebrownian,biane1998stochastic,biane2003large}).  The free Brownian motion is a process $S_t = (S_t^1,\dots,S_t^d)$ of self-adjoint operators in a von Neumann algebra $\mathcal{A}$ (a certain type of algebra of bounded operators on a Hilbert space) with a tracial state $\tau: \mathcal{A} \to \bC$ representing the expectation; these operators are not random in the classical sense but come from a non-commutative analogue of probability spaces.  The spectral distribution of each $t^{-1/2} S_t^j$ is given by Wigner's semicircular measure $(1/2\pi) \mathbf{1}_{[-2,2]}(x) \sqrt{4 - x^2}\,dx$, which  is the almost sure limit of the empirical spectral distribution of the GUE matrix (see e.g.\ \cite[\S 2]{anderson2010introduction}).  The different coordinates and the different increments of the free Brownian motion are \emph{freely independent}, a non-commutative form of independence named for its relationship with free products of groups and operator algebras (see e.g.\ \cite{voiculescu1992freerandom}, \cite[\S 5]{anderson2010introduction}, and \S \ref{subsec: free independence}).  Voiculescu's theory of asymptotic freeness shows that the joint moments $\tr_N(W_{t_1}^{n,j_1} \dots W_{t_k}^{n,j_k})$ converge almost surely to the deterministic limit $\tau(S_{t_1}^{j_1} \dots S_{t_k}^{j_k})$ \cite{voiculescu1991limit,voiculescu1998strengthened}.  The deterministic limiting behavior arises from high-dimensional concentration of measure (see \cite{voiculescu1998strengthened,guionnet2000concentration,meckes2013}). 
		
		Meanwhile, the common noise in the multimatrix setting is still a single \emph{classical} Brownian motion that is shared along the diagonal of each of the matrices.  While the individual noise is described by a deterministic limit using free Brownian motion, the classical randomness of common noise persists in the large-$n$ limit.  Hence, the equation that we formulate in the free probability setting incorporates both non-commutative randomness in the individual noise and classical randomness in the common noise; this combination itself is a new technical challenge that has not been studied much before.
		
		The multimatrix stochastic control problems are formulated as follows.  We fix two nonnegative parameters $\beta_C$ and $\beta_F$, which we refer to respectively as the \emph{common noise} and the \emph{free individual noise} parameters.  Fix a probability space $(\Omega, \mathbb P)$ and a standard one-dimensional Brownian motion $(W^0_t)_{t \in [0,T]}$ on $\Omega$, which we embed in the matrix algebra by multiplying by the identity matrix $\mathbf{1}_{M_{n}(\mathbb C)}$. We let $\big(\widehat{W}^{n,j}_t\big)_{t \in [0, T]}^{j=1, \cdots, d}$ denote $d$ many i.i.d. GUE($n$) Brownian motions on $\Omega$, independent of $(W^0_t)_{t \in [0,T]}$. In order to feature the simplest example of the problems of interest, let us start with a time-dependent random control $(\alpha^{n}_t)_{t\in [0,T]}$ taking values in $M_{n}(\mathbb C)^d_{\rm sa}$. Under appropriate conditions, there exists a unique random path $t \in (t_0, T) \mapsto X_t^{n,j}[\alpha^n]$,  a solution to the stochastic differential equation 
		\begin{equation}\label{eq:d1}
			dX_t^{n,j}[\alpha^n] =  \alpha_t^{n,j}\, dt + \beta_C\, \mathbf{1}_{M_{n}(\mathbb C)} dW_t^0 + \beta_F\, d\widehat{W}^{n, j}_t,\quad j \in \{1,\ldots, d\},
		\end{equation}
		where at the initial time $(X_{t_0}^{n,j}[\alpha])_{j=1,\cdots,d}$ is a prescribed deterministic  $d$--tuple $x_0^n\in M_{n}(\mathbb C)^d_{\rm sa}$. Given a Lagrangian  $L_{M_n(\mathbb C)}$ 
		representing a running cost and a function $g_{M_n(\mathbb C)}$ representing a terminal value function, we search for an optimal control, in an admissible class $\widehat{\mathbb{A}}_{M_{n}(\mathbb C)}^{t_0,T}$, for the problem   
		\begin{equation}\label{eq:d2}
			\widehat {V}_{M_n(\mathbb C)}(t_0,x_0^n) =  \inf_{\alpha \in \widehat{\mathbb{A}}_{M_{n}(\mathbb C)}^{t_0, T}}\Bigg\{\mathbb E\bigg[\int_{t_0}^T L_{M_n(\mathbb C)} \big(X_t[\alpha], \alpha_t\big)dt + g_{M_n(\mathbb C)}\big(X_T[\alpha]\big)\bigg]:\  X_{t_0}[\alpha]=x_0^n\Bigg\}.
		\end{equation}

		Our goal is to study the large-$n$ limit of the value function analogously to the mean-field limit where the number of players $n$ tends to infinity, but in the multimatrix setting, we have to first discuss what we even mean by convergence.  In the mean field setting, convergence is understood as convergence of the empirical measures averaged over the players to a limit measure.  Similarly, for a single matrix, one studies convergence of the empirical spectral measure.  But as mentioned before, the joint distributions of non-commuting random variables are not described by classical measures. Rather, the non-commuting random variables are elements $(X_1,\dots,X_d)$ in a von Neumann algebra $\mathcal{A}$ (a certain type of algebra of bounded operators on a Hilbert space) with a tracial state $\tau: \mathcal{A} \to \bC$ representing the expectation.  The probability distribution of these operators is its \emph{non-commutative law} described as the linear functional $\lambda_X$ on non-commutative polynomials $\bC\ip{x_1,\dots,x_d}$ sending each test polynomial $p$ to its expectation $\tau(p(X_1,\dots,X_d))$ when evaluated on our given tuple of random variables.  The space of non-commutative laws of $d$-tuples $(X_1,\dots,X_d)$ with $\sup_j \norm{X_j}_\infty \leq R$ has a weak-$*$ topology that makes it into a compact space.  Analogous to classical probability, one can think of the space of non-commutative laws as a quotient space consisting of equivalence classes of $d$-tuples of random variables, where for von Neumann algebras $\mathcal{A}$ and $\cB$ along with  tracial states $\tau_\cA$ and $\tau_\cB$, we say that    $X \in L^2(\mathcal{A})_{\sa}^d$ and $Y \in L^2(\mathcal{B})_{\sa}^d$ with $\norm{X}_\infty, \norm{Y}_\infty \leq R$ satisfy $X \sim Y$ if $\tau_{\mathcal{A}}(p(X)) = \tau_{\mathcal{B}}(p(Y))$ for all $p \in \bC\ip{x_1,\dots,x_d}$.  If $X, Y \in M_n(\bC)_{\sa}^d$, then $X \sim Y$ if and only if there is a unitary $U$ such that $UX_jU^* = Y_j$ for all $j$ (although this is not exactly true for tuples in general von Neumann algebras).

		Thus the question of convergence for the multimatrix control problems can be formulated as follows:  Suppose that $x_0^n$ is a $d$-tuple of self-adjoint $n \times n$ matrices and the operator norms of $x_0^n$ are bounded as $n \to \infty$.  Suppose that the non-commutative laws of $x_0^n$ converge in the weak-$*$ topology.  Then do the value functions $\widehat{V}_{M_n(\mathbb{C})}(t_0,x_0^n)$ converge as $n \to \infty$, and if so, how do we describe the limit?  This paper will formulate a candidate for the limit in the setting of free probability and tracial von Neumann algebras.  We will set up the stochastic control problems, prove a dynamic programming principle, and define viscositiy solutions.  Because these developments alone are sufficiently involved, we will address convergence of the matrix models in the large-$n$ limit in a follow-up paper.

		The infinite-dimensional analogue of the stochastic control problem \eqref{eq:d1} is as follows.  Let $\cA = (A,\tau)$ be a tracial von Neumann algebra, i.e.\ a von Neumann algebra $A$ with a chosen (faithful normal) tracial state $\tau$.  Let $L^2(\mathcal{A})$ be the completion of $A$ with respect to $\norm{x}_{L^2(\cA)} = \tau(x^*x)^{1/2}$. Assume that $\mathcal A$ is large enough to contain a family of freely independent semi-circular Brownian motions $S_t^j$, for $j = 1$, \dots, $d$, that are freely independent from our chosen initial condition $x_0 \in L^2(\mathcal{A})_{\sa}^d$.  Moreover, the control $(\alpha_t)_{t\in[0,T]}$ will be a process taking values in $L^2(\mathcal A)^d_{\rm sa},$ the subset of $d$-tuples of self-adjoint operators in $L^2(\mathcal A)$, satisfying appropriate independence conditions (namely, the increments of the free Brownian motion $S_t$ for $t \geq t_0$ are freely independent of the values of $\alpha_t$ for $t \leq t_0$, and the increments of the common Brownian motion $W_t^0$ for $t \geq t_0$ are probabilistically independent of $\alpha_t$ for $t \leq t_0$).  Then let $X_t = (X_t^j)_{j=1,\dots,d}$ be the solution to the stochastic differential equation
		\begin{equation}\label{eq:d1noncom}
			dX^j_t = \alpha_t^j\, dt + \beta_C\, \mathbf{1}_{\mathcal A}\, dW_t^0 + \beta_F\, dS_t^j,\quad j \in \{1,\ldots, d\},
		\end{equation} 
		with initial condition $X_{t_0}[\alpha] = x_0 \in L^2(\cA)^d_{sa}$.  We emphasize that in contrast to the expression multiplied by $\beta_F$ in \eqref{eq:d1} being stochastic, the analogous expression in \eqref{eq:d1noncom} with the free Brownian motion is not random in the classical sense.
		
		In the special case  $d=1$, we may understand solutions to (\ref{eq:d1noncom}) in terms of their law $\mu_t\in \mathcal{P}_2(\mathbb{R})$, where $\mathcal{P}_2(\mathbb{R})$ denotes the set of probability measures having finite second moments.  Equation (\ref{eq:d1noncom}) expresses the following stochastic non-linear non-local partial differential equation, supposing $\alpha_t = a_t(X_t)$,
		$$
		d\mu_t = -\partial_x(a_t\, \mu_t)dt -\partial_x(\beta_C\, \mu_t)dW_t^0 +\frac{\beta_C^2}{2}\partial_x^2 \mu_t\, dt + \frac{\beta_F^2}{2} \partial_x( h[\mu_t]\, \mu_t) dt,
		$$
		where $h[\mu](y) := \text{p.v.}\, \int_{\mathbb{R}} \frac{\mu(dx)}{y-x}$ denotes the Hilbert transform (see \S \ref{subsec: cylindrical}).
		
		In light of the existing theory of viscosity solutions on Hilbert spaces, given a family $(L_{\mathcal A})_{\mathcal A}$ of Lagrangians, a family $(g_{\mathcal A})_{\mathcal A}$ of value functions, an admissible set of control policies $\mathbb{A}_{\cA}^{t_0,T}$   and $x_0 \in L^2(\mathcal A)^d_{\rm sa},$ it is tempting to set 
		\begin{equation}\label{eq:d4-b}
			\widetilde V_{\mathcal A}(t_0, x_0): =   \inf_{\alpha \in \mathbb{A}_{\cA}^{t_0, T}}\Bigg\{\mathbb E\bigg[\int_{t_0}^T L_{\mathcal A} \big(X_t[\alpha], \alpha_t\big)dt + g_{\mathcal A}\big(X_T[\alpha]\big)\bigg], \; X_{t_0}[\alpha]=x_0\Bigg\}.
		\end{equation}
		The problem with this na{\"\i}ve formulation is that $\widetilde{V}_{\cA}$ lacks the consistency properties necessary to make it well-defined on the quotient spaces of interest; it is not invariant under embedding the non-commutative probability space $\cA$ into a larger one $\cB$ even if we assume such invariance for the functions $L$ and $g$. Specifically, assume that the families of Lagrangians $(L_{\mathcal A})_{\mathcal A}$ and value functions $(g_{\mathcal A})_{\mathcal A}$ satisfy the consistency property that   
		\[
		g_{\mathcal B}(\iota x_0)  =  g_{\mathcal A}(x_0), \quad \hbox{for all}\  x_0 \in L^2(\mathcal A)^d_{\rm sa}
		\] 
		for any $W^*$--algebra $\cB$ and a $W^*$--embedding $\iota: \mathcal A \to \mathcal B$. In general,
		\[
		\widetilde  V_{\mathcal B}(t_0,  \iota(x_0)) \not= \widetilde V_{\mathcal A}(t_0,  x_0).
		\]
		This behavior should be contrasted with the classical setting, in which any probability space that supports a Brownian motion independent of the initial condition would yield the same answer for stochastic optimization problems.  The issue of different possible behaviors in different larger tracial von Neumann algebras already arose in \cite{gangbo2022duality} in the study of Monge--Kantorovich duality for the non-commutative $L^2$-Wasserstein metric of Biane and Voiculescu \cite{biane2001free}.  To obtain invariance for our value function, we will take the infimum over all possible embeddings into larger tracial von Neumann algebras, and hence the relevant functions to study are
		\begin{equation}\label{eq:d4-c}
			\overline{V}_{\mathcal A}(t_0, x_0) :=   \inf_{(\mathcal B,\, \iota)} \Big\{ \widetilde V_{\mathcal B}(t_0, \iota x_0): \iota: \mathcal A \to \mathcal B \; \text{is a} \; \text{tracial } W^*  \text{-embedding} \Big\}.
		\end{equation} 
		In order to show that this function has the desired invariance property, we use a novel joint embedding lemma that allows us to embedding two given von Neumann algebras with non-commutative filtrations into a larger one while identifying the two copies of the free Brownian motion and initial condition (see \S \ref{subsec: value function}).

		
		We aim to show that the value function $\overline{V}$ is a viscosity solution for a certain Hamilton--Jacobi equation, which also requires formulating the definition of viscosity solution for the non-commutative setting.  Abstractly, the Hamilton-Jacobi equation appears on the space of non-commutative laws as
		\begin{align*}
			-\partial_t V(t,\lambda) + H\big(\lambda, -\partial V(t,\lambda)\big) - \frac{\beta_C^2}{2}\, \Delta V(t,\lambda)- \frac{\beta_F^2}{2}\, \Theta V(t,\lambda)=&\ 0,\\
			V(T,\lambda) =&\ g(\lambda).
		\end{align*}
		In this equation, $H=(H_\cA)_{\cA}$ is a tracial $W^*$--function of the state and generalized momentum, precisely defined in (\ref{eqn:Hamiltonian_equivalence}) later. The operator $\Delta$ is a common noise Laplacian, which corresponds to a second-order differential operator, for example, defined on the functions of the Hilbert spaces $C^2(L^2(\cA)_{\sa}^d)$.  Finally, the operator $\Theta$ is a free individual noise Laplacian, which is defined by means of introducing a freely independent free Brownian motion, and hence also requires enlarging the tracial von Neumann algebra and relies on invariance under embeddings.
		
		We close with some brief comments on convergence of the multimatrix value functions to the free limit, which will be studied in a follow-up paper.  First, it is essential to assume stronger continuity properties for the functions $L$ and $g$ than simply their being Lipschitz tracial $\mathrm{W}^*$-functions.  The reason is that, when $d > 1$, the non-commutative Wasserstein distance studied by Biane and Voiculescu \cite{biane2001free} gives a much stronger topology than the weak-$*$ topology on the space of non-commutative laws of variables bounded by $R$, and in fact it is impossible for the non-commutative laws of multimatrices to converge in Wasserstein distance unless the limiting von Neumann algebra is amenable as shown in \cite[\S 5.4-5.5]{gangbo2022duality}.  Hence, for convergence we would assume weak-$*$ continuity for $L$ and $g$ in addition to the hypotheses in this paper.  Secondly, we must also deal with the fact that enlarging the von Neumann algebra could alter the value of the infimum, hence the large-$n$ limit of the infima in the matrix models can in general be greater than the infimum given in $\overline{V}$.  We will remove this issue by making a stronger convexity assumption  on $L$ and $g$ called $E$-convexity \cite{gangbo2022duality}; this is analogous to the way that some convergence results in mean field games require displacement convexity for Hamilton--Jacobi equation \cite{gangboMesz} or Lasry--Lions monotonicity condition \cite{cardaliaguet2019master}.

		\subsection{Overview of results and organization}
		
		We approach the problem of understanding Hamilton-Jacobi equations on non-commutative spaces by first defining optimal control problems in von Neumann algebras.  We start in \S \ref{sec: prelim} with the necessary background on tracial $\mathrm{W}^*$ (von Neumann) algebras, amalgamated free products, and non-commutative laws (which play the role of probability distributions in this theory).  In \S \ref{sec:problem}, we describe the setup of these stochastic optimization problems.  After introducing motivating examples in \S \ref{subsec: example problems}, we spell out the general assumptions in \S \ref{sec: setup}, and then develop the properties of the Hamiltonian in \S \ref{subsec: Hamiltonian properties} and the value function in \S \ref{subsec: value function}.
		
		Our framework relies on the notion of a \emph{tracial $\mathrm{W}^*$-function}, that is, a real-valued function that is defined in a consistent way on all tracial $\mathrm{W}^*$-algebras, so that the output only depends on the non-commutative law of the input, similar to how probabilistic phenomena are independent of the choice or probability space.  We also use the related notion of tracial $\mathrm{W}^*$-vector fields.  Like tracial $\mathrm{W}^*$-function, these are defined in terms of consistency with respect to tracial $\mathrm{W}^*$-embeddings of a tracial $\mathrm{W}^*$-algebra into a larger one.  We shows that although the na{\"\i}ve value function $\widetilde{V}_{\mathcal{A}}$ is not necessarily a tracial $\mathrm{W}^*$-function, $\overline{V}_{\mathcal{A}}$ will be under our assumptions.  We assume Lipschitz conditions for the Lagrangian and terminal cost functions, and establish that the resulting control-theoretic Hamiltonian and value function are tracial $\mathrm{W}^*$-functions and satisfy a similar Lipschitz estimate.  While our assumptions are far from the greatest possible generality, we are guided by a handful of diverse examples, including an Eikonal equation, Linear-Quadratic-Gaussian framework, and a controlled von Neumann equation.
		
		In addition, we study \emph{$E$-convexity}, a convexity condition from \cite{gangbo2022duality} that also takes into account tracial $\mathrm{W}^*$-embeddings $\iota: \mathcal{A} \to \mathcal{B}$ by demanding that if $E: \mathcal{B} \to \mathcal{A}$ is the conditional expectation adjoint to $\iota$, then $f^{\mathcal{A}} \circ E \leq f^{\mathcal{B}}$.  We show that $E$-convexity for the terminal and running cost functions implies that $\overline{V}_{\mathcal{A}} = \tilde{V}_{\mathcal{A}}$, thus in this case enlarging the von Neumann algebra does not produce any smaller infimum (see Lemma \ref{lem:decreasing}, and we also obtain $E$-convexity for the value function (see Lemma \ref{lem: decreasing 2}).

		After the development of free stochastic control problems in \S \ref{sec:problem}, we move on to the concept of viscosity solutions in \S \ref{sec: viscosity solutions}. We consider two notions of viscosity solution.  The first notion in \S \ref{subsec: intrinsic viscosity} is defined on the space of non-commutative laws and uses tracial $W^*$--functions as test functions.  We show that the value function is a sub- and super-solution in our definition by showing it satisfies the sub- and super-dynamic programming principle in Propositions \ref{prop:subdynamic_programming} and \ref{prop:dynamic_programming}.  In our framework, the subsolution property can already be shown for the functions $\tilde{V}_{\mathcal{A}}$ that work with stochastic processes in a fixed tracial $\mathrm{W}^*$-algebra $\mathcal{A}$, but the supersolution property requires consider all possible embeddings into a larger tracial $\mathrm{W}^*$-algebra.  This slight asymmetry is a new subtlety that arises in the non-commutative setting due to different possible behaviors of embeddings into larger tracial $\mathrm{W}^*$-algebras.
		
		The second notion of viscosity solution that we consider in \S \ref{subsec: Hilbert viscosity} is based on the standard notion of viscosity solution in Hilbert spaces applied to the non-commutative $L^2$ space $L^2(\cA)_{sa}^d$ associated to a tracial $\mathrm{W}^*$-algebra $\cA$.  In this setting, we do not handle the individual noise, only the common noise.  However, in this setting, we can show a comparison principle for viscosity solutions that implies uniqueness (Theorem \ref{thm:comparison}) based on the comparison principle in the existing Hilbert-space theory.
		
		Finally, Section \ref{sec:examples} discusses our main examples in further depth, including an Eikonal equation, Linear-Quadratic-Gaussian framework, and a controlled von Neumann equation. 
		
		The appendices provide necessary technical background and detail.  In \S \ref{subsec: diff eq appendix}, we give background on vector-valued and stochastic differential equations.  In \S \ref{apx:free_laplacian}, we explain how to compute the free probabilistic Laplacian used in our differential equations on a certain class of test functions.  In particular, this section provides a general class of cost functions satisfying our assumptions.  In \S \ref{sec: AFP}, we give background on amalgamated free products and show that two different non-commutative filtrations and Brownian motions can be jointly embedded into another non-commutative filtration such that the given Brownian motions are identified.  This is the technical machinery needed to establish the consistency of the value function $\overline{V}$ that makes it a well-defined tracial $\mathrm{W}^*$-function.  These computations may also be of interest for free probability in general.


		\subsection{Acknowledgements}
		
		W.G. was supported by NSF grant DMS-2154578 and Air Force grant FA9550-18-1-0502. D.J. was partially supported by the National Science Foundation (US) grant DMS-2002826, the National Sciences and Engineering Research Council (Canada) grant RGPIN-2017-05650, and the Independent Research Fund of Denmark grant 1026-00371B.  K.N. was supported by the National Research Foundation of Korea (RS-2019-NR040050). A.Z.P. also acknowledges the support of Air Force grant FA9550-18-1-0502. 
		
		We thank Dimitri Shlyakhtenko for numerous conversations that motivated this work.

		\section{Preliminaries} \label{sec: prelim}
		
		\subsection{Von Neumann Algebras}
		
		We first formulate the framework of non-commutative probability spaces, following similar conventions as in \cite{gangbo2022duality} (note that further bibliography on von Neumann algebras is given there as well). Recall  $A$ is a \emph{$W^*$--algebra} (von Neumann algebra) if $A$ is a unital $C^*$--algebra together with an operator norm $\|\cdot\|_\infty$ such that $A$ as a Banach space is the dual of some Banach space $A_*$.  If $\tau \in A^*$ is a faithful normal trace in the dual Banach space, then we call $\mathcal{A} = (A,\tau)$ a \emph{tracial $\mathrm{W}^*$-algebra} or a \emph{non-commutative probability space}.  The intuition behind non-commutative probability spaces is that the elements of $A$ are non-commutative random variables, and the trace $\tau$ is analogous to the expectation.
		
		The GNS construction \cite{segal1947irreducible} produces a Hilbert space $L^2(\cA)$ as follows:  A pre-inner product can be defined on $\cA$ by
		$$
		\langle X,Y\rangle_{L^2(\cA)} := \tau(X^*\, Y) , \hbox{ for } X,Y\in L^2(\cA).
		$$
		Faithfulness of $\tau$ implies that this is non-degenerate.  Thus, $\cA$ be completed to a Hilbert space $L^2(\cA)$; we continue to use the same notation $\ip{\cdot,\cdot}_{L^2(\cA)}$ for the inner product on the completion.  Moreover, for $d$-tuples $X \in L^2(\cA)^d$, we will indicate the component of the $d$-tuple with superscripts, as $X = (X^1,\cdots,X^d)$ and  $Y = (Y^1,\cdots,Y^d)$, and denote the inner product on $L^2(\cA)^d$ with the subscript $L^2(\cA)$ as
		$$
		\langle X,Y\rangle_{L^2(\cA)} := \sum_{j=1}^d\tau({X^j}^*\, Y^j).
		$$
		In either case, we define the Hilbert space norm by
		$$
		\|X\|_{L^2(\cA)} : = \sqrt{ \langle X,X\rangle_{L^2(\cA)}}.
		$$
		
		We let $L^\infty(\cA)\subset L^2(\cA)$ be the collection of elements in $L^2(\cA)$ that are bounded in the operator norm, and we may consider $A\subset L^\infty(\cA)$. The operator norm also naturally extends to $d$-tuples by
		$$
		\|X\|_\infty := \max_{j\in \{1,\ldots,d\}} \|X^j\|_\infty, \hbox{ for } X = (X^1,\cdots,X^d) \in L^\infty(\cA)^d.
		$$
		Let $\mathbf{1}_{\cA}$ be the identity element  in $L^2(\cA)$. We denote by $L^2(\cA)^d_{sa}$ the collection of  self-adjoint elements in $L^2(\cA)^d$. We will also consider $\mathbbm{1}_{\cA}\in L^2(\cA)^d_{sa}$ to be the $d$-tuple with the identity in each component, and $\mathbf{e}_{\cA}^j\in L^2(\cA)^d_{sa}$ to have the identity in the $j$th component and 0 in other components, for $j\in \{1,\ldots,d\}$.  


		Next, we introduce an important notion of {tracial $W^*$--embedding}.
		We say that $\iota:\cA=(A,\tau)\rightarrow \mathcal{B}=(B,\rho)$  is  a \emph{tracial $W^*$--embedding} if it is unital $*$ -homomorphism (it respects addition, multiplication, and adjoints) that is also trace-preserving, meaning that
		\begin{align*}
			\rho(\iota\, X) = \tau(X)\quad  \text{for any } X\in A.
		\end{align*} 
		For such a {tracial $W^*$--embedding}  $\iota:\cA \rightarrow \cB$, there is an induced linear isometry  $\iota:L^2(\cA)^d\rightarrow L^2(\mathcal{B})^d$ where we use the same notation. For a concise notation,   throughout the paper, when we write $\iota:\cA\rightarrow \cB$, we always assume that $\iota$ is a tracial $W^*$--embedding. Its adjoint, or the conditional expectation, is denoted by $E:\cB \rightarrow \cA$. 
		For more details on this setup, see \cite{gangbo2022duality}.
		
		Here are some examples:
		\begin{itemize}
			\item The space of $n\times n$ complex matrices,  $M_n(\bC)$, where the normalized trace (resp. inner product) is given by 
			\[
			\trn(A):=\frac{1}{n} \tr(A), \quad \langle A, B \rangle_{\trn}:=\trn(A^*B).
			\]
			In this case, we may identify $M_n(\bC) \cong L^2(\cA)\cong L^\infty(\cA)$ for $\cA=(M_n(\bC),\trn)$. 
			
			We denote by $\cU_n$ the set of unitary matrices $U \in M_n(\bC)$. For any unitary matrix $U$, it  defines a isometry by a conjugation $X \mapsto U^*\, X\, U$.  We say that two $d$--tuples $X$ and $Y$ of self-adjoint elements in $M_n(\bC)^d_{sa}$ are equivalent if there exists $U \in \cU_n$ such that $Y^j=U^* X^j U$ for all $j =1, \cdots, d.$ 
			
			\item We recover a classical probabilistic setting by fixing a probability space $(\Omega,\mathcal{F},\bP)$, where $\Omega$ is a separable metric space and $\mathcal{F}$ is the associated Borel $\sigma$-algebra. We consider the $C^*$-algebra $A= C(\Omega; \bC),$  the collection of continuous functions from $\Omega$ to $\mathbb{C}$.  Algebra products are given by  pointwise products, and the trace corresponds to the expectation,
			$$
			\tau(X) := \bE[X].
			$$
			The GNS construction produces the Hilbert space of equivalence classes of $\bP$-square-integrable functions, for $\cA=(A, \tau)$:
			$$
			L^2(\cA) \cong L^2(\Omega,\mathcal{F},\bP;\bC).
			$$
			The self-adjoint elements are simply the real-valued random variables. 
			
			\item Given a (discrete) group $G$, let $\ell^2(G)$ be its $\ell^2$ space. For each $g \in G$, let $\lambda(g) \in B(\ell^2(G))$ be the left translation by $g$.  Let $L(G)$ be the von Neumann subalgebra of $B(\ell^2(G))$ generated by $\{\lambda(g): g \in G\}$.  Let $\tau_G(x) = \ip{\delta_e, x \delta_e}_{\ell^2(G)}$.  One can show that $(L(G),\tau_G)$ is a tracial $\mathrm{W}^*$-algebra.  Free probability is motivated by examining what happens when the group $G$ is decomposed as a free product (see e.g.\ \cite[Example 5.3.3]{anderson2010introduction}).
		\end{itemize}
		
		Given a tracial von Neumann algebra $\cA = (A,\tau)$ and $S \subseteq A$, we denote $\mathrm{W}^*(S)$ the von Neumann subalgebra generated by $S$, i.e., the intersection of all tracial von Neumann algebras containing $S$, equipped with the trace $\tau|_{\mathrm{W}^*(S)}$.  Note that the inclusion map $\mathrm{W}^*(S) \to \cA$ is a tracial $\mathrm{W}^*$-embedding.  For tracial von Neumann subalgebras $(\cB_j)_{j \in J}$ in $\cA$, we write
		\[
		\bigvee_{j \in J} \cB_j = \mathrm{W}^*\bigg( \bigcup_{j \in J} \cB_j \bigg)
		\]
		for their \emph{join}, that is, the von Neumann subalgebra that they generate, equipped with the restriction of the trace $\tau$.
		
		\subsection{Free independence and free products} \label{subsec: free independence}
		
		In non-commutative probability, there is an analog of independence known as \emph{free independence} that relates closely to free products of tracial von Neumann algebras (and of groups).
		
		If $\cA=(A, \tau)$ is a tracial $\mathrm{W}^*$--algebra and $\{\cA_j=(A_j,\tau): j\in J\}$ are tracial $\mathrm{W}^*$-subalgebras of $\cA$ with an index set $J$, we say that $\{\cA_j: j\in J\}$ are \emph{freely independent} if for all positive integers $n$ and $j:\{1,\ldots,n\}\rightarrow J$ such that $j(k)\not= j(k+1)$  for $k=1,\cdots,n-1,$
		$$
		\tau\bigg( \prod_{k=1}^{n}\big(a_k-\tau(a_k)\big)\bigg)=0, \qquad \hbox{for all }(a_1,\ldots,a_{n}) \in A_{j(1)}\times\ldots \times A_{j(n)},
		$$
		where the terms in the product are understood to be multiplied in order from left to right; see \cite{voiculescu1985symmetries,voiculescu1992freerandom}.  More generally, if $\{\cA_j: j \in J\}$ is a collection of tracial $\mathrm{W}^*$-algebras in $\cA$ containing a common subalgebra $\cB$ and if $E_{\cB}$ denote the trace-preserving conditional expectation $\cA \to \cB$, then we say that $\{\cA_j: j \in J\}$ are \emph{freely independent with amalgamation over $\cB$} or \emph{freely independent over $\cB$} if for all positive integers $n$ and $j:\{1,\ldots,n\}\rightarrow J$ such that $j(k)\not= j(k+1)$   for $k=1,\cdots,n-1,$ 
		$$
		E_{\cB} \bigg( \prod_{k=1}^n\big(a_k-E_{\cB}(a_k)\big)\bigg)=0, \qquad \hbox{for all }(a_1,\ldots,a_{n}) \in A_{j(1)}\times\ldots \times A_{j(n)}.
		$$
		In the case where $\cB = \bC$, this reduces to plain free independence as defined above.  We remark that, given the inclusions $\cB \to \cA_j$ and $\tau|_{\cA_j}$, free independence with amalgamation over $\cB$ uniquely determines the trace of any product $a_1 \dots a_k$ where $a_i$ is from $\cA_{i_j}$ \cite[Proposition 1.3]{voiculescu1995operations}.

		As in classical probability, we often need to construct independent joins of non-commutative probability spaces.  Given tracial von Neumann algebras $\cA_j$ for $j \in J$ containing a common subalgebra $\cB$, there exists a tracial von Neumann algebra $*_{\cB}( \cA_j)_{j\in J}$ containing $\cB$  with trace-preserving unital $*$-homomorphisms $\iota_j: \cA_j \to *_{\cB} (\cA_j)_{j\in J}$ such that $\iota_j|_{\cB} = \operatorname{id}$ and $(\iota_j(\cA_j))_{j \in J}$ are freely independent over $\cB$ and $*_{\cB} (\cA_j)_{j \in J}$ is generated by $(\iota_j(\cA_j))_{j \in J}$.  This algebra $*_{\cB} (\cA_j)_{j \in J}$ is unique up to a canonical isomorphism (one respecting the inclusions $\iota_j$), and is called the \emph{free product of $\cA_j$ (with amalgamation) over $\cB$}.  In the case $\cB = \bC$, it is called simply the free product of the $\cA_j$'s and denoted $\median (\cA_j)_{j \in J}$.  Note also that in the case of two algebras, we will sometimes use the notation $\cA_1 \median_{\cB} \cA_2$.  For more background on free independence and free products with amalgamation, \cite[\S 5]{voiculescu1985symmetries}, \cite[\S 3.8]{voiculescu1992freerandom}, \cite[p. 384-385]{popa1993markov}, \cite[\S 1]{voiculescu1995operations}, \cite[\S III]{speicher1998combinatorial}, \cite[\S 4.7]{brown2008c*algebras}.  Because we sometimes deal with multiple different inclusions of the same two algebras, let us explicitly state what we mean by the existence and uniqueness of the free product.
		
		\begin{lemma} \label{lem: uniqueness of free product}
			Let $J$ be any index set.  Let $(\cA_j)_{j \in J}$ and $\cB$ be tracial $\mathrm{W}^*$-algebras, and let $\varphi_j: \cB \to \cA_j$ be a tracial $W^*$--embedding.  Then there exists a tracial $\mathrm{W}^*$-algebra $\cC$ and tracial $W^*$--embeddings $\iota_j: \cA_j \to \cC$ such that
			\begin{enumerate}
				\item $\varphi = \iota_j \circ \varphi_j: \cB \to \cC$ is independent of $j$.
				\item The images $(\iota_j(\cA_j))_{j \in J}$ are freely independent with amalgamation over $\varphi(\cB)$.
				\item $\cC$ is generated by $(\iota_j(\cA_j))_{j \in J}$.
			\end{enumerate}
			Moreover, if $\widetilde{\cC}$ and $\widetilde{\iota}_j$ are another tracial $\mathrm{W}^*$-algebra and  tracial $W^*$--embeddings satisfying these properties, then there exists a unique isomorphism $\Phi: \cC \to \widetilde{\cC}$ such that $\Phi \circ \iota_j = \widetilde{\iota}_j$ for $j \in J$.
		\end{lemma}
		
		In particular, this implies the following observation, which we will use throughout to switch perspectives between free independence and free products.
		
		\begin{corollary} \label{cor: independence and embeddings}
			Let $J$ be any index set. Let $\cA$ be a tracial $\mathrm{W}^*$-algebra and $(\cA_j)_{j \in J}$ be a collection of subalgebras containing a common subalgebra $\cB$.  Let $\median_{\cB} (\cA_j)_{j \in J}$ be the free product with amalgamation over $\cB$ and let $\iota_j$ be the canonical inclusion of $\cA_j$ into this free product. Then the following are equivalent:
			\begin{enumerate}
				\item $(\cA_j)_{j \in J}$ are freely independent in $\cA$ with amalgamation over $\cB$.
				\item There exists a tracial $\mathrm{W}^*$-isomorphism $\phi: \median_{\cB} (\cA_j)_{j \in J} \to \bigvee_{j \in J} \cA_j \subseteq \cA$ such that the following diagram commutes for each $j\in J$:
				\[
				\begin{tikzcd}
					\mathcal{A}_j \arrow{d}{\iota_j} \arrow{dr} & \\
					\median_{\mathcal{B}}(\mathcal{A}_j)_{j \in J} \arrow{r}{\phi} & \bigvee_{j \in J} \mathcal{A}_j,
				\end{tikzcd}
				\]
				where the diagonal map is the inclusion.
			\end{enumerate}
		\end{corollary}
		
		\begin{proof}
			(1) $\implies$ (2) because of the uniqueness of the free product up to canonical isomorphism (Lemma \ref{lem: uniqueness of free product}).
			
			For (2) $\implies$ (1), first note that conditional expectations commute with isomorphisms, that is, if $\mathcal{P} \subseteq \mathcal{M}_1$ is an inclusion of tracial $\mathrm{W}^*$-algebras and $\phi: \mathcal{M}_1 \to \mathcal{M}_2$ is an isomorphism, then {{$\phi\circ E_{\mathcal{P}} = E_{\phi(\mathcal{P})}\circ \phi$.}}  In particular, {{$\phi$}} preserves the conditional expectation onto $\cB$. Since the vanishing moment conditions for free independence of $(\iota_j(\cA_j))_{j \in J}$ over $\cB$ hold in $\median_{\mathcal{B}}(\mathcal{A}_j)_{j \in J}$, it follows that they hold for $(\cA_j)_{j \in J}$ in $\bigvee_{j \in J} \cA_j \subseteq \cA$ as well.
		\end{proof}
		
		The free product construction allows us, for instance, to build a larger non-commutative probability space containing any two given non-commutative probability spaces, and which agree on a common subspace.  More precisely, if $\cA_1$ and $\cA_2$ are non-commutative probability spaces containing $\cB$, then there is some non-commutative probability space $\cA$ containing $\cB$, which also contains both $\cA_1$ and $\cA_2$ (namely, $\cA = \cA_1 *_{\cB} \cA_2$).  Further properties of amalgamated free products that we will use in our arguments are given in Appendix \ref{sec: AFP}.

			\subsection{Non-commutative Laws and the weak* topology}\label{subsecNon-commutativeLaws}
			
			
			In this section, we describe the analog of laws or probability distributions for non-commutative random variables from a tracial von Neumann algebra. Following the terminology in \cite{gangbo2022duality}, we denote by $\bW$ a set of representatives of the isomorphism classes of tracial $\mathrm{W}^*$-algebras with separable predual, so that each tracial von Neumann algebras with separable predual is isomorphic to a unique element of $\bW$. 
			
			We denote by  ${\rm NCP}_d:=\bC\langle x_1, \cdots, x_d \rangle$, the universal unital algebra generated by variables $x_1 , \cdots , x_d$.  Note that ${\rm NCP}_d$ can be equipped with a unique $*$-operation (i.e., an antilinear involution satisfying $(xy)^* = y^*x^*$) such that $x_j=x_j^*$, which makes it into a $*$-algebra.
			We then define $\Sigma_{d,R}$ to be the linear functionals $\lambda: {\rm NCP}_d\rightarrow \bC$ that satisfy 
			\begin{align*}
				\lambda(1)=1,\quad \lambda(pp^*)\geq 0,\quad  \lambda(pq)=\lambda(qp) \qquad \forall p,q\in {\rm NCP}_d
			\end{align*}
			which are $R$-exponentially bounded, meaning that for every $k \in \mathbb{N}$ and any monomial $\phi$ of degree $k$,
			$$
			|\lambda(\phi)| \leq R^k.
			$$
			For $\cA\in \bW$, there is a natural map $\lambda:\{X\in L^\infty(\cA)_{sa}:\|X\|_\infty\leq R\}\rightarrow \Sigma_{d, R}$ given by
			$$
			\lambda_X(p) := \tau\big(p(X)\big) \quad \hbox{ for }p\in {\rm NCP}_d.
			$$
			It is shown in \cite{gangbo2022duality} that given $\lambda \in \Sigma_{d,R}$, one can construct $\cA\in \bW$ and find $X\in L^\infty(\cA)_{sa}^d$ such that ${\lambda}_X = \lambda$. The Wasserstein metric is defined for $\lambda_1,\lambda_2\in \Sigma_{d,R}$ as
			$$
			d_W^2(\lambda_1,\lambda_2) := \inf \Big\{\|X_1-X_2\|_{L^2(\cA)}^2: \cA\in \bW,\ X_1,X_2\in L^\infty(\cA)_{sa}^d,\ {\lambda}_{X_1}=\lambda_1,\  {\lambda}_{X_2}=\lambda_2\Big\}.
			$$
			We let $\Sigma_d^\infty$ denote the union of $\Sigma_{d,R}$ over all $R>0$ with the natural equivalence.
			
			Since we want to work with random variables in the $L^2$ space of the von Neumann algebra, we define ${\Sigma}_d^2$ to be the closure of $\Sigma_d^\infty$ in the Wasserstein metric.  We will show that laws in $\Sigma_d^2$ can be represented by elements of $L^2(\cA)_{sa}^d$ for some $\cA\in \bW$, by proving an equivalent construction of $\Sigma_d^2$.
			
			Alternatively, $\widetilde{\Sigma}_d^2$ is defined as the equivalence classes of operators in $\sqcup_{\cA\in \bW} L^2(\cA)_{sa}^d$, with the relation that $X\in L^2(\cA)_{sa}^d$ is equivalent to $Y\in L^2(\cB)_{sa}^d$ if there exists a coupling that consists of $\cC\in \bW$ and tracial $W^*$-embeddings $\iota_1:\cA\rightarrow \cC$ and $\iota_2:\cB\rightarrow \cC$ such that $\iota_1(X) = \iota_2(Y)$.  Given $X\in L^2(\cA)_{sa}^d$, we denote its law by $\widetilde{\lambda}_X\in \widetilde{\Sigma}_d^2$.  This space is naturally equipped with a quotient metric, which is equivalent to the non-commutative Wasserstein metric defined as
			\begin{align} \label{wass}
				\widetilde{d}_W^2(\lambda,\mu) := \inf_{\cA\in \bW}\big\{\|X-Y\|_{L^2(\cA)}^2: X,Y\in L^2(\cA)_{sa}^d \text{ such that } \widetilde{\lambda}_X=\lambda \text{ and } \widetilde{\lambda}_Y=\mu\big\}.
			\end{align}
			
			
			\begin{lemma}
				The definition of $\widetilde{\Sigma}_d^2$ as a quotient space with respect to the couplings of von Neumann algebra is equivalent to $\Sigma_d^2$ as the closure of $\Sigma_d^\infty$ under the Wasserstein metric.
			\end{lemma}
			\begin{proof}
				Given $\widetilde{\lambda} \in \widetilde{\Sigma}_d^2$, we may find $\cA \in \bW$ and $X \in L^2(\cA)_{sa}^d$ such that $\widetilde{\lambda}_X = \widetilde{\lambda}$. By the density of $L^\infty(\cA)_{sa}^d$ in $L^2(\cA)_{sa}^d$, we take a sequence $\{X_i\}_{i  \ge 1}$ in $ L^\infty(\cA)_{sa}^d$ that converges to $X$ in $L^2(\cA)_{sa}^d$. Clearly, $\widetilde{\lambda}_{X_i}$ converges in the Wasserstein metric, showing that $\widetilde{\Sigma}_d^2 \subset\Sigma_d^2$.

				For the other direction, we consider a sequence with $\lambda_i \in \Sigma_d^\infty$ and ${d}_W^2(\lambda_i,\lambda_{i+1})\leq 2^{-i}$ for $i\in \{1,2,\ldots\}$. We construct by induction a sequence of von Neumann algebra, $\{\cA_i\}_{i \ge 1}$, representing the non-commutative law $\lambda_i$ by $X_i\in L^\infty(\cA_i)_{sa}^d$ with tracial $W^*$-embeddings $\iota_i:\cA_i\rightarrow \cA_{i+1}$  so that
				$$
				\|\iota_i(X_i)-X_{i+1}\|_{L^2(\cA_{i+1})}^2 = {d}^2_W(\lambda_i,\lambda_{i+1}).
				$$
				From the definition of the Wasserstein metric, there is a von Neumann algebra $\cB$ and $\widetilde{X},\widetilde{Y}\in L^\infty(\cB)_{sa}^d$ such that $\lambda_{\widetilde{X}} =\lambda_i$, $\lambda_{\widetilde{Y}} = \lambda_{i+1}$, and
				$$
				\widetilde{d}^2_W(\lambda_i,\lambda_{i+1}) = \|\widetilde{X}-\widetilde{Y}\|_{L^2(\cB)}^2.
				$$
				We may define a tracial $W^*$-embedding $\widetilde{\iota}:W^*(\widetilde{X}) \rightarrow \cA_i$ by identifying $\iota(\widetilde{X}) = X$. We then define $\cA_{i+1}$ as the free product of $\cA_i$ and $\cB$ with almagamation over $W^*(\widetilde{X})$. We let $\iota_i$ be the tracial $W^*$-embedding of $\cA_i$ into $\cA_{i+1}$, which satisfies the inductive hypothesis.
				
				We next define $\cA_\infty$ as the inductive limit of the $\cA_i$, and one can consider the sequence $\{\widetilde{X}_i\}_{ i \ge 1}$ taking values in $ L^\infty(\cA_\infty)_{sa}^d$.  It follows that
				$$
				\|\widetilde{X}_i - \widetilde{X}_{i+1}\|_{L^2(\cA_{\infty})}^2 = \|\iota(X_i)-X_{i+1}\|_{L^2(\cA_{i+1})}^2 \leq 2^{-i}.
				$$
				Therefore, we have a Cauchy sequence and a limit point $X\in L^2(\cA_\infty)_{sa}^d$, which defines $\widetilde{\lambda}_X\in \widetilde{\Sigma}_d^2$ and shows $\Sigma_d^2\subset \widetilde{\Sigma}_d^2$.
			\end{proof}

			Here is a list of some properties for non-commutative laws:
			\begin{itemize}

				
				\item It was shown in \cite[\S 5.5]{gangbo2022duality} that $\Sigma_{d,R}$ is not separable with the Wasserstein metric (and therefore neither is $\Sigma_d^2$).
				\item The spaces $\Sigma_{d,R}$ (which can be viewed as subsets of $\Sigma_d^2$ with the operator norm bounded by $R$) are compact in the weak* topology of convergence for the evaluation of every polynomial in ${\rm NCP}_d$.
				
				
				\item The continuous functions $C(\Sigma_d^2)$ can be identified with tracial $W^*$--functions $(f_\cA)_{\cA\in \bW}$ such that $f_{\cA}\in C(L^2(\cA)_{sa}^d)$ for all $\cA\in \bW$, where we recall the following definition of tracial function from \cite[Definition 3.5]{gangbo2022duality}. 
			\end{itemize}


			
			\begin{definition}[Definition 3.5 in \cite{gangbo2022duality}] A tracial $W^*$--function on $\Sigma_d^2$ with values in $(-\infty, +\infty]$ is a collection of functions $f_\cA: L^2(\cA)^d_{sa} \to (-\infty, +\infty]$ for $\cA\in \bW$, such that whenever $\iota: \cA \to \cB$ is a tracial $W^*$--embedding, $f_\cA=f_\cB \circ \iota$ (here $\iota$ is extended to a map $L^2(\cA)^d_{sa} \to L^2(\cB)^d_{sa}$).
			\end{definition}

			
			

			\section{Control problem with free individual noise and common noise}\label{sec:problem}
			
			\subsection{Example Problems} \label{subsec: example problems}
			
			Before introducing the general problems we will consider, we introduce a handful of the motivating problems. These examples will be further worked out in Section \ref{sec:examples}.

			\subsubsection{Quadratic cost} \label{subsec:quadratic} 
			
			The simplest type of drift is when the control space corresponds directly with the drift of the process, there is a single common noise, and free individual noise for each element.  The dynamics may be expressed in $L^2(\cA)_{sa}^d$, for a von Neuman algebra $\cA\in \bW$, as
			\begin{align}\label{eqn:quadratic_dynamics}
				dX^j_t = \alpha_t^j\, dt + \beta_C\, \mathbbm{1}_{\cA}\, dW_t^0 + \beta_F\, dS_t^j, \quad j\in\{1,\ldots,d\},
			\end{align}
			where $(W_t^0)_{t\in [0,T]}$ is a standard Brownian motion that affects all elements proportional to the identity element, and $(S_t^j)_{t\in [0,T]}^{j\in \{1,\ldots,d\}}$ are freely independent semi-circular processes.  These processes represent a source of non-commutative noise, for example, arising from the Gaussian unitary ensemble of random matrices.
			
			In Section \ref{sec:examples}, we solve exactly the problem of minimizing a quadratic terminal cost involving the second moments along with first moments squared, 
			$$
			g_{\cA}(X) = \sum_{i=1}^d\sum_{j=1}^d g_{ij}^0\tau(X^i\, X^j)+ \sum_{i=1}^d\sum_{j=1}^d g_{ij}^1\tau(X^i)\tau(X^j).
			$$
			That is, given $t_0 \in [0,T]$ and a non-commutative law $\lambda_0$, setting $\mathbb{A}^{t_0,T}_{\cA}$ to be the set of admissible control policies for $\cA\in \bW$ (defined in Section \ref{sec 2.2.2} later precisely), we solve the variational problem
			$$
			\overline{V}(t_0,\lambda_0) := \inf_{\cA\in \bW} \inf_{\widetilde{\alpha}\in \mathbb{A}_{\cA}^{t_0,T}}\Big\{ \mathbb{E}\big[\int_{t_0}^T \frac{1}{2}\|\alpha_t\|^2_{L^2(\cA)}dt + g_{\cA}(X_T)\big]: \lambda_{X_{t_0}}=\lambda_0, X \hbox{ solves (\ref{eqn:quadratic_dynamics}) on $[t_0,T]$}\Big\}.
			$$

			One can generalize this example in several ways. For instance,  one can add a running cost $\phi_{\cA}(X_t)$ to the integral in the cost for some  tracial $W^*$--function $ (\phi_{\cA})_{\cA\in \bW}$, and   $(g_{\cA})_{\cA\in \bW}$  can be any  tracial $W^*$--function.  In this case, 
			the  corresponding Hamiltonian can be expressed as a tracial $W^*$--function on $\Sigma_{2d}^2$, $(H_{\cA})_{\cA\in \bW}$, by 
			$$
			H_{\cA}(X,P) = \frac{1}{2}\|P\|_{L^2(\cA)}^2 - \phi_{\cA}(X).
			$$
			The value function $\overline{V}$ defined above solves the Hamilton-Jacobi equation
			\begin{align*}
				-\partial_t \overline{V}(t,\lambda) + \frac{1}{2}\|\partial \overline{V}(t,\lambda)\|^2_\lambda -\frac{\beta_C^2}{2}\, \Delta \overline{V}(t,\lambda) -\frac{\beta_F^2}{2}\, \Theta \overline{V}(t,\lambda) =&\ \phi(\lambda),\\
				\overline{V}(T,\lambda) =&\ g(\lambda),
			\end{align*}
			in a manner to be described later.
			
			{This example arises in the theory of large deviation for random matrices, which we will explore in more detail in the subsequent paper.}

			\subsubsection{Eikonal controls}\label{subsec:eikonal}  
			
			For a related example, we suppose that the controls are constrained so that $\|\alpha\|_{L^2(\cA)}\leq 1$. This results in a similar equation with the $1$-homogeneous Hamiltonian $H_{\cA}(P) = \|P\|_{L^2(\cA)}$ rather than the quadratic one.  Alternatively, instead of restricting the $L^2$ norm of the control, it is possible to restrict the operator norm $\|\alpha\|_{\infty}\leq 1$, resulting in an $L^1$ norm, i.e.,  $H_{\cA}(P) = \tau(|P|)$ where $|P| = P_+ + P_-$ given the decomposition into positive and negative definite parts $P= P_+ - P_-$. These examples have nonsmooth solutions, as illustrated in Section \ref{sec:examples}.
			
			
			
			\subsubsection{Controlled von Neumann equation}\label{subsec:con-schrodinger}
			
			We can consider the dynamics driven by the commutator of the state $X_t$ and the control $\alpha_t$:
			\begin{align}\label{eqn:schrodinger_dynamics}
				dX^j_t = {\rm i}[X^j_t,\alpha_t^j]\, dt,  \quad j\in\{1,\ldots,d\}.
			\end{align}
			This equation is motivated by the von Neumann equation in quantum mechanics where a system evolves through unitary conjugation.  Indeed, a straightforward computation shows that
			\[
			X_t^j = U_t^j X_0^j (U_t^j)^*
			\]
			where $U_t^j$ is the unitary given by
			\[
			dU_t^j = {\rm i} \alpha_t^j U_t^j\,dt.
			\]
			In Section \ref{sec:schrodinger}, we show that when the value function
			\begin{align*}
				\overline{V}(t_0, \lambda_0):= &\ \inf_{\cA\in \bW}\inf_{\widetilde{\alpha}\in \mathbb{A}_\cA^{t_0,T}}\Big\{\int_{t_0}^T \frac{1}{2}\|\alpha_t\|^2_{L^2(\cA)}dt + g_{\cA}(X_T): \lambda_{X_{t_0}}=\lambda_0, (X_t)_t \hbox{ solves (\ref{eqn:schrodinger_dynamics}) on $[t_0,T]$}\Big\} 
			\end{align*}
			admits a minimizer, we have a \emph{non-commutative Hopf-Lax formula} 
			\begin{multline*}
				\overline{V}(t_0, \lambda_0) \\
				= \inf_{\cA\in \bW}\inf_{\alpha\in L^2(\cA)^d_{sa}}\Big\{\frac{(T-t_0)}{2}\|\alpha\|_{L^2(\cA)}^2 + g_{\cA}\big(e^{-{\rm i}\alpha (T-t_0)}\, X\, e^{{\rm i}\alpha (T-t_0)}\big): X\in L^2(\cA)^d_{sa}, \lambda_X=\lambda_0\Big\}.
			\end{multline*}
			
			Another motivation for this setup is Voiculescu's liberation theory \cite{voiculescu1999analogues6}, a version of information theory where additive perturbations are replaced by perturbations through unitary conjugation, so for instance the analogue of semi-circular Brownian motion to an initial condition would be conjugating the initial condition by a tuple of free unitary Brownian motions as in \cite{biane1997freebrownian}.  Similarly, the non-commutative Hopf-Lax formula above is the liberation analogue of the standard Hopf-Lax formula based on additive perturbations.  More generally, one could consider both a control term and stochastic term in the von Neumann equation, setting
			\[
			dU_t^j = {\rm i} (dS_t^j + \alpha_t^j \,dt) U_t^j
			\]
			and
			\[
			dX_t^j = {\rm i} [dS_t^j + \alpha_t^j\,dt, X_t^j]
			\]
			for a free Brownian motion $S_t^j$, although we will focus in this paper on the deterministic case.
			
			
			This example does not quite satisfy our general assumptions in \S \ref{sec: setup} since we want the Lagrangian to extend to $X$ and $\alpha$ in $L^2$ and be a globally Lipschitz function plus a quadratic function of $\alpha$.  However, it is an important example of a drift which is an $E$-linear tracial vector-field as defined in \S \ref{sec 2.2.3}(b) below.  In other words, whenever there is a tracial $W^*$--embedding $\iota:\cA\rightarrow \mathcal{B}$, 
			$$
			[\iota\, X, \iota\, \alpha] =  \iota\, [X, \alpha] \hbox{ and } E [\iota\, X, \beta] = [X,E\, \beta] \quad \text{ for any } X,\alpha\in L^2(\cA)_{sa}^d \text{ and } \beta\in L^2(\cB)_{sa}^d,
			$$
			where $E: \cB\rightarrow \mathcal{A}$ denotes the conditional expectation adjoint to  $\iota:\cA\rightarrow \mathcal{B}$.

		\subsection{Setup and Assumptions for the General Problem} \label{sec: setup}
		
		Given a tracial $\mathrm{W}^*$-algebra $\cA$, the ``state space'' in mean field games terminology corresponds will be the space of the self-adjoint $d$-tuples in the space $L^2(\cA)_{sa}^d$.  Throughout this section, we assume that $T>0$ is a given terminal time.  We now describe the terminology and setup for the Brownian motions, control policies, drift function, and cost functions for the general non-commutative stochastic optimization problem \eqref{eq:d4-c}.  Throughout the paper, \textbf{Assumption A} will refer to the conditions on the control policies, drift, and cost functions in \S \ref{sec 2.2.2}, \S \ref{sec 2.2.3}, and \S \ref{sec:cost_functions}.  Another more restrictive set of hypotheses called \textbf{Assumption B} will be given in \S \ref{sec: Assumption B}.
		
		\subsubsection{Classical and free Brownian motions}
		We assume to be given a non--atomic complete filtered probability space $\big(\Omega,\mathcal{F},(\cF_t)_{0\leq t \leq T}, \bP\big)$ which supports the Brownian motion $(W_t^0)_{t\in [0,T]}\in C([0,T]; L^2(\Omega, \mathcal{F},\bP))$: 
		\begin{itemize}
			\item[(a)]  $W_0^0=0$.
			\item[(b)]  For $0\leq s\leq t\leq T$, $W_t^0 - W_s^0$ is normally distributed with mean 0 and variance $t-s$.
			\item[(c)]  For any sequence of times  $0=t_0\le t_1 \le t_2 \le \cdots \le t_{k-1}  \le t_k= T,$ the collection of increments  $W_{t_{j+1}}^0-W_{t_j}^0$ for $j\in \{0,1,\ldots, k-1\}$ are (mutually) independent.
			
		\end{itemize}
		We assume that  $\mathcal F= \sigma (W^0_s : 0\le s\le T)$ and  $\mathcal F_t= \sigma (W^0_t : 0\le s\le t)$ for $t\in [0,T]$.  

		Throughout the paper, we write products between the Brownian motions and the von Neumann algebra, corresponding to the canonical embedding of scalars by $\mathbf{1}_{\cA}\, W_t^0\in L^\infty(\cA)_{sa}$ or $\mathbbm{1}_{\cA}\, W_t^0\in L^\infty(\cA)_{sa}^d$. We emphasize that the Wiener probability space may be fixed throughout, whereas it is not possible to fix the non-commutative probability space.

		~

		One can similarly define a free semi-circular process. Analogous to the notion of filtration of $\sigma$-algebra, we call an increasing collection of sub-von Neumann algebra a free filtration. Let $0\le t_0\le t_1\le T.$ For a given $\cA\in \bW$ and a free filtration $(\cA_t)_{t\in [t_0,t_1]}$,  we say that a $d$-dimensional process $(S_t)_{t\in [t_0,t_1]}\in C([t_0,t_1];L^\infty(\cA)_{sa}^{d})$ is a \emph{free semi-circular process (or free Brownian motion) compatible with the free filtration $(\cA_t)_{t\in [t_0,t_1]}$} on the interval $[t_0,t_1]$, if it satisfies  the following properties:
		
		\begin{itemize}
			\item[(a)] $S_{t_0} = 0$.
			
			\item[(b)] For $t_0\leq s\leq t\leq t_1$ and $l\in \{1,\ldots, d\}$, the increment $S_t^l - S_s^l$ is semi-circularly distributed with mean 0 and variance $t-s$, and the components $\{S_t^l - S_s^l\}_{l=1}^d$ are freely independent.
			
			\item[(c)] $S_t\in L^\infty(\cA_t)_{sa}^{d}$ for all $t\in [t_0,t_1]$.
			\item[(d)] For $t_0 \leq s\leq t\leq t_1$,   $S_t - S_s$ is freely independent of $\cA_s$.
		\end{itemize}
		\hfill\break
		Note that (c) and (d) together imply that 
		for  $t_0=s_0\le s_1 \le \cdots \le s_k= t_1,$
		the collection of increments $S_{s_{j+1}}^l-S_{s_j}^l$ for $j\in \{0,1,\ldots, k-1\}$ and  $l\in \{1,\ldots, d\}$ are freely independent (but the converse is not true since $(\mathcal{A}_t)_t$ could be larger than algebra generated by the initial conditions and semi-circulars).
		
		We list now the general assumptions on the problem that we will refer to as Assumption \textbf{A}.
		\subsubsection{Sets of the controls and admissible control policies} \label{sec 2.2.2}
		We assume that controls in $\cA$  belong to some subset $\bA_\cA$ of $L^2(\cA)_{sa}^d$ which satisfies
		\begin{itemize}
			\item[(a)] $\bA_\cA\subset L^2(\cA)_{sa}^d$ is closed and convex.
			\item[(b)]  $0 \in \bA_\cA$.
			\item[(c)]  For any $\cB\in \bW$ and a tracial $W^*$--embedding $\iota:\cA \to \cB$ (with its adjoint $E$), we have 
			$$\iota\, \bA_\cA \subset \bA_\cB \quad {\rm and} \quad E\, \bA_\cB \subset \bA_\cA.$$
		\end{itemize}
		Given $\cA\in \bW$, $[t_0, t_1] \subset [0, T]$, and $x_0\in L^2(\cA)_{sa}^d$, we let  $\bA_{\cA,x_0}^{t_0,t_1}$ be the collection of \emph{admissible} control policies 
		$$
		\widetilde{\alpha} = \Big((\alpha_t)_{t\in [t_0,t_1]},(\cA_t)_{t\in [t_0,t_1]}, (S_t)_{t\in [t_0,t_1]}\Big)
		$$
		that satisfy the following properties:
		\begin{itemize}
			\item[(a)]  $(\cA_t)_{t\in [t_0,t_1]}$ is a free filtration, i.e., an increasing sequence of tracial $W^*$ subalgebras of $\cA$, such that $x_0\in L^2(\cA_{t_0})_{sa}^d$.
			\item[(b)]   $(S_t)_{[t_0,t_1]}$ is a $d$-dimensional free semi-circular process compatible with $(\cA_t)_{t\in [t_0,t_1]}$.
			\item[(c)]   For every $s\in [t_0,t_1]$, we have $(\alpha_t)_{t\in [t_0,s]}\in L^2\big([t_0,s]\times (\Omega, \mathcal{F}_s,\mathbb{P}); \bA_{\cA_s}\big)$,  i.e., $(\alpha_t)_{t\in [t_0,t_1]}$ is progressively measurable and freely progressive.  
		\end{itemize}
		
		Note that progressive measurability is in general a stronger condition than adaptedness. However if each path is continuous, then  adaptedness does imply progressive measurability. 
		
		\begin{remark}We may have $\bA_{\cA, x_0}^{t_0,t_1}=\emptyset$ even when  $\bA_{\cA} \not=\emptyset$, since not every tracial $W^*$--algebra $\cA$ contains a free semi-circular process.
		\end{remark}

		\subsubsection{Drift function} \label{sec 2.2.3}
		We consider drift functions $b_\cA: L^2(\cA)^d_{sa}\times \bA_\cA\rightarrow L^2(\cA)^d_{sa}$ that associate a `tangent vector' for every state and control. We make the following assumptions.
		\begin{itemize}
			\item[(a)] $(b_{\cA})_{\cA  \in \bW}$ defines a tracial vector-field in the sense that for any $\cA,\cB\in \bW$ with a tracial $W^*$--embedding $\iota:\cA\rightarrow \mathcal{B}$, $X\in L^2(\cA)^d_{sa}$ and  $\alpha \in \bA_{\cA}\subset L^2(\cA)^d_{sa}$, 
			\begin{align}\label{eqn:tracial_vector_field}
				\iota\, b_\cA(X, \alpha) =&\  b_\mathcal{B}(\iota\, X,{\iota}\, \alpha).
			\end{align}
			\item[(b)]  $(b_{\cA})_{\cA\in \bW}$ is \emph{$E$-affine} in the sense that for any $\cA\in \bW$ and $X\in L^2(\cA)^d_{sa}$, $\alpha\mapsto b_{\cA}(X,\alpha)$  is affine, also for any tracial $W^*$--embedding  $\iota:\cA\rightarrow \cB$ with its adjoint $E:L^2(\cB)_{sa}^d\rightarrow L^2(\cA)_{sa}^d$,  
			\begin{align}\label{eqn:E_linear}
				b_{\cA}(X, {E}\, \alpha) = E\, b_\mathcal{B}(\iota\, X,\alpha) \hbox{ for all }  \alpha\in\bA_\mathcal{B}.
			\end{align}
			\item[(c)]  $(b_{\cA})_{\cA\in \bW}$ is uniformly continuous on bounded sets, i.e. for any $M>0$ there is a modulus of continuity $\omega_M$ such that for all $\cA\in \bW$, $X,Y\in L^2(\cA)$ and  $\alpha,\beta\in \bA_\cA$ with $\max\{\|X\|_{L^2(\cA)}, \|Y\|_{L^2(\cA)}, \|\alpha\|_{L^2(\cA)}, \|\beta\|_{L^2(\cA)}\}\leq M$,  
			\begin{align*}
				\|b_\cA(X,\alpha) - b_\cA(Y,\beta)\|_{L^2(\cA)}^2 \leq&\ \omega_M\big(\|X-Y\|_{L^2(\cA)}^2 + \|\alpha-\beta\|_{L^2(\cA)}^2\big).
			\end{align*}
			\item[(d)] There is a constant $\bar{C}>0$ such that for all $\cA\in \bW$, $X_1,X_2\in L^2(\cA)$ and  $\alpha\in \bA_\cA$,
			\begin{align}\label{eqn:C_monotonicity}
				\| b_{\cA}(X_1,\alpha)-b_{\cA}(X_2,\alpha)\|_{L^2(\cA)} \leq \bar{C}\|X_1-X_2\|_{L^2(\cA)}.
			\end{align}
		\end{itemize}
		
		\begin{remark}
			These assumptions on the drift, especially (d), imply that for any $M>0$, there exists a constant $\upsilon_M>0$ such that for all $X\in L^2(\cA)_{sa}^d$ with $\|X\|_{L^2(\cA)}\leq M$ and  $\alpha,\alpha'\in \mathbb{A}_{\cA}$,
			\begin{align}\label{eqn:b_continuity}
				\|b_{\cA}(X,\alpha)-b_{\cA}(X,\alpha')\|_{L^2(\cA)}\leq \upsilon_M\|\alpha-\alpha'\|_{L^2(\cA)} \hbox{ and } \|b_{\cA}(X,\alpha)\|_{L^2(\cA)}\leq \upsilon_M\big(1+\|\alpha\|_{L^2(\cA)}\big).
			\end{align}
		\end{remark}
		
		\subsubsection{Cost functions}\label{sec:cost_functions}
		We consider a running cost $L_{\cA}:L^2(\cA)_{sa}^d\times \mathbb{A}_{\cA}\rightarrow \mathbb{R}$ and a terminal cost $g_{\cA}:L^2(\cA)_{sa}^d \rightarrow   \mathbb{R}$ that satisfy the following assumptions.
		\begin{itemize}
			\item[(a)] Both $(L_{\cA})_{\cA\in \bW}$ and $(g_{\cA})_{\cA\in \bW}$ are tracial $W^*$--functions. Equivalently, $(L_{\cA})_{\cA\in \bW}$ may be defined on the space of joint non-commutative laws in $\Sigma_{2d}^2$ and $(g_{\cA})_{\cA\in \bW}$ is a function on the space of non-commutative laws in $\Sigma_d^2$.
			\item[(b)]   $(L_\cA)_{\cA\in \bW}$ is \emph{$E$-convex} in the control variable, meaning that for any $\cA\in \bW$ and $X\in L^2(\cA)^d_{sa}$, $\alpha \mapsto L_\cA(X,\alpha)$ is convex, and   for any tracial $W^*$--embedding  $\iota:\cA\rightarrow \mathcal{B}$ with its adjoint $E:L^2(\mathcal{B})_{sa}^d\rightarrow L^2(\mathcal{A})_{sa}^d$, 
			\begin{align}\label{eqn:E_convex}
				L_\cA( X, {E}\, \alpha) \leq  L_{\mathcal{B}}(\iota\, X, \alpha) \hbox{ for all }  \alpha\in\bA_\mathcal{B}.
			\end{align}
			\item[(c)] Similar to the drift, we assume that $(L_{\cA})_{\cA\in \bW}$ is uniformly continuous on bounded sets.
			\item[(d)]  There exists a constant $C_1>0$ such that for all $\cA\in \bW$, $X\in L^2(\cA)^d_{sa}$ and $\alpha \in \bA_{\cA}$,
			\begin{align}\label{eqn:lower_and_upper_bounds}
				-C_1 +\frac{1}{C_1}\|\alpha\|_{L^2(\cA)}^2\leq&\  L_{\cA}(X,\alpha)\leq C_1\big(1+ \|X\|_{L^2(\cA)} + \|\alpha\|_{L^2(\cA)}^2\big),\\
				-C_1\leq&\  g_{\cA}(X)\leq C_1\big(1+ \|X\|_{L^2(\cA)}\big).\nonumber
			\end{align}
			Also  $(L_{\cA})_{\cA\in \bW}$ and 
			$ (g_{\cA})_{\cA\in \bW}$ are Lipschitz with respect to $X$:  There exists a constant $C_2>0$ such that for all $\cA\in \bW$, $X_1,X_2\in L^2(\cA)^d_{sa}$ and $\alpha \in \bA_{\cA}$,
			\begin{align}\label{eqn:Lipschitz_bounds}
				|L_{\cA}(X_1,\alpha) -L_{\cA}(X_2,\alpha)| \leq&\ C_2\|X_1-X_2\|_{L^2(\cA)}, \\
				|g_{\cA}(X_1) -g_{\cA}(X_2)|\leq&\ C_2\|X_1-X_2\|_{L^2(\cA)}. \nonumber 
			\end{align}
		\end{itemize}

		
		
		\subsubsection{Additional assumptions (Assumption B)} \label{sec: Assumption B}
		
		As mentioned above the assumptions in Sections \ref{sec 2.2.2}, \ref{sec 2.2.3} and \ref{sec:cost_functions} are refered to \textbf{Assumption A}.  Sometimes to simplify the analysis, we also assume the following conditions, which we refer to as \textbf{Assumption B:}
		\begin{enumerate}[label=(\alph*)] \item The control set is $\mathbb{A}_{\cA} = L^2(\cA)_{sa}^d$ and the drift function has the form
			$$
			b_{\cA}(X,\alpha)=\alpha.
			$$ 
			\item The Lagrangian  $(L_{\cA})_{\cA\in \bW}$ is jointly $E$-convex in $(X,\alpha)$ and the terminal cost $(g_{\cA})_{\cA\in \bW}$ is $E$-convex. This means that for any $\cA\in \bW$, the maps $(X,\alpha) \mapsto L_\cA(X,\alpha)$ and $X\mapsto g_{\cA}(X)$ are convex, and  for any tracial $W^*$--embedding  $\iota:\cA\rightarrow \mathcal{B}$ with its adjoint $E:L^2(\mathcal{B})_{sa}^d\rightarrow L^2(\mathcal{A})_{sa}^d$, 
			\begin{align*}
				L_\cA( E\, X, {E}\, \alpha) \leq  L_{\mathcal{B}}(X, \alpha)& \hbox{ for all }  X\in L^2(\cB)_{sa}^d \hbox{ and }\alpha\in\bA_\mathcal{B},\\
				g_\cA( E\, X) \leq  g_{\mathcal{B}}(X)& \hbox{ for all }  X\in L^2(\cB)_{sa}^d.
			\end{align*}
			
		\end{enumerate}

		\subsubsection{Stochastic differential equations on the von Neumann algebra}
		
		Now we describe how the process $X_t$ in the variational problem \eqref{eq:d4-c} is constructed from the control as a solution to a free stochastic differential equation (SDE).  We specify diffusion coefficients $\beta_C\geq0$ and  $\beta_F\geq 0$.
		Let $\cA\in \bW$ and $[t_0, t_1] \subset [0, T]$. For $x_0\in L^2(\cA)^d_{sa}$ and $\widetilde{\alpha}\in \mathbb{A}_{\mathcal{A},x_0}^{t_0,t_1}$, we consider the SDE on  $ L^2(\cA)^d_{sa}$:  
		\begin{align}\label{eqn:common_noise}
			\begin{cases} dX_t = b_\cA(X_t,\alpha_t)\, dt +  \beta_C\, \mathbbm{1}_{\cA}\, dW_t^0 +  \beta_F\, dS_t,\\
				X_{t_0}=x_0.
			\end{cases}
		\end{align}
		We consider strong solutions such that $t\mapsto X_t$ is continuous in the $L^2$ norm, which are adapted to the common noise and freely adapted; this means more explicitly that $\omega\mapsto X_t(\omega)$ is $\mathcal{F}_t$ measurable and $X_t\in L^2(\cA_t)_{sa}^d$ $\mathbb{P}$-a.s.\ for each $t\in [t_0,t_1]$, and satisfies the integral equation, expressed here component-wise,
		$$
		X_t^j = x_0^j + \int_{t_0}^t b^j_\cA(X_s,\alpha_s)\, ds + \beta_C\, \mathbf{1}_{\cA}\, (W_s^0-W_{t_0}^0)+ \beta_F\, (S_s^j-S_{t_0}^j), \quad j=1, \cdots, d.
		$$
		{Here $(S_t)_{t\in [t_0,T]}$ denotes the free Brownian motion on $[t_0,T]$. Note that $S_{t_0}=0$.}
		
		We denote by $X_t[t_0,x_0,\widetilde{\alpha}]$ the solution of (\ref{eqn:common_noise}) on $[t_0,t_1]$, or $X_t[\widetilde{\alpha}]$ when $(t_0,x_0)$ are clear from the context. 
		%
		%
		The cost associated to a control policy $\widetilde{\alpha}\in \mathbb{A}^{t_0,T}_{\cA,x_0}$ is defined as
		$$
		\bE\bigg[\int_{t_0}^T L_{\cA}(X_t[\widetilde{\alpha}],\alpha_t)dt + g_{\cA}(X_T[\widetilde{\alpha}])\bigg].
		$$
		
		\begin{remark}\label{rem:commute}
			Observe that under the assumptions above, if $\cA, \cB\in \bW$ and $\iota: \cA \to \cB$ is a tracial $W^*$--embedding, then  $Y:=\iota\, X$ satisfies the analogue of \eqref{eqn:common_noise} where $(X, b_\cA, \alpha, S)$ is replaced by $(Y, b_\cB, \iota\, \alpha, \, \iota\, S)$, where the original filtration $\cA_t$ (embedded in $\cB$) is used as the filtration in $\cB$ as well.
		\end{remark}

		
		
		\subsubsection{The value function}
		
		As mentioned in the introduction, we want to allow the ambient algebra $\cB$ to vary in the optimization problems.  We will therefore define several versions of the value function.  The first $\widetilde{V}_{\cA}$ only looks at filtrations in a fixed algebra, the second $\overline{V}_{\cA}$ allows arbitrary extensions of the given algebra $\cA$, and the third $\overline{V}$ is defined on the space of non-commutative laws.
		
		For $\cA \in \bW$,  we consider the value function on $L^2(\cA)^d_{sa}$, defined as follows: For $t_0 \in [0,T]$ and  $x_0\in L^2(\cA)_{sa}^d$, 
		\begin{align}\label{eqn:l2_value}
			\widetilde{V}_{\cA}(t_0,x_0) := \inf_{\widetilde{\alpha} \in \bA_{\cA,x_0}^{t_0,T}}\Big\{\bE\Big[\int_{t_0}^T L_{\cA} (X_t[\widetilde{\alpha}], \alpha_t)dt + g_{\cA}(X_T[\widetilde{\alpha}])\Big]:  X_{t_0}[\tilde\alpha]=x_0 \Big\},
		\end{align}
		where  $X_t[\widetilde \alpha]$ denotes a solution to \eqref{eqn:common_noise} and $\bE$ denotes expectation with respect to
		the common noise. Note that if $\cA$ does not support a free Brownian motion $(S_t)_{t\in [t_0,T]}$ freely independent of $x_0$, i.e. if   $\bA_{\cA,x_0}^{t_0,T}$ is an empty set,  then this definition will result in $+\infty$.
		
		The eventual value function  is defined as follows: For $\cA \in \bW,$
		\begin{align} \label{def:bar}
			\overline{V}_\cA(t_0, x_0):= \inf_{\iota: \mathcal A \to \mathcal B} \widetilde{V}_{\mathcal B}(t_0, \iota\, x_0),
		\end{align}
		where the infimum is performed over the set of $(\cB, \iota)$ such that $\cB \in \bW$ and $\iota: \mathcal A \to \mathcal B$ is a tracial $W^*$--embedding.  {The function $\overline{V}_{\cA}$ will be our main object of study, while $\widetilde{V}_{\cA}$ is largely an intermediate step in making the definition and lacks many desirable properties such as continuous dependence on the initial condition (since for some choices of $\cA$, a small perturbation of $x_0$ can preclude the existence of a freely independent Brownian motion).}
		
		We will also view this value function as a function on the space of non-commutative laws, i.e.\ the non-commutative Wasserstein space.  For $t_0 \in [0,T]$ and  $\lambda \in \Sigma_{d}^2,$ take  any $\cA \in \bW$ and $x_0\in \cA$ such that $\lambda_{x_0}=\lambda$, and then define the value function 
		\begin{align}\label{eqn:value}
			\overline{V}(t_0, \lambda) := \overline{V}_\cA(t_0, x_0).
		\end{align}
		We will show that this is well-defined, i.e.\ independent of the particular $\cA$ and $x_0$ used to represent $\lambda$, in Lemma \ref{lem:nov18.2023.2} below.
			
			\subsubsection{The Hamiltonian}
			
			The Hamilton-Jacobi equation satisfied by the value function includes the \emph{Hamiltonian} $H$ which is the Fenchel-Legendre dual of the Lagrangian.  In the non-commutative setting, the Hamiltonian is defined as follows: For a tracial function $L_{\cA}:L^2(\cA)_{sa}^d\times \mathbb{A}_{\cA}\rightarrow \mathbb{R}$, let 
			\begin{align}\label{eqn:Hamiltonian_equivalence}
				H_\cA( X,P) := \sup_{\iota:\cA\rightarrow \cB\in \bW}\ \sup_{\alpha\in \bA_{\mathcal{B}}}\Big\{ \big\langle b_{\mathcal{B}}(\iota\, X,\alpha),\iota\, P\big\rangle_{L^2(\cB)} - L_{\cB}(\iota\, X,\alpha)\Big\},\qquad    X,P\in L^2(\cA)_{sa}^d.
			\end{align}
			We will show that the Hamiltonian is a tracial $\mathrm{W}^*$-function in Lemma \ref{lem:Hamiltonian-bis} below.
			
			\subsubsection{Generalizations} \label{sec: generalizations of setup}
			
			Before going on to prove properties of the value function that follow from Assumption A, we remark that the setup could be significantly generalized, although this would make the analysis correspondingly more complicated, and hence we leave it for future work.
			
			First, we could include non-constant coefficients in the diffusion terms.  For the free semi-circular process, this might correspond to tracial $W^*$-tensor field $\eta_{\cA}:L^2(\cA)_d^{sa}\rightarrow L^2(\cA)_d^{sa}\otimes L^2(\cA)_d^{sa}$.   We refer to \cite{biane1998stochastic,freesde1,freesde2} for a theory of free stochastic differential equations. Most of the theory we develop will go through, but we will remark at certain points where complications arise.
			
			Many important examples of controlled partial differential equations feature drifts that are unbounded. Similar, unbounded operators commonly arise in quantum mechanics. In these cases, the drift may be defined on a dense subset and continuous with respect to a stronger topology while  we impose that $X \to \bar C\, X-b_\cA(X, \alpha)$ is monotone, i.e.: 
			\begin{equation}
				\label{eq:june28.2024.1}
				\langle b_{\cA}(X_1,\alpha)-b_{\cA}(X_2,\alpha),X_1-X_2\rangle_{L^2(\cA)}\leq \bar{C}\|X_1-X_2\|^2_{L^2(\cA)}.
			\end{equation}
			Significant effort was made in the theory of infinite dimensional viscosity solutions to handle such cases.   Due to the complications when handling such unbounded terms, we do not include them in our current analysis. We note that the example of the controlled von Neumann equation $b_{\cA}(X,\alpha)={\rm i}[X,\alpha]$ does not satisfy the uniform continuity requirements but does satisfy \eqref{eq:june28.2024.1} with $\bar{C}=0$.

			
			\subsection{Properties of the Hamiltonian} \label{subsec: Hamiltonian properties}
			
			In this section, we establish basic properties of the Hamiltonian and show that is a tracial $\mathrm{W}^*$-function.  We first show that under only the assumption that $(b_{\cA})_{\cA\in \bW}$ is a tracial $W^*$ vector-field and $(\cL_{\cA})_{\cA\in \bW}$ is a tracial $W^*$--function, the Hamiltonian is a tracial $W^*$--function.  The additional assumptions will allow the Hamiltonian to be realized in a single von Neumann algebra, and we will adopt Assumption \textbf{A} for the remainder after this lemma, although some analogous results may hold under weaker assumptions.
			
			%
			%
				%
			
			\begin{lemma}\label{lem:Hamiltonian-bis} 
				If  $(L_{\cA})_{\cA\in \bW}$ is a tracial $W^*$--function and $(b_{\cA})_{\cA\in \bW}$ is a tracial $W^*$ vector-field, then $(H_{\cA})_{\cA\in \bW}$ is a tracial $W^*$--function.  
			\end{lemma}
			\begin{proof} Let $\cA,\cB\in \bW$ and $\iota:\cA\rightarrow \mathcal{B}$ be a tracial $W^*$--embedding. Note that if $\cC\in \bW$ and $\kappa:\cB\rightarrow \mathcal{C}$ is a tracial $W^*$--embedding then $\kappa \circ \iota:\cA\rightarrow \mathcal{C}$ is a tracial $W^*$--embedding. Thus for any  $X, P\in L^2(\cA)^d_{sa}$  and $\alpha \in \bA_{\cC}$, we have 
				\[
				H_\cA(X,P) \geq \Big\langle b_\cC\big(\kappa\circ\iota\, X,\alpha\big),\kappa  \circ\iota\, P\Big\rangle_{L^2(\cC)} - L_{\cC}\big(\kappa\circ\iota\, X, \alpha\big).
				\]
				Maximizing over the set of $(\cC, \kappa)$, we obtain that $H_\cA(X,P)\geq H_\cB( \iota\, X, \iota\, P).$
				
				To show the converse inequality, suppose that $\cB_2 \in \bW$ and $\iota_2: \cA \to \cB_2$ is a tracial $W^*$--embedding.  Let $\cC \in \bW$ be the free product of $\cB$ and $\cB_2$ with amalgamation over $\cA$.
				By  Lemma \ref{lem: uniqueness of free product}, there exist tracial $W^*$--embeddings $\phi_1: \cB \to \cC$ and $\phi_2: \cB_2 \to \cC$ such that $\phi_1 \circ \iota= \phi_2 \circ \iota_2$. Note that for any  $\alpha \in \bA_{\cB_2}$,
				\begin{align*}
					\Big\langle b_{\cB_2}\big( \iota_2\, X,\alpha\big),\iota_2\, P\Big\rangle_{\cB_2} - L_{\cB_2}\big(\iota_2\, X, \alpha\big)
					=&\ \Big\langle b_\cC\big(\phi_2\circ  \iota_2 \, X, \phi_2\, \alpha\big),\phi_2\circ  \iota_2 P\Big\rangle_{L^2(\cC)} - L_{\cC}\big(\phi_2\circ  \iota_2 X, \phi_2\, \alpha\big)\\
					=&\ \Big\langle b_\cC\big(\phi_1\circ  \iota \, X, \phi_2\, \alpha\big),\phi_1\circ  \iota P\Big\rangle_{L^2(\cC)} - L_{\cC}\big(\phi_1\circ  \iota X, \phi_2\, \alpha\big)\\
					\leq &\ H_\cB( \iota\, X, \iota\, P).
				\end{align*}
				Maximizing over the set of $(\cB_2, \iota_2)$, we conclude that  $H_\cA(X,P)\leq H_\cB( \iota\, X, \iota\, P).$ \end{proof}
			
			In the following lemma, we show that the Hamiltonian inherits both structure and estimates from our assumptions. The estimate we obtain in (\ref{eqn:Hamiltonian_bound}) below will be indispensable for the theory of viscosity solutions developed in \cite{ishii1993viscosity,lions1988viscosity}.
			\begin{lemma}\label{lem:Hamiltonian} Under Assumption \textbf{A}, the following hold.  
				\begin{enumerate}
					\item[(i)] The Hamiltonian satisfies 
					\begin{align}\label{eqn:Hamiltonian}
						H_\cA(X,P)= \sup_{\alpha \in \bA_{\cA}}\big\{ \langle b_\cA(X,\alpha),P\rangle_{L^2(\cA)} - L_{\cA}(X, \alpha)\big\}  \quad \hbox{ for } X,P\in L^2(\cA)_{sa}^d.
					\end{align}
					\item[(ii)] $(H_\cA)_{\cA\in \bW}$ is a tracial $W^*$--function (in both variables) which is $E$-convex in the second variable. 
					\item[(iii)] $H_{\cA}$ is uniformly continuous on bounded sets. Also for any $X,Y\in L^2(\cA)_{sa}^d$ and $\gamma>1$,
					\begin{align}\label{eqn:Hamiltonian_bound}
						H_{\cA}\big(X,\gamma(X-Y)\big)-H_{\cA}\big(Y,\gamma(X-Y)\big) \geq - \left(\frac{2\bar{C}+C_2^2}{2}\gamma\,\|X-Y\|_{L^2(\cA)}^2 + \frac{1}{2\gamma}\right).
					\end{align}
				\end{enumerate}
			\end{lemma}
			
			\begin{proof} (i)  We first note that $\geq$ in  \eqref{eqn:Hamiltonian} follows immediately from taking $\mathcal{B}=\cA$ in \eqref{eqn:Hamiltonian_equivalence}.  For the other direction, by assumptions \eqref{eqn:E_linear} and \eqref{eqn:E_convex},  for any $X\in L^2(\cA)^d_{sa}$, $\mathcal{B}\in \bW$, and $\alpha\in\bA_\mathcal{B}$, we have
				$$
				\langle b_\cA( X,E\, \alpha),P\rangle_{L^2(\cA)} - L_{\cA}(X,E\, \alpha)\geq \langle  b_{\mathcal{B}}(\iota\, X, \alpha),\iota\, P\rangle_{L^2(\cB)} - L_{\mathcal{B}}(\iota\, X, \alpha).
				$$
				Thus for every $\mathcal{B}\in \bW$ and $\alpha\in\bA_\mathcal{B}$, we may bound the argument of the supremum in \eqref{eqn:Hamiltonian_equivalence} using  $E\, \alpha\in \bA_\cA$,  proving $\leq$ in  \eqref{eqn:Hamiltonian}.
				
				(ii) By Lemma \ref{lem:Hamiltonian-bis}, $(H_\cA)_{\cA\in \bW}$ is a tracial function. 
				To prove the $E$-convexity of $(H_{\cA})_{\cA\in \bW}$, consider $\cA,\cB\in \bW$ with a tracial $W^*$--embedding $\iota:\cA\rightarrow \cB$. Then for any $X\in L^2(\cA)^d_{sa}$, $\alpha\in \bA_{\cA}$, and $P\in L^2(\mathcal{B})^d_{sa}$,
				$$
				\langle b_{\cA}(X,\alpha),E\, P\rangle_{L^2(\cA)} - L_{\cA}(X,\alpha)=\langle b_{\cB}(\iota\, X,\iota\, \alpha), P\rangle_{L^2(\cB)} - L_{\mathcal{B}}(\iota\, X, \iota\, \alpha) \leq H_{\mathcal{B}}(\iota\, X, P).
				$$
				Maximizing over $\alpha \in \bA_{\cA}$, we use \eqref{eqn:Hamiltonian} to conclude that $ H_{\mathcal{A}}( X, E\, P) \leq H_{\mathcal{B}}(\iota\, X, P).$

				(iii)    Uniform continuity on bounded sets follows from the uniform continuity on bounded sets of $b_{\cA}$ and $L_{\cA}$ along with the coercive lower bound on $L_{\cA}$ in \eqref{eqn:lower_and_upper_bounds}, which allows us to restrict to a bounded set for $\alpha\in \bA_{\cA}$. To provide more details, 
				fix some $M>0$ and $X,P\in L^2(\cA)_{sa}^d$ with $\max\{\|X\|_{L^2(\cA)},\|P\|_{L^2(\cA)}\}\leq M$. We first note that the Hamiltonian is bounded below by, using that $0\in \mathbb{A}_{\cA}$,
				\[
				H_{\cA}(X,P) \geq\ \langle b_{\cA}(X, 0),P\rangle_{L^2(\cA)} - L_{\cA}(X,0)\geq -M\, \|b_{\cA}(X,0)\|_{L^2(\cA)} - C_1(1+M),
				\]
				and by (\ref{eqn:b_continuity})
				$$
				\|b_{\cA}(X,0)\|_{L^2(\cA)} \leq \upsilon_M.
				$$
				Hence there is a constant $C_M>0$ depending only on $M$ such that 
				\begin{align}\label{eq:nov14.2023.1}
					H_{\cA}(X,P) \geq& -C_M.
				\end{align}
				
				In light of \eqref{eqn:Hamiltonian}, there exists $\bar{\alpha}\in \bA_\cA$ such that
				$$
				H_{\cA}(X,P) \leq1+ \langle b_{\cA}(X, \bar{\alpha}),P\rangle_{L^2(\cA)} - L_{\cA}(X,\bar{\alpha}).
				$$
				Then by conditions  \eqref{eqn:b_continuity} and  \eqref{eqn:lower_and_upper_bounds},
				\begin{align}\label{eq:nov14.2023.2}
					H_{\cA}(X,P) +\frac{1}{C_1}\|\bar{\alpha}\|_{L^2(\cA)}^2 \leq 1+ \upsilon_M(1+\|\bar{\alpha}\|_{L^2(\cA)})\, M+C_1 .
				\end{align}
				By \eqref{eq:nov14.2023.1} and \eqref{eq:nov14.2023.2},  for some constant $M'$ that depends only on $M$,
				\begin{align} \label{201}
					\|\bar{\alpha}\|_{L^2(\cA)}\leq M'.
				\end{align}
				We can go back to \eqref{eq:nov14.2023.2} to conclude that $H_{\cA}(X,P)\leq C_M$ (by increasing the value of $C_M$ if necessary).
				
				Now given $X_1,P_1,X_2,P_2\in L^2(\cA)_{sa}^d$ with $\max\{\|X_1\|_{L^2(\cA)},\|X_2\|_{L^2(\cA)},\|P_1\|_{L^2(\cA)},\|P_2\|_{L^2(\cA)}\}\leq M$ and $\epsilon>0$, one can find $\alpha_1\in \bA_\cA$ with $\|\alpha_1\|_{L^2(\cA)}\leq M'$ (see \eqref{201}) such that
				\begin{align*}
					H_{\cA}(X_1,P_1)&-H_{\cA}(X_2,P_2) \\
					&\leq \epsilon+ \langle b_{\cA}(X_1, \alpha_1),P_1\rangle_{L^2(\cA)} - L_{\cA}(X_1,\alpha_1) - \langle b_{\cA}(X_2, \alpha_1),P_2\rangle_{L^2(\cA)} + L_{\cA}(X_2,\alpha_1)\\
					&\leq \epsilon+\| b_{\cA}(X_1, \alpha_1) -  b_{\cA}(X_2, \alpha_1)\|_{L^2(\cA)}\, \|P_1\|_{L^2(\cA)}\\
					&\ \ \  \ \  + \|b_\cA(X_2,\alpha_1)\|_{L^2(\cA)}\, \|P_1-P_2\|_{L^2(\cA)} + |L_\cA(X_1,\alpha_1)-L_\cA(X_2,\alpha_1)|\\
					&\leq  \epsilon+(1+ M)\, C_2\, \|X_1-X_2\|_{L^2(\cA)} +\upsilon_M(1+M')\, \|P_1-P_2\|_{L^2(\cA)}.
				\end{align*}
				Taking $\epsilon\rightarrow 0$,
				we establish a uniform continuity of $H_{\cA}$ on bounded sets.

				
				To prove (\ref{eqn:Hamiltonian_bound}), for given $\gamma>1$, $\epsilon>0$ and $X,Y\in L^2(\mathcal{A})_{sa}^d$,  let $\alpha_2\in \mathbb{A}_{\cA}$ be such that
				$$
				H_{\cA}\big(Y,\gamma(X-Y)\big) \leq \epsilon + \langle b_{\cA}(Y, \alpha_2),\gamma(X-Y)\rangle_{L^2(\cA)} - L_{\cA}(Y,\alpha_2).
				$$
				Then
				\begin{align*}
					&\ H_{\cA}\big(X,\gamma(X-Y)\big)-H_{\cA}\big(Y,\gamma(X-Y)\big)\\
					\geq&\ \langle b_{\cA}(X, \alpha_2),\gamma(X-Y)\rangle_{L^2(\cA)} - \langle b_{\cA}(Y, \alpha_2),\gamma(X-Y)\rangle_{L^2(\cA)} -L_{\cA}(X,\alpha_2) +L_{\cA}(Y,\alpha_2) - \epsilon\\
					\geq&\ -\bar{C}\, \gamma\, \|X-Y\|_{L^2(\cA)}^2 - C_2\|X-Y\|_{L^2(\cA)}-\epsilon\geq - \big(\frac{2\bar{C}+C_2^2}{2}\gamma\,\|X-Y\|_{L^2(\cA)}^2 + \frac{1}{2\gamma}\big)-\epsilon,
				\end{align*}
				using the Lipschitz bound conditions (\ref{eqn:C_monotonicity}) and (\ref{eqn:Lipschitz_bounds}). Taking $\epsilon\rightarrow 0$ completes the proof.
			\end{proof}
			
			We now introduce a theorem on well-posedness for solutions to the SDE. This is essentially  \cite[Theorem 6.16]{yong1999stochastic} in a Hilbert space setting.  Much more general results can be readily found, for example, in \cite{fabbri2017stochastic}, but the generality complicates things, so we present a simpler setting.
			
			We could use a general It{\^o} formula, such as proven later in Lemma \ref{lem:Ito}, but for the theorem we need only to consider the simpler case of the test functions made up of the norm squared, which we can handle explicitly.

			\begin{theorem}\label{thm:well_posedness}
				Suppose that Assumption \textbf{A} holds.  Given $\cA\in \bW$, $x_0\in L^2(\cA_{t_0})^d_{sa}$, and $\widetilde{\alpha}\in \bA_{\cA,x_0}^{t_0,t_1}$, there exists a unique solution $(X_t)_{t\in [t_0,t_1]}\in C([t_0,t_1];L^2(\Omega,\mathcal{F},\bP;L^2(\cA)^d_{sa}))$ to  (\ref{eqn:common_noise}) with $X_{t_0}=x_0$ which is adapted in the sense that for every $s\in [t_0,t_1]$, $X_s\in L^2(\Omega,\mathcal{F}_s,\bP;L^2(\cA_s)^d_{sa})$. Furthermore, letting 
				\begin{align*}
					M:=  \|x_{0}\|_{L^2(\cA)}^2,\qquad  N:=\mathbb{E}\Big[\int_{t_0}^{t_1}\|\alpha_t\|_{L^2(\cA)}^2 dt\Big],
				\end{align*} 
				there exists $\widetilde{C}_{M,N}>0$, depending only on $M$ and $N$, such that for $t_0\leq t\leq t_1$,
				\begin{align}\label{eqn:time_bound_L2}
					\bE\Big[\big\|X_t-x_0\big\|_{L^2(\cA)}^2\Big]\leq \widetilde{C}_{M,N}\, (t-t_0).
				\end{align}
				
				Also, there exists $\widetilde{C}>0$ depending only on $\bar{C}$ (see \eqref{eqn:C_monotonicity}) such that if $\widehat{X}$ is a solution with another random initial condition $\widehat{x}_0\in L^2(\Omega,\mathcal{F}_{t_0},\bP;L^2(\cA_{t_0})^d_{sa})$ and the same control, then
				\begin{align}\label{eqn:initial_condition_continuity_2}
					\bE\Big[\big\|X_t-\widehat{X}_t\big\|_{L^2(\cA)}^2\Big]\leq \widetilde{C}\, \mathbb{E}\big[\|x_0 - \widehat{x}_0\|_{L^2(\cA)}^2\big].
				\end{align}


			\end{theorem}
			\begin{proof}
				Using a standard Picard iteration method, we can find a unique solution for small times that is adapted to the filtration; see Appendix \ref{apx:SDE_theory}.  Let $\delta>0$ be arbitrary and recall the constant
				$v_M$ from \eqref{eqn:b_continuity}. Then,
			using Assumption \textbf{A}, 
			\begin{align*}
				&\frac{d}{dt}\bE\Big[\|X_t-x_0\|^2_{L^2(\cA)}\Big]\\
				=&\   2\bE\Big[\langle X_t-x_0,b_{\cA}(X_t,\alpha_t)\rangle_{L^2(\cA)}\Big] + d\, (\beta_C^2 + \beta_F^2)\\
				=&\  2\bE\Big[\langle X_t-x_0,b_{\cA}(X_t,\alpha_t)-b_{\cA}(x_0,\alpha_t)\rangle_{L^2(\cA)} + \langle X_t-x_0,b_{\cA}(x_0,\alpha_t)\rangle_{L^2(\cA)}\Big] + d\, (\beta_C^2 + \beta_F^2) \\
				\leq&\ 2\bE\Big[\bar{C}\, \|X_t-x_0\|_{L^2(\cA)}^2\big] +\|X_t-x_0\|_{L^2(\cA)}\, \upsilon_M\big(1+\|\alpha_t\|_{L^2(\cA)}\big)\Big] + d\, (\beta_C^2 + \beta_F^2)\\
				\leq&\ (2\bar{C} +1 + \delta)\, \bE\big[\|X_t-x_0\|_{L^2(\cA)}^2\big] +\frac{\upsilon_M^2}{ \delta}\bE\big[\|\alpha_t\|^2_{L^2(\cA)}\big] + \upsilon_M^2 + d\, (\beta_C^2 + \beta_F^2),
			\end{align*}
			{where the first identity follows from Appendix \ref{apx:SDE_theory}.} 
			By Gr\"onwall's inequality,
			$$
			\bE\Big[\|X_t-x_0\|^2_{L^2(\cA)}\Big] \leq y_t,
			$$
			where $y_t$ is defined by, setting $r:=2\bar{C}+1$ and $D_M := \upsilon_M^2  + d\, (\beta_C^2 + \beta_F^2)$,
			\begin{align*}
				y_t :=&\  e^{(r+\delta)\, (t-t_0)}\Big(\int_{t_0}^t e^{-(r+\delta)\, (s-t_0)}\big( \frac{\upsilon_M^2}{ \delta}\bE\big[\|\alpha_s\|^2_{L^2(\cA)_{sa}^d}\big] + D_M\big)ds\Big)\\
				\leq&\ e^{(r+\delta)\, (t-t_0)}\big( \frac{\upsilon_M^2}{ \delta}\, N + (t-t_0) D_M\big) .
			\end{align*}
			Choosing $\delta: = \frac{1}{t-t_0}$, we arrive at (\ref{eqn:time_bound_L2}) for some constant $\widetilde{C}_{M,N}$. 

			To show (\ref{eqn:initial_condition_continuity_2}),  note that 
			\begin{align*}
				\frac{d}{dt}\bE\big[\|X_t-\widehat{X}_t\|^2_{L^2(\cA)}\Big] =&\ 2 \bE\Big[\langle X_t-\widehat{X}_t,b_{\cA}(X_t,\alpha_t)-b_{\cA}(\widehat{X}_t,{\alpha}_t)\rangle_{L^2(\cA)} \Big]\\
				\leq&\ 2\mathbb{E}\Big[\|X_t-\widehat{X}_t\|_{L^2(\cA)}\|b(X_t,{\alpha}_t)-b(\widehat{X}_t,{\alpha}_t)\|_{L^2(\cA)} \Big]\\
				\leq&\ 2\bar{C}\, \bE\Big[\|X_t-\widehat{X}_t\|^2_{L^2(\cA)}\Big].
			\end{align*}
			Then (\ref{eqn:initial_condition_continuity_2}) follows from the Gr\"onwall's inequality. 
		\end{proof}
		
		\subsection{Properties of the value function} \label{subsec: value function}
		
		In this section, we prove well-definedness and continuity properties for the value function, including showing that $\overline{V}$ is well-defined on the space of non-commutative laws.
		
		We begin by proving an auxiliary lemma about amalgamating several different non-commutative filtrations and associated Brownian motions.  Of course, the analogs of these lemmas in classical probability would be proved using conditional distributions, but for the non-commutative setting we will use amalgamated free products.  In the following lemma, we describe how to glue together different choices of filtration and free Brownian motion on an interval $[t_1,T]$ together with a fixed choice of filtration and free Brownian motion on $[t_0,t_1]$, where $0\le t_0\le t_1\le T$.  While the general case is used in Proposition \ref{prop:dynamic_programming}, the applications in this section are mostly in the case $t_0 = t_1$. These lemmas require significant technical effort to prove, so in order to maintain the flow of this section, we include detailed arguments in the appendix (\S \ref{subsec: amalgamation of filtrations}).
		
		\begin{lemma}[{Lemma \ref{lem: amalgamation of Brownian motions appendix}}] \label{lem: amalgamation of Brownian motions}
			Let $0 \leq t_0 \leq t_1 \leq T$ and $K$ be any index set.  Let $\cA$ be a tracial von Neumann algebra equipped with a filtration $(\cA_t)_{t \in [t_0,t_1]}$ and a compatible $d$-variable free Brownian motion $(S_t^0)_{t \in [t_0,t_1]}$. 
			Let $(\cB^k)_{k \in K}$ be a family of tracial von Neumann algebras and let $\iota_k: \cA \to \cB^k$ be a tracial $\mathrm{W}^*$-embedding.  Suppose each $\cB^{k}$ has a filtration $(\cB_t^k)_{t \in [t_1,T]}$ and a compatible $d$-variable free Brownian motion $(S_t^k)_{t \in [t_1,T]}$.  Assume that $\iota_k(\cA_{t_1}) \subseteq \cB_{t_1}^k$.
			
			Then there exists an algebra $\cB$, a tracial $\mathrm{W}^*$-embeddings $\iota: \cA \to \cB$ and $\widetilde{\iota}_k: \cB^k \to \cB$, a filtration $(\cB_t)_{t \in [t_0,T]}$, and a compatible  $d$-variable free Brownian motion $(S_t)_{t \in [t_0,T]}$ such that the following hold:
			\begin{enumerate}
				\item We have $\widetilde{\iota}_k \circ \iota_k|_{\cA_{t_0}} = \iota|_{\cA_{t_0}}$ for each $k \in K$.
				\item We have $\iota(S_t^0) = S_t$ for $t \in [t_0,t_1]$.
				\item We have $\widetilde{\iota}_k(S_t^k) = S_t - S_{t_1}$ for $t \in [t_1,T]$.
				\item We have $\iota(\cA_t) \subseteq \cB_t$ for $t \in [t_0,t_1]$.
				\item We have $\widetilde{\iota}_k(\cB_t^k) \subseteq \cB_t$ for $t \in [t_1,T]$.
			\end{enumerate}
			Furthermore, for any given $k_0 \in K$, the embedding $\iota$ can be taken to be $\iota = \widetilde{\iota}_{k_0} \circ \iota_{k_0}$.
		\end{lemma}

		The following is an important application of the previous lemma, which will allow us to choose filtrations in $\cB$ such that $\cB_{t_0} \supseteq \iota(\cA)$.  In turn, the freedom in choosing the initial subalgebra of filtration provided by Lemma \ref{lem: arranging initial filtration} will help us to apply Lemma \ref{lem: amalgamation of Brownian motions}.
		
		\begin{lemma}[{Lemma \ref{lem: arranging initial filtration appendix}}] \label{lem: arranging initial filtration}
			Suppose that Assumption \textbf{A} holds.  Let $\cA$ be a tracial $\mathrm{W}^*$-algebra, let $x_0 \in L^2(\cA)_{\sa}^d$, and let $t_0 \in [0,T]$.  Then for every $\epsilon > 0$, there exists a tracial $\mathrm{W}^*$-algebra $\cB$, a tracial $\mathrm{W}^*$-embedding $\iota: \cA \to \cB$, and a control policy $\alpha \in \mathbb{A}_{\mathcal{B},\iota x_0}^{t_0,T}$ associated with a filtration $(\cB_t)_{t \in [t_0,T]}$ and compatible $d$-variable free Brownian motion $(S_t)_{t \in [t_0,T]}$, such that
			\begin{enumerate}
				\item $\iota(\cA) \subseteq \cB_{t_0}$
				\item $\mathbb{E}\left[ \int_{t_0}^T L_{\mathcal{B}}(X_t[\widetilde \alpha],\alpha_t)\,dt + g_{\mathcal{B}}(X_T[\alpha]) \right] \leq 
				\overline{V}_{\mathcal{A}}(t_0,x_0) + \epsilon$.
			\end{enumerate}
		\end{lemma}
		
	%
	
	With Lemmas \ref{lem: amalgamation of Brownian motions} and \ref{lem: arranging initial filtration}, we can now prove various properties of the value function that we want, starting with consistency condition of being a tracial $\mathrm{W}^*$-function.
	
	\begin{lemma}\label{lem:nov18.2023.2} Suppose that Assumption \textbf{A} holds.
		Then for all $t_0 \in [0,T]$, $(\overline{V}_\cA(t_0,\cdot))_{\cA \in \bW}$ is a tracial $W^*$--function. 
	\end{lemma}

	\begin{proof} Let $\cA,\cB\in \bW$ and $\iota:\cA\rightarrow \mathcal{B}$ be a tracial $W^*$--embedding. 
		Note that if $\cC\in \bW$ and $\kappa:\cB\rightarrow \mathcal{C}$ is a tracial $W^*$--embedding then $\kappa \circ \iota:\cA\rightarrow \mathcal{C}$ is also a tracial $W^*$--embedding and thus, $\overline{V}_\cA(t_0, x_0) \leq \widetilde{V}_\cC(t_0, \kappa \circ \iota\, x_0)$ for any $x_0 \in L^2(\cA)^d_{sa}$.  Minimizing over all $\cC\in \bW$ and $\kappa:\cB\to\cC$, we conclude that $\overline{V}_\cA(t_0, x_0) \leq \overline{V}_\cB(t_0, \iota\, x_0).$ 
		
		We now show the reverse direction. Given $\epsilon>0$, there exists $\mathcal{B}^1 \in \bW$ and a tracial $W^*$--embedding $\iota_1:\cA\rightarrow \mathcal{B}^1$ such that 
		\begin{align} \label{291}
			\widetilde{V}_{\mathcal{B}^1}(t_0, \iota_1\, x_0) < \overline{V}_\cA(t_0, x_0) + \epsilon.
		\end{align}
		Let $\widetilde \alpha^1=\big((\alpha^1_t)_t, (\mathcal{B}^1_t)_t, (S^1_t)_t\big) \in  \mathbb{A}^{t_0,T}_{\mathcal{B}^1, \iota_1 x_0}$ be a control policy 
		such that
		\begin{equation}\label{eq:nov18.2023.2}
			\bE\bigg[\int_{t_0}^T L_{\mathcal{B}^1} (X_t[\widetilde \alpha^1], \alpha_t^1)dt + g_{\mathcal{B}^1}(X_T[\widetilde \alpha^1])\bigg] < \widetilde{V}_{\mathcal{B}^1}(t_0, \iota_1\, x_0) + \epsilon ,
		\end{equation}
		where $X_t[\widetilde \alpha^1]$ denotes a solution to \eqref{eqn:common_noise} (with $\cA$ replaced by $\mathcal{B}^1$) on $ L^2(\mathcal{B}^1)_{sa}^d$  with the initial condition $X_{t_0}[\widetilde \alpha^1]=\iota_1\, x_0$.

		
		
		{Using Lemma \ref{lem: uniqueness of free product}, one can take a tracial $W^*$--algebra $\cC$ and
			tracial $W^*$--embeddings $\phi: \cB \to \cC$ and $\phi_1: \mathcal{B}^1 \to \cC$ such that $\phi \circ \iota= \phi_1 \circ \iota_1$.}
		In light of Remark \ref{rem:commute}, setting
		\begin{align*}
			X_t&:= \phi_1\big(X_t [\widetilde \alpha^1] \big) \in C([t_0, T];L^2( \Omega,\mathcal{F},\mathbb{P}; L^2(\cC)^d_{sa})) , \\
			\widetilde \alpha&:= \Big((\phi_1\, \alpha_t^1)_t, (\phi_1\, \mathcal{B}^1_t)_t, (\phi_1\, S^1_t)_t\Big) 
			\in \mathbb{A}^{t_0,T}_{\cC , \phi_1(\iota_1 x_0)}
		\end{align*}
		(i.e. in particular $\alpha_t := \phi_1\, \alpha_t^1$)
		and
		\begin{align*}
			S_t:=  \phi_1\, S^1_t\quad  \text{for } t\in [t_0,T],  
		\end{align*} 
		we have  
		\begin{align*}
			\begin{cases} dX_t = b_\cC(X_t,\alpha_t)\, dt +  \beta_C\, \mathbbm{1}_{\cC}\, dW_t^0 +  \beta_F\, dS_t, \\
				X_{t_0}=\phi_1(\iota_1\, x_0).
			\end{cases}
		\end{align*}
		Thus interchangeably using the notations $X_t$ and  $X_t[\widetilde{\alpha}],$
		\begin{align*}
			\widetilde{V}_{\cC}\big(t_0, \phi_1(\iota_1\, x_0)\big) \leq&\  \bE\bigg[\int_{t_0}^T L_{\cC} (X_t[\widetilde{\alpha}], \alpha_t)dt + g_{\cC}(X_T[\widetilde{\alpha}])\bigg]\\
			=&\ \bE\bigg[\int_{t_0}^T L_{\mathcal{B}^1} (X_t[\widetilde{\alpha}^1], \alpha_t^1)dt + g_{\mathcal{B}^1}(X_T[\widetilde{\alpha}^1])\bigg].
		\end{align*}
		Therefore using this along with the fact that $\phi \circ \iota=\phi_1 \circ \iota_1,$ we conclude that 
		\[
		\widetilde{V}_{\cC}\big(t_0, \phi ( \iota\, x_0)\big) \overset{\eqref{eq:nov18.2023.2}}{\le} \widetilde{V}_{\mathcal{B}^1}(t_0, \iota_1\, x_0) + \epsilon \overset{\eqref{291}}{\le}  \overline{V}_\cA(t_0, x_0) + 2 \epsilon.
		\]
		Since $\phi: \cB \to \cC$ is a tracial $W^*$--embedding, $\overline{V}_\cB\big(t_0, \iota\, x_0\big) \leq 2 \epsilon+ \overline{V}_\cA(t_0, x_0).$ By the arbitrariness of $\epsilon>0$, we infer $\overline{V}_\cB\big(t_0, \iota\, x_0\big) \leq \overline{V}_\cA(t_0, x_0)$.
	\end{proof}
	
	In the next lemma, we verify that two versions of values functions \eqref{eqn:l2_value} and  \eqref{def:bar} are same, under the additional Assumption \textbf{B}, provided that $\cA$ admits a free Brownian motion freely independent of the initial condition.  To clarify the terminology in the statement below, for a function $f: X \to \bR$ and $X_0 \subseteq X$, we say that the infimum $\inf_{x \in X} f(x)$ is \emph{witnessed by $X_0$} if $\inf_{x \in X_0} f(x) = \inf_{x \in X} f(x)$, but note that this does not require that the infimum is \emph{achieved} in $X_0$ or even in $X$.
	
	\begin{lemma}\label{lem:decreasing}
		Suppose that Assumptions \textbf{A} and \textbf{B} hold. 
		
		1. {Let $\cA \in \mathbb{W}$, and let $x_0 \in L^2(\cA)_{\sa}^d$ and $t_0 \in [0,T]$.  Suppose that $\cA$ admits a $d$-variable free Brownian motion $(S_t^0)_{t \in [t_0,T]}$ freely independent of $\mathrm{W}^*(x_0)$. For  $t \in [t_0,T],$ let
			\[
			\cA_t^0 := \mathrm{W}^*(x_0, (S^0_s)_{s \in [t_0,t]}).
			\]
			Then we have
			\[
			\overline{V}_{\cA}(t_0,x_0) = \widetilde{V}_{\cA}(t_0,x_0),
			\]
			and the infimum in \eqref{eqn:l2_value} is witnessed by control policies $\widetilde{\alpha}\in \mathbb{A}_{\cA,x_0}^{t_0,T}$  that use the given filtration $(\cA_t^0)_{t \in [t_0,T]}$ and free Brownian motion $(S_t^0)_{t \in [t_0,T]}$.
			
			2. Let $\cA \in \mathbb{W}$. Let $\mathcal{C}$ be a tracial von Neumann algebra generated by a $d$-variable free Brownian motion $(S^1_t)_{t \in [t_0,T]}$, and define $\mathcal{C}_{t_0,t} := \mathrm{W}^*(S^1_s: s \in [t_0,t])$ for $t \in [t_0,T]$. Let $\iota_1: \cA \to \cA \median \cC$ and $\iota_2: \cC \to \cA * \cC$ be the inclusions associated to the free product.  Then for any  $x_0 \in L^2(\cA)_{\sa}^d$,
			\[
			\overline{V}_{\cA}(t_0,x_0) = \widetilde{V}_{\cA * \cC}(t_0,\iota_1(x_0)),
			\]
			and the infimum is witnessed by control polices that use the fixed filtration $(\iota_1(\cA) \vee \iota_2(\cC_{t_0,t}))_{t \in [t_0,T]}$ and free Brownian motion $(\iota_2(S^1_t))_{t \in [t_0,T]}$.
		} 
	\end{lemma}
	
	\begin{proof}
		\textbf{Step 1: Conditioning control policy.} Let $\iota: \cA \to \cB$ be a tracial $\mathrm{W}^*$-embedding and let $ \widetilde{\alpha}  = ((\alpha_t)_{t \in [t_0,T]}, (\cB_t)_{t \in [t_0,T]},(S_t)_{t \in [t_0,T]})$ be a control policy  in $\cB$.  Let
		\[
		\cB_t^0 := \mathrm{W}^*(\iota(x_0), (S_s)_{s \in [t_0,t]});
		\]
		it is straightforward to check that the Brownian motion $(S_t)_{t \in [t_0,T]}$ is compatible with the filtration $(\cB_t^0)_{t \in [t_0,T]}$. For $t \in [t_0,T]$, let $E_{\cB_t^0}:\cB \to \cB_t^0$ be the trace-preserving conditional expectation, and define
		\[
		\alpha_t^0 := E_{\cB_t^0}[\alpha_t].
		\]
		We claim that $\alpha_t^0$ is measurable as a function on $[t_0,T]$ times the underlying probability space.  Since the Hilbert space $L^2(\cB)$ is assumed to be separable, it suffices to check weak measurability, so fix some $y \in L^2(\cB)$ and by linearity we can assume without loss of generality that $y$ is self-adjoint.  Then
		\[
		\ip{\alpha_t^0,y}_{L^2(\cB)} = \ip{E_{\cB_t^0}[\alpha_t],y}_{L^2(\cB)} = \ip{\alpha_t, E_{\cB_t^0}[y]}_{L^2(\cB)}.
		\]
		Since $t\mapsto\alpha_t$ is measurable, it suffices to check that $t \mapsto E_{\cB_t^0}[y]$ is weakly measurable.  Let $z \in L^2(\cB)_{\sa}$.  Then,
		\[
		\ip{E_{\cB_t^0}[y],z}_{L^2(\cB)} = \ip{E_{\cB_t^0}[y], E_{\cB^0_t}[z]}_{L^2(\cB)} = \frac{1}{4} \left( \norm{E_{\cB_t^0}[x+y]}_{L^2(\cB)}^2 - \norm{E_{\cB_t^0}[x-y]}_{L^2(\cB)}^2 \right),
		\]
		and $\norm{E_{\cB_t^0}[x+y]}_{L^2(\cB)}^2$ along with $\norm{E_{\cB_t^0}[x-y]}_{L^2(\cB)}^2$ are increasing functions of $t$, hence measurable.  Hence, $t\mapsto\alpha_t^0$ is measurable as desired, and by construction it is adapted to $(\cB_t^0)_t$.
		
		~
		
		\textbf{Step 2: Conditioning trajectory.}  Let $\widetilde{\alpha}^0 = ((\alpha_t^0)_{t \in [t_0,T]}, (\cB_t^0)_{t \in [t_0,T]},(S_t)_{t \in [t_0,T]})$ be the control policy in $\cB$.  We claim that $$X_t[\widetilde{\alpha}^0] = E_{\cB_t^0}  (X_t[\widetilde{\alpha}]),$$ which in particular implies that the initial condition is $X_{t_0}[\widetilde{\alpha}^0] = X_{t_0}[\widetilde{\alpha}] = \iota(x_0)$.  Note that
		\[
		E_{\cB_t^0}[X_t[\widetilde{\alpha}]] = x_0 + \int_{t_0}^t E_{\cB_t^0}[\alpha_s] ds + \beta_C\, \mathbbm{1}_\cA\, (W_t^0-W^0_{t_0}) + \beta_F\, (S_t - S_{t_0}) = X_t[\widetilde{\alpha}],
		\]
		while
		\[
		X_t[\widetilde{\alpha}^0] = x_0 + \int_{t_0}^t E_{\cB_s^0}[\alpha_s]ds + \beta_C\, \mathbbm{1}_\cA\, (W_t^0-W^0_{t_0}) + \beta_F\, (S_t - S_{t_0}) = X_t[\widetilde{\alpha}].
		\]
		Hence, it suffices to show that for $s \leq t$, we have
		\[
		E_{\cB_t^0}[\alpha_s] = E_{\cB_s^0}[\alpha_s].
		\]
		For this, we use a certain fact about conditional expectations inside free products (Lemma \ref{lem: free product commuting square}).  Let $\cD := \mathrm{W}^*(S_{s'} - S_s: s' \in [s,t])$, and recall that $\cD$ is freely independent of $\cB_s$ since the free Brownian motion $(S_t)_{t \in [t_0,T]}$ is compatible with the filtration $(\cB_t)_{t \in [t_0,T]}$.  Thus, $\cB_s \vee \cD \cong \cB_s * \cD$ (Corollary \ref{cor: independence and embeddings}); moreover, the subalgebra $\cB_t^0 = \cB_s^0 \vee \cD$ in $\cB$ corresponds to the subalgebra $\cB_s^0 * \cD$ in $\cB_s * \cD$.  Let $\overline \iota: \cB_s \to \cB_s * \cD$ be the canonical inclusion into the free product; then Lemma \ref{lem: free product commuting square} shows that
		\[
		E_{\cB_s^0 * \cD} \circ \overline \iota(\alpha_s) = \overline \iota \circ E_{\cB_s^0}[\alpha_s] \text{ in } \cB_s * \cD,
		\]
		hence
		\[
		E_{\cB_t^0}[\alpha_s] = E_{\cB_s^0}[\alpha_s] \text{ in } \cB,
		\]
		which proves our claim.
		
		~
		
		\textbf{Step 3: Comparing the value functions.} For the control policy  $\widetilde{\alpha}^0$ defined above, 
		by $E$-convexity of $(L_{\cA})_{\cA\in \bW}$ and $(g_{\cA})_{\cA\in \bW}$ (see Assumption \textbf{B}), 
		\begin{align*}
			\bE\bigg[\int_{t_0}^T L_{\cB} ( X_t[\widetilde{\alpha}^0], \alpha_t^0)dt + g_{\cB}(X_T[\widetilde{\alpha}])\bigg] &= \bE\bigg[\int_{t_0}^T L_{\cB} ( E_{\cB_t^0}[X_t[\widetilde{\alpha}^0]], E_{\cB_t^0}[\alpha_t^0])dt + g_{\cB}(E_{\cB_T^0}[X_T[\widetilde{\alpha}]])\bigg] \\
			&\leq \bE\bigg[\int_{t_0}^T L_{\cB} ( X_t[\widetilde{\alpha}], \alpha_t) dt + g_{\cB}(X_T[{\widetilde{\alpha}}])\bigg].
		\end{align*}

		\textbf{Step 4: Transformation from $\cB_t^0$ to $\cA_t^0$.}  Recall that
		\[
		\cB_T^0 = \mathrm{W}^*(\iota(x_0),(S_t)_{t \in [t_0,T]}) \subseteq \cB, \qquad \cA_T^0 = \mathrm{W}^*(x_0,(S_t^0)_{t \in [t_0,T]}) \subseteq \cA.
		\]
		Because the tracial von Neumann algebra generated by a free Brownian motion is unique up to a canonical isomorphism, and the free product of two given tracial von Neumann algebras is unique up to a canonical isomorphism (Lemma \ref{lem: uniqueness of free product}), similar to the beginning of the proof of Lemma \ref{lem: amalgamation of Brownian motions}, we conclude there is a unique tracial $\mathrm{W}^*$-isomorphism $\phi: \cB_T^0 \to \cA_T^0$ such that $\phi(\iota(x_0)) = x_0$ and $\phi(S_t) = S_t^0$ for $t \in [t_0,T]$.  In particular, we also have $\phi(\cB_t^0) = \cA_t^0$ for $t \in [t_0,T]$.  Let $\alpha_t' := \phi(\alpha_t^0)$ for $t \in [t_0,T]$, which is now a process in $\cA_0^T \subseteq \cA$; and consider the control policy $\widetilde{\alpha}' = ((\alpha_t')_{t \in [t_0,T]}, (\cA_t^0)_{t \in [t_0,T]},(S_t^0)_{t \in [t_0,T]})$ in $\cA$.  It is immediate that $X_t[\widetilde{\alpha}'] = \phi(X_t[\widetilde{\alpha}^0])$.  Since $L$ and $g$ are tracial $\mathrm{W}^*$-functions, the value achieved by the control policy $\widetilde{\alpha}'$ is the same as the value achieved by $\widetilde{\alpha}^0$.  Therefore, we have shown that the infimum for $\overline{V}_{\cB}(t_0,\iota (x_0))$ is witnessed by control policies in $\cA$ that use the given free Brownian motion $(S_t^0)_{t \in [t_0,T]}$ and filtration $(\cA_t^0)_{t \in [t_0,T]}$, and so in particular $\overline{V}_{\cA}(t_0,x_0) = \widetilde{V}_{\cA}(t_0,x_0)$.
		
		~
		
		\textbf{Step 5: Proof of second claim.}    Finally, we prove the  second part of Lemma.   Since the free Brownian motion $(\iota_2(S^1_t))_{t \in [t_0,T]}$ is freely independent of the initial condition $\iota_1(x_0)$, the preceding argument applies and shows that the infimum in the definition of $\overline{V}_{\cA * \cC}(t_0, \iota_1(x_0))$ is witnessed by control policies using the given  free Brownian motion $(\iota_2(S^1_t))_{t \in [t_0,T]}$ and the filtration $(\cB_t)_{t \in [t_0,T]}$ given by $\cB_t = \mathrm{W}^*(\iota_1(x_0), (\iota_2(S^1_s))_{s \in [t_0,t]})$.  Of course, $\cB_t \subseteq \iota_1(\cA) \vee \iota_2(\cC_{t_0,t})$, and the free Brownian motion $(\iota_2(S^1_t))_{t \in [t_0,T]}$ is still compatible with this larger filtration.  Hence, any control policy using this free Brownian motion and the smaller filtration is also a valid control policy with respect to the larger filtration.
	\end{proof}
	
	\begin{lemma} \label{lem: decreasing 2}
		Suppose Assumptions \textbf{A} and \textbf{B} hold.  Then for any  $t_0\in [0,T]$, $(\overline{V}_\cA(t_0,\cdot))_{\cA \in \bW}$ is $E$-convex. 
	\end{lemma}
	
	\begin{proof}
		\textbf{Step 1: Convexity  of $x\mapsto \overline{V}_\cA(t_0,x)$.} 
		Let  $\cA \in \bW$ be any tracial $W^*$--algebra and $t_0 \in [0,T]$.  Fix $x_0, x_1 \in L^2(\cA)^d_{sa}$ and $\epsilon>0$. 
		By Lemma \ref{lem: arranging initial filtration}, for $k=0,1,$ there exist $\cB^k \in \bW$, a tracial $W^*$--embedding $\iota_k: \cA \to \cB^k$, and a control policy
		\begin{equation}\label{eq:oct15.2024.1}
			\widetilde{\alpha}^k := \Big((\alpha^k_t)_{t\in [t_0, T]},(\cB^k_t)_{t\in [t_0, T]}, (S^k_t)_{t\in [t_0,T]}\Big) \in \bA_{\cB^k, \iota_k (x_k)}^{t_0, T}
		\end{equation}
		such that $\iota_k(\cA) \subseteq \cB^k_{t_0}$, $X_{t_0}\big[\widetilde{\alpha}^k\big]=\iota_k(x_k)$ and 
		\begin{equation}\label{eq:oct15.2024.0}
			\bE\bigg[\int_{t_0}^T L_{\cB^k} \Big(X_t[\widetilde{\alpha}^k], \alpha^k_t\Big)dt + g_{\cB^k}\Big(X_T[\widetilde{\alpha}^k]\Big)\bigg]  \leq \overline{V}_{\cA}(t_0,x_k) + \epsilon .
		\end{equation}
		Now we apply Lemma \ref{lem: amalgamation of Brownian motions} with $t_0 = t_1$ and $\cA_{t_0} = \cA_{t_1} = \cA$ (since the time interval has length zero, the free Brownian motion on $[t_0,t_1]$ reduces to $0$).
		Recalling $\iota_k(\cA) \subset \cB^k_{t_0}$, by Lemma \ref{lem: amalgamation of Brownian motions}, there exists a tracial $W^*$--algebra $\cB,$ tracial $W^*$--embeddings $\iota: \cA \to \cB$, $\widetilde \iota_k: \cB^k \to \cB$, and a filtration $(\cB_t)_{t \in [t_0,T]}$ with a $d$--variable free Brownian motion $(S_t)_{t\in [t_0,T]}$ such that 
		\begin{enumerate}
			\item $\widetilde{\iota}_k(S^k_t)=S_t$ for each $t \in [t_0, T].$
			\item $\widetilde{\iota}_k(\cB^k_t) \subset \cB_t$ for each $t \in [t_0, T].$
			\item $(S_t)_{t\in [t_0,T]}$ is a Brownian motion compatible with the filtration $(\cB_t)_{t \in [t_0,T]}.$
			\item $\widetilde{\iota}_k \circ \iota _k|_{\cA} =\iota|_{\cA}$.
		\end{enumerate}
		By \eqref{eq:oct15.2024.1},
		\[
		{\widetilde \iota_k  \widetilde{\alpha}^k}  = \Big((\widetilde{\iota}_k \alpha^k_t)_{t\in [t_0, T]},(\widetilde{\iota}_k \cB^k_t)_{t\in [t_0, T]}, (\widetilde{\iota}_k S^k_t)_{t\in [t_0,T]}\Big) \in \bA_{\widetilde{\iota}_k\cB^k, \widetilde{\iota}_k \circ \iota_k (x_k)}^{t_0, T}.
		\]
		Since $(\widetilde{\iota}_k \alpha^k_t)_{t\in [t_0,s]}\in L^2\Big([t_0,s]\times (\Omega, (\widetilde{\iota}_k \mathcal{B}^k_t)_t,\mathbb{P})\Big)$, (2) implies that  
		$
		(\widetilde{\iota}_k \alpha^k_t)_{t\in [t_0,s]}\in L^2\Big([t_0,s]\times (\Omega, (\mathcal{B}_t)_t,\mathbb{P})\Big).
		$
		This, together with (2) and (3), implies that 
		\[
		\beta^k: =  \Big((\widetilde{\iota}_k \alpha^k_t)_{t\in [t_0, T]},(\cB_t)_{t\in [t_0, T]}, (S_t)_{t\in [t_0,T]}\Big) \in \bA_{\cB, \widetilde{\iota}_k \circ \iota_k (x_k)}^{t_0, T}.
		\]
		Observe that
		\[
		X_t[\beta^k] = X_t\Big[\widetilde{\iota}_k\Big(\widetilde{\alpha}^k\Big)\Big]= \widetilde{\iota}_k\Big( X_t\Big[\widetilde{\alpha}^k\Big] \Big), \qquad X_{t_0}[ \beta^k]= \widetilde{\iota}_k \circ \iota_k (x_k).
		\]
		Since $L$ and $g$ are tracial functions, we have 
		\begin{align} \label{eq:oct15.2024.3}
			\bE\bigg[\int_{t_0}^T L_{\cB} \Big(X_t[ \beta^k ], \widetilde{\iota}_k \alpha^k_t\Big)dt + g_{\cB}\Big(X_T [\beta^k] \Big)\bigg] &=  \bE\bigg[\int_{t_0}^T L_{\cB^k} \Big(X_t[\widetilde{\alpha}^k], \alpha^k_t\Big)dt + g_{\cB^k}\Big(X_T[\widetilde{\alpha}^k]\Big)\bigg] \nonumber  \\
			&  \overset{\eqref{eq:oct15.2024.0}}{\le}  \overline{V}_{\cA}(t_0,x_k) + \epsilon. 
		\end{align}
		For $s \in (0,1), $ we set 
		\[
		\widetilde{\beta}^s:=\Big(\big(\beta^s_t\big)_{t\in [t_0, T]},(\cB_t)_{t\in [t_0, T]}, (S_t)_{t\in [t_0,T]}\Big), \qquad \beta^s_t:= (1-s)  \widetilde \iota_0 \alpha^0_t + s  \widetilde \iota_1 \alpha^1_t,
		\]
		and 
		$$
		z_s:= (1-s)  \widetilde \iota_0 \circ \iota_0(x_0)+ s \widetilde \iota_1 \circ \iota_1(x_1).
		$$ 
		Since $\widetilde{\beta}^s \in \bA_{\cB, z_s}^{t_0, T}$, we infer that
		\[
		\overline{V}_{\cB}(t_0, z_s) \leq \bE\bigg[\int_{t_0}^T L_{\cB} \Big(X_t[\widetilde{\beta}^s], \beta^s _t\Big)dt + g_{\cB}\Big(X_T[\widetilde{\beta}^s]\Big)\bigg].
		\]
		We use the fact that 
		\[
		X_t[\widetilde{\beta}^s]=(1-s) X_t[{\beta^0}]+ s X_t[{\beta^1}], \qquad \beta^s_t= (1-s)  \widetilde \iota_0 \alpha^0_t + s \widetilde \iota_1 \alpha^1_t , \qquad X_{t_0}[\widetilde{\beta}^s]= z_s
		\]
		along with the convexity property of $L_\cB$ and $g_\cB$ to conclude that  
		\begin{multline*}
			\overline{V}_{\cB}(t_0, z_s) \leq 
			(1-s) \bE\bigg[\int_{t_0}^T L_{\cB} \Big(X_t[{\beta^0}],  \widetilde \iota_0 \alpha^0_t \Big)dt + g_{\cB}\Big( X_T[{\beta^0}]\Big)\bigg] 
			\\
			+  s \bE\bigg[\int_{t_0}^T L_{\cB} \Big( X_t[{\beta^1}],  \widetilde \iota_1 \alpha^1_t \Big)dt + g_{\cB}\Big( X_T[{\beta^1}]\Big)\bigg].
		\end{multline*}
		Combining this with \eqref{eq:oct15.2024.3}, we  obtain
		\begin{equation}\label{eq:oct15.2024.5}
			\overline{V}_{\cB}(t_0, z_s)  \leq (1-s) \overline{V}_{\cA}(t_0, x_0) +s \overline{V}_{\cA}(t_0, x_1) +\epsilon.
		\end{equation} 
		Since $\widetilde{\iota}_k \circ \iota_k(x_k)=\iota(x_k),$ 
		\[
		z_s= (1-s)\iota(x_0)+ s\iota(x_1)=\iota\big((1-s)x_0 + s x_1 \big).
		\]
		Since $\iota: \cA \rightarrow \cB$ is a tracial $W^*$--embedding and $\overline V$ is a tracial $\mathrm{W}^*$--function (see Lemma \ref{lem:nov18.2023.2}), \eqref{eq:oct15.2024.5} implies that 
		\[
		\overline{V}_{\cA}\big(t_0,(1-s)x_0 + s x_1 \big) \leq (1-s) \overline{V}_{\cA}(t_0, x_0) +s \overline{V}_{\cA}(t_0, x_1) +\epsilon.
		\]
		Since $\epsilon>0$ was arbitrary, we have shown convexity of $\overline{V}_{\cA}(t_0, \cdot)$.
		
		~
		
		\textbf{Step 2: $E$-convexity  of $(\overline{V}_\cA(t_0,\cdot))_{\cA \in \bW}$.}  Since we already showed that $\overline{V}$ is a tracial $\mathrm{W}^*$--function (Lemma \ref{lem:nov18.2023.2}) and we showed convexity of $\overline{V}_{\cA}$ in Step 1, it only remains to show that for any  tracial $W^*$--embedding $\iota: \cA \rightarrow \cB$ and its adjoint $E: \cB \rightarrow \cA$, we have
		\begin{align} \label{226}
			\overline{V}_\cA(t_0, Ey_0) \le \overline{V}_\cB(t_0, y_0),\qquad \forall y_0 \in L^2(\cB)^d_{sa}. 
		\end{align}
		As in Lemma \ref{lem:decreasing}, let $\cC$ be the tracial $\mathrm{W}^*$--algebra generated by a free Brownian motion $(S_t)_{t \in [t_0,T]}$.  Let $\overline \iota_1: \cB \to \cB * \cC$ and $\overline \iota_2: \cC \to \cB * \cC$ be the inclusions from the free product construction.  By Lemma \ref{lem:decreasing}, for each $\epsilon > 0$, there exists a control policy of the form
		\[
		\widetilde{\alpha} = ((\alpha_t)_{t \in [t_0,T]}, (\overline \iota_1(\cA) \vee \overline \iota_2(\cC_{t_0,t}))_{t \in [t_0,T]}, (\overline \iota_2(S_t))_{t \in [t_0,T]}),
		\]
		where $\cC_{t_0,t} := \mathrm{W}^*(S_s: s \in [t_0,t])$, such that
		\[
		\bE\bigg[\int_{t_0}^T L_{\cB * \cC} \Big(X_t[\widetilde{\alpha}], \alpha_t\Big)dt + g_{\cB * \cC}\Big(X_T[\widetilde{\alpha}]\Big)\bigg]  \leq \overline{V}_{\cB}(t_0,y_0) + \epsilon.
		\]
		Now let $\iota': \cA * \cC \to \cB * \cC$ be the tracial $\mathrm{W}^*$--embedding induced from the embedding $\iota: \cA \to \cB$ (Lemma \ref{lem: free product embedding}), and let $E': \cB * \cC \to \cA * \cC$ be the corresponding conditional expectation.  We want to define a control policy $\widetilde{\beta}$ in $\cA * \cC$ by
		\[
		\beta_t = E' \alpha_t,\qquad t \in [t_0,T],
		\]
		where the filtration is $(\overline \iota_1\circ \iota(\cA) \vee \overline \iota_2(\cC_{t_0,t}))_{t \in [t_0,T]}$ and the Brownian motion is the same $(\overline \iota_2(S_t))_{t \in [t_0,T]}$.  Measurability of $\beta_t$ follows from the fact that $E'$ is a contraction from $L^2(\cB * \cC)$ to $L^2(\cA * \cC)$, hence continuous.
		
		We next need to check that $\beta_t \in \overline \iota_1\circ \iota(\cA) \vee \overline \iota_2(\cC_{t_0,t})$.  Here note by associativity (Lemma \ref{lem: associativity}), we have a canonical isomorphism $\cB * \cC \cong \cB * \cC_{t_0,t} * \cC_{t,T}$, and this also restricts to an isomorphism $\cA * \cC \cong \cA * \cC_{t,t_0} * \cC_{t,T}$.  Let $\iota_t':\cA * \cC_{t_0,t} \to \cB * \cC_{t_0,t}$  be the inclusion and let $E_t': \cB * \cC_{t_0,t} \to \cA * \cC_{t_0,t}$ be the corresponding conditional expectation.  Let $\phi_t: \cB * \cC_{t_0,t} \to \cB * \cC$ be the canonical inclusion.  By Lemma \ref{lem: free product commuting square}, we have
		\[
		E' \circ \phi_t = \phi_t \circ E_t' \text{ on } \cB * \cC_{t_0,t}.
		\]
		Recall $\alpha_t \in \overline \iota_1(\cA) \vee \overline \iota_2(\cC_{t_0,t}) \subseteq \cB * \cC$, that is, $\alpha_t$ is in the image of $\phi_t$, and therefore the above identity implies that $E' \alpha_t$ is in the image of $\phi_t \circ E_t'$.  This means that $\beta_t$ is in the image of $\cA * \cC_{t_0,t}$ or $\beta_t \in \overline \iota_1 \circ \iota(\cA) \vee \overline \iota_2(\cC_{t_0,t})$, as desired.  Hence, $\tilde{\beta}$ is a valid control policy using the asserted filtration.
		
		Finally,   noting that $X_t[\widetilde{\beta}] = E' X_t[\widetilde{\alpha}]$, by $E$-convexity of $L$ and $g$ (Assumption \textbf{B}), we have
		\begin{align*}
			\overline{V}_{\cA}(t_0, E y_0) &\leq \bE\bigg[\int_{t_0}^T L_{\cA * \cC} \Big(X_t[\widetilde{\beta}], \beta_t\Big)\, dt + g_{\cA * \cC}\Big(X_T[\widetilde{\beta}]\Big)\bigg] \\
			&\leq \bE\bigg[\int_{t_0}^T L_{\cB * \cC} \Big(X_t[\widetilde{\alpha}], \alpha_t\Big)\, dt + g_{\cB * \cC}\Big(X_T[\widetilde{\alpha}]\Big)\bigg] \\
			&\leq \overline{V}_{\cB}(t_0,y_0) + \epsilon.
		\end{align*}
		Since $\epsilon>0$ was arbitrary, we have $\overline{V}_{\cA}(t_0,E y_0) \leq \overline{V}_{\cB}(t_0,y_0)$ as desired.
	\end{proof}

	
	

	\begin{proposition}\label{prop:continuity} 
		Suppose that Assumption \textbf{A} holds. For $\cA \in \bW$,  $(t,x)\mapsto \overline{V}_\cA(t,x)$ is continuous and bounded such that  for $(t_0,x_0)\in [0,T]\times L^2(\mathcal{A})_{sa}^d,$
		\begin{align}\label{eqn:global_bound}
			-C_1\, (1+T)\leq \overline{V}_\cA(t_0,x_0) \leq C  (\|x_0\|_{L^2(\cA)} + 1).
		\end{align}
		Also, it is Lipschitz in space and H\"{o}lder in time: 
		\begin{align}\label{eqn:Lipschitz_estimate}
			\big|\overline{V}_\cA(t_1,x_1) - \overline{V}_\cA(t_2,x_2)\Big| \leq C\, \|x_1-x_2\|_{L^2(\cA)} + C_{M_2}\, \sqrt{|t_1-t_2|}
		\end{align}
		for $(t_1,t_2,x_1,x_2)\in [0,T]\times[0,T]\times L^2(\mathcal{A})_{sa}^d\times L^2(\mathcal{A})_{sa}^d$, where $M_2:=\max\{\|x_1\|_{L^2(\cA)},\|x_2\|_{L^2(\cA)}\}$.
		
		Equivalently, for the value function $\overline{V}$ defined on the space of laws in \eqref{eqn:value}, we have 
		\begin{align}\label{eqn:free_global_bound}
			-C_1\, (1+T)\leq \overline{V}(t_0,\lambda_0) \leq C_{M_3}\quad  \hbox{ for }(t_0,\lambda_0)\in [0,T]\times \Sigma_d^2,
		\end{align}
		where $M_3$ is the second moment of $\lambda_0$,
		and
		\begin{align}\label{eqn:free_continuity estimate}
			\big|\overline{V}(t_1,\lambda_1) - \overline{V}(t_2,\lambda_2)\big| \leq C \, d_W(\lambda_1,\lambda_2) + C_{M_4}\,  \sqrt{|t_1-t_2|}
		\end{align}
		for $(t_0,\lambda_1,\lambda_2)\in [0,T]\times \Sigma_d^2\times \Sigma_d^2,$ where $M_4$ denotes the maximum of the second moments of $\lambda_1,\lambda_2$.
	\end{proposition}
	
	
	\begin{proof}
		We first prove a version of (\ref{eqn:global_bound}) for the function $\widetilde{V}_{\cA}(t_0,x_0).$
		The lower bound follows immediately from the lower bounds in (\ref{eqn:lower_and_upper_bounds}) with any control policy, and the upper bound in (\ref{eqn:global_bound}) follows from considering the trivial control $\alpha_t \equiv 0$. {{Indeed, the corresponding trajectory $(X_t)_t$ becomes
				\[
				X_t=x_0+ \beta_C \mathbbm{1}_{\cA}(W^0_t-W^0_{t_0}) +\beta_F(S_t-S_{t_0}).
				\]
				Thus,  for any $0\le  t \le T,$
				\[
				\bE[\|X_t\|_{L^2(\cA)}]  \leq \|x_0\|_{L^2(\cA)}+ (\beta_C+\beta_F) \sqrt{t}.
				\]
				By the upper bound condition on $\cL_\cA$ and $g_\cA$ in (\ref{eqn:lower_and_upper_bounds}), we obtain the upper bound  in (\ref{eqn:global_bound}) for $\widetilde{V}_{\cA}(t_0,x_0).$ Hence,  the estimate  (\ref{eqn:global_bound})  follows by taking the infimum over all tracial $W^*$-embeddings $\iota :\cA \rightarrow \cB,$ noting that  $\|x_0\|_{L^2(\cA)} = \|\iota x_0\|_{L^2(\cB)}$.
		}}

		~

		{Now, we prove Lipschitz continuity in the spatial variable. Let $x_1\in L^2(\cA)^d_{sa},$ $t_0\in [0,T]$ and $\epsilon>0$. 
			By Lemma \ref{lem: arranging initial filtration},   there exists a tracial $\mathrm{W}^*$-algebra $\cB$, a tracial $W^*$--embedding $\iota: \cA \to \cB$, and a control policy $\widetilde \alpha \in \mathbb{A}_{\mathcal{B},\iota x_1}^{t_0,T}$ associated with a filtration
			$(\cB_t)_{t \in [t_0,T]}$ and a compatible free Brownian motion $(S_t)_{t \in [t_0,T]}$, such that $\iota(\cA) \subseteq \cB_{t_0}$ and
			\begin{align} \label{333}
				\mathbb{E}\left[ \int_{t_0}^T L_{\mathcal{B}}(X_t[\widetilde 
				\alpha],\alpha_t)\,dt + g_{\mathcal{B}}(X_T[\alpha]) \right] \leq 
				\overline{V}_{\mathcal{A}}(t_0,x_1) + \epsilon,
			\end{align}
			where $({X}_t[\widetilde{\alpha}])_{t\in [t_0,T]}$ satisfies the initial condition $\widehat{X}_{t_0}[\widetilde{\alpha}]=\iota x_1$.  
			We remark, which will be useful later, that from  (\ref{eqn:global_bound}) and the lower bound (\ref{eqn:lower_and_upper_bounds}),   such   control policy $\widetilde  \alpha$ satisfies a priori bound
			\begin{align} \label{l2 bound}
				\bE\Big[\int_{t_0}^T \|\alpha_t\|_{L^2(\cB)}^2dt\Big] \leq   C( \|x_1\|_{L^2(\cA)}+ T+1) .
			\end{align}
			Note that $\widetilde \alpha \in \mathbb{A}_{\mathcal{B},\iota x_2}^{t_0,T}$, since $\iota x_2 \in \iota(A) \subseteq \cB_{t_0}.$
			Let $(\widehat{X}_t[\widetilde{\alpha}])_{t\in [t_0,T]}$ be a solution to \eqref{eqn:common_noise}  in $L^2(\cB)^d_{sa}$ with the same filtration
			$(\cB_t)_{t \in [t_0,T]}$ and free Brownian motion $(S_t)_{t \in [t_0,T]}$, satisfying the initial condition $\widehat{X}_{t_0}[\widetilde{\alpha}]=\iota x_2$. Then  using  \eqref{333},
			\begin{align*}
				&\ \overline V_{\cA}(t_0,x_2) -\overline V_{\cA}(t_0,x_1)\\
				&\leq \epsilon + \mathbb{E}\Big[\int_{t_0}^T \Big(L_{\cB} (\widehat{X}_t[\widetilde{\alpha}], \alpha_t)- L_{\cB} (X_t[\widetilde{\alpha}], \alpha_t)\Big)dt + g_{\cB}(\widehat{X}_T[\widetilde{\alpha}]) - g_{\cB}({X}_T[\widetilde{\alpha}])\Big]\\
				&\overset{\eqref{eqn:Lipschitz_bounds}}{\leq} \epsilon + \mathbb{E}\Big[\int_{t_0}^T C_2\, \|\widehat{X}_t[\widetilde{\alpha}] -X_t[\widetilde{\alpha}]\|_{L^2(\cB)}dt +C_2\, \|\widehat{X}_T[\widetilde{\alpha}] -X_T[\widetilde{\alpha}]\|_{L^2(\cB)} \Big]\\
				&\overset{\eqref{eqn:initial_condition_continuity_2}}{\le} \epsilon +   C_2\, (1 +T) \, \sqrt{\widetilde{C}}\, \|\iota x_2-\iota x_1\|_{L^2(\cB)} = \epsilon +   C_2\, (1 +T) \, \sqrt{\widetilde{C}}\, \|x_2-x_1\|_{L^2(\cA)}.
			\end{align*}
			As $\epsilon>0$ is arbitrary, we deduce a Lipschitz bound (\ref{eqn:Lipschitz_estimate}) in the spatial variable.
		}
		
		~

		{We now show H\"{o}lder continuity in time.  Take $t_1,t_2\in [0,T]$, $x_0\in L^2(\cA)_{sa}^d$ and $\epsilon>0$. We may assume $t_1<t_2$. By definition of $\overline V_\cA$ in \eqref{def:bar}, we take a tracial $\mathrm{W}^*$-algebra $\cB$, a tracial $W^*$--embedding $\iota: \cA \to \cB$, and a control policy $\widetilde  \alpha \in \mathbb{A}_{\mathcal{B},\iota x_0}^{t_1,T}$ associated with free Brownian motion $(S_t)_{t\in [t_1,T]}$ and a  filtration
			$(\cB_t)_{t \in [t_1,T]}$  such that 
			\begin{align} \label{334}
				\mathbb{E}\left[ \int_{t_1}^T L_{\mathcal{B}}(X_t[\alpha],\alpha_t)\,dt + g_{\mathcal{B}}(X_T[\alpha]) \right] \leq 
				\overline{V}_{\mathcal{A}}(t_1,x_0) + \epsilon.
			\end{align}
			Set $N:=\mathbb{E}\big[\int_{t_1}^{T}\|\alpha_t\|^2_{L^2(\cB)}dt\big]
			$, which is bounded independent of $t_1$ and $t_2$, due to \eqref{l2 bound}.  We define $\widetilde{\alpha}' $ to be a  restriction of $\widetilde{\alpha} $ to $[t_2,T]$, associated with  free Brownian motion $(S_t - S_{t_2})_{t\in [t_2,T]}$ and  a filtration $(\cB_t)_{t\in [t_2,T]}$. Note that $\widetilde \alpha' \in \mathbb{A}_{\mathcal{B},\iota x_0}^{t_2,T}$, since $\iota x_0 \in \cB_{t_1} \subseteq \cB_{t_2}  $.
			Let $(\widehat{X}_t[\widetilde{\alpha}'])_{t\in [t_2,T]}$ be the solution to \eqref{eqn:common_noise} in $L^2(\cB)^d_{sa}$, satisfying the initial condition  $\widehat{X}_{t_2}[\widetilde{\alpha}']=\iota  x_0$.
			By (\ref{eqn:time_bound_L2}) in  Theorem \ref{thm:well_posedness}, setting $M:=\|\iota x_0\|_{L^2(\cB)} = \|x_0\|_{L^2(\cA)}$, 
			$$
			\mathbb{E}\Big[\|\widehat{X}_{t_2}[\widetilde{\alpha}']-X_{t_2}[\widetilde{\alpha}]\|^2_{L^2(\cB)}\Big] = \mathbb{E}\Big[\| X_{t_2}[\widetilde{\alpha}] - \iota x_0\|^2_{L^2(\cB)}\Big] \leq \widetilde{C}_{M,N}\, (t_2-t_1).
			$$
			Hence for $t\in [t_2,T]$, it follows from (\ref{eqn:initial_condition_continuity_2}) that
			$$
			\mathbb{E}\Big[\|\widehat{X}_{t}[\widetilde{\alpha}']-X_{t}[\widetilde{\alpha}]\|^2_{L^2(\cB)}\Big]\leq \widetilde{C}\, \mathbb{E}\Big[\|\widehat{X}_{t_2}[\widetilde{\alpha}']-X_{t_2}[\widetilde{\alpha}]\|^2_{L^2(\cB)}\Big] \le \widetilde{C}\widetilde{C}_{M,N}\, (t_2-t_1). 
			$$
			Thus, using this along with   the assumptions on $L_\cB$ and $g_\cB,$
			\begin{align*}
				&\ \overline {V}_\cA(t_2,x_0) -\overline 
				{V}_\cA(t_1,x_0)\\
				&\overset{\eqref{334}}{\le} \epsilon + \bE\Big[
				-\int_{t_1}^{t_2}L_{\cB} (X_t[\widetilde{\alpha}], \alpha_t)dt + \int_{t_2}^T \Big(L_{\cB} (\widehat{X}_t[\widetilde{\alpha}'], \alpha_t')- L_{\cB} (X_t[\widetilde{\alpha}], \alpha_t)\Big)dt \\
				&  \qquad + g_{\cB}(\widehat{X}_T[\widetilde{\alpha}']) - g_{\cB}({X}_T[\widetilde{\alpha}])\Big]\\
				&\leq \epsilon + 2\, C_1\, (t_2-t_1)  + C_2\, (1+T)\, \sqrt{\widetilde{C}\, \widetilde{C}_{M,N}}\, \sqrt{t_2-t_1}.
			\end{align*}
			As $\epsilon>0$ is arbitrary, 
			we deduce a Lipschitz bound (\ref{eqn:Lipschitz_estimate}) in time variable.
		}
		
		~

		Finally, (\ref{eqn:free_global_bound}) and (\ref{eqn:free_continuity estimate}) follow immediately from the same bounds by selecting representatives of $\lambda_1,\lambda_2$ in a von Neumann algebra.
	\end{proof}

	\section{Viscosity Solutions} \label{sec: viscosity solutions}
	
	In this section, we develop a new theory of viscosity solutions on the space of  non-commutative laws.  In particular, we show that the value function for the stochastic optimal control problem described in the previous section is a viscosity solution.  Although we do not establish a comparison principle on the space of non-commutative laws, we can do this in the case where there is no common noise by relating our problem to already developed theory of viscosity solutions on Hilbert space.
	
	The notion of viscosity solution was invented to understand non-smooth solutions of first and second-order elliptic and parabolic PDE.  The basic idea is that whenever the solution $u$ has a Taylor approximation from above or from below (or equivalently can be touched from above or below by a smooth test function), then substituting the first and second order terms from the Taylor expansion in place of the gradient and Hessian in the differential equation will produce an inequality in one direction.  The theory of viscosity solutions was first developed in \cite{vis1,vis2,vis3} and then has been extended to infinite-dimensional Hilbert space \cite{visinf1, crandall1990viscosity, ishii1993viscosity}.  The first author and Tudorascu adapted the theory of viscosity solutoins to Hamilton-Jacobi equation on the Wasserstein space \cite{gangbo2019}.  The notion of viscosity solution was applied in the random matrix setting in \cite{jekel2020elementary}, which is one motivation for the present work.
	
	We consider the following equation on  the space of non-commutative laws $\Sigma_{d}^2$:
	\begin{align}\label{eqn:non_commutative}
		-\partial_t V(t,\lambda) + H\big(\lambda, -\partial V(t,\lambda)\big) - \frac{\beta_C^2}{2}\, \Delta V(t,\lambda)- \frac{\beta_F^2}{2}\, \Theta V(t,\lambda)=&\ 0,\\
		V(T,\lambda) =&\ g(\lambda). \nonumber
	\end{align}
	In Definition \ref{def:WassSpaceViscosity}, we define viscosity solutions for this equation through the related collection of equations for each von Neumann algebra $\cA\in \bW$,
	\begin{align}\label{eqn:von_Neumann}
		-\partial_t V_\cA(t,X) +H_{\cA}\big(X, -\nabla V_\cA(t,X)\big) - \frac{\beta_C^2}{2}\, \Delta_{\cA} V(t,X)- \frac{\beta_F^2}{2}\, \Theta_{\cA} V(t,X)=&\ 0,\\
		V_\cA(T,X)=&\ g_{\cA}(X), \nonumber
	\end{align}
	where the common noise Laplacian $\Delta_\cA$ and the free individual noise Laplacian $\Theta_\cA$ will be defined below.  In fact, we will take a supremum of the left-hand side over the von Neumann algebras, leading to some asymmetry between the arguments for subsolutions and supersolutions.  Note that this is a parabolic equation expressed backwards in time, with the terminal condition at time $T$.
	
	Given a function $V:[0,T]\times \Sigma_{d}^2 \rightarrow \bR$, for each $\cA\in \bW$ we define
	\begin{align} \label{defv}
		V_{\cA}(t,X) := V(t,\lambda_X),\quad  X\in L^2(\cA)_{sa}^d \text{ and } t \in [0,T].
	\end{align}
	This is consistent with our notation for the value function $(\overline{V}_{\cA})_{\cA\in \bW}$ and $\overline V$ in \eqref{eqn:value}, due to Lemma \ref{lem:nov18.2023.2}.
	
	
	For every $\cA \in \bW$, viscosity subsolutions and supersolutions to (\ref{eqn:von_Neumann}) can be defined in a standard way on the Hilbert space $L^2(\cA)^d_{sa}$.
	However, as will be seen later, this is the case only if we assume that free individual noise is \emph{not} present, since the operator $\Theta$ is not defined on every $L^2(\cA)^d_{sa}$. The standard class of test functions on the Hilbert space  $L^2(\cA)^d_{sa}$ is given by 
	$$
	\mathcal{X}_{\cA} := \Big\{\phi\in C^1([0,T]\times L^2(\cA)_{sa}^d): \nabla^2 \phi\in C([0,T]\times L^2(\cA)_{sa}^d; \text{BL}(L^2(\cA)_{sa}^d)) \Big\},
	$$
	{where  BL denotes the space of bounded linear operators.} Examples of $\Phi\in \mathcal{X}_\cA$ include, $\Phi(x) = \|x-x_0\|_{L^2(\cA)}^2$ for $x_0\in L^2(\cA)_{sa}^d$ fixed, and $\Phi(x) = g(P_N\, x)$ where $P_N$ is a projection onto a $N$-dimensional subspace and $g$ is a smooth function.
	
	It is not so straightforward to choose a set of smooth test functions on $\Sigma_d^2$ that are  tracial $W^*$--functions. For first order equations, one could get around this by expressing the notion of viscosity solutions using subdifferentials on the Wasserstein space, which is done in the commutative setting in \cite{gangbo2019}, but this does not immediately help with second-order equations.  Instead, in our setting of non-commutative laws, following the idea of an $L$-derivative \cite{cardaliaguet2019master}, we define the collection $\mathcal{X}_\Sigma$ of  admissible test functions $ (U_{\cA})_{\cA \in \bW}$ on $[0,T]\times \Sigma_d^2$ that satisfy the following properties:
	\begin{enumerate}[label=(\alph*)]
		\item For any $\cA\in \bW$, $U_{\cA}:[0,T]\times L^2(\cA)_{sa}^d \rightarrow \mathbb{R}$ and for each $t\in [0,T]$, $(U_{\cA}(t,\cdot))_{\cA\in \bW}$ is a tracial $W^*$--function.
		\item For any $\cA\in \bW$, $U_\cA\in C^{1,1}([0,T]\times L^2(\cA)^d_{sa})$ (i.e. first derivatives, both in time and spatial variables, are Lipschitz).
		\item For any $\cA\in \bW$,  $X,A\in L^2(\cA)_{sa}^d$,  $B\in L^\infty(\cA)_{sa}^d$ and $t\in [0,T]$, the partial second derivatives ${\rm Hess}\, U_{\cA}(t,X)[A,B]$ exists and $t\mapsto {\rm Hess}\, U_{\cA}(t,X)[A,B]$ is continuous.
		\item There is a constant $K>0$ such that, with $\cA,X,A,B,t$ as in (c) and $Y\in L^2(\cA)_{sa}^d$,
		\begin{align} \label{hess}
			\big|\text{Hess}\, U_{\cA}(t,X)[A,B] - \text{Hess}\, U_{\cA}(t,Y)[A,B] \big|\leq K\,  \|X-Y\|_{L^2(\cA)}\, \|A\|_{L^2(\cA)}\, \|B\|_{L^\infty(\cA)}.
		\end{align}
	\end{enumerate}

	On the space $\mathcal{X}_\Sigma$, we define the infinite-dimensional operators that will be part of the Hamilton-Jacobi-Bellman equation. For  $(U_{\cA})_{\cA\in \bW}\in \mathcal{X}_\Sigma$ or $U_{\cA}\in \mathcal{X}_\cA$,
	the \emph{common noise Laplacian} is defined to be
	$$
	\Delta_{\cA} U(t,X) : =  {\rm Hess}\, U_\cA(t,X)\big[\mathbbm{1}_{\cA}, \mathbbm{1}_{\cA}\big],\quad  X\in L^2(\cA)_{sa}^d,  
	$$
	where we recall that $\mathbbm{1}_{\cA}\in L^2(\cA)_{sa}^d$ has the algebra unit in every component.
	
	When defining the free individual noise Laplacian, it is not sufficient to consider a function $U_\cA$ defined on a single von Neumann algebra $\cA$, but instead we should use a tracial $\mathrm{W}^*$-function $(U_{\cA})_{\cA\in \bW}$. This is because not all von Neumann algebras support freely independent semicircle laws.  For $\cA\in \bW$, we consider a tracial $W^*$-embedding $\iota: \cA \to \cB$ where $\cB$ contains a $d$-dimensional semicircle element $S = (S^1,\cdots, S^d)$ that is freely independent of $\iota(\cA)$.  Then, for $(U_{\cA})_{\cA\in \bW}\in \mathcal{X}_\Sigma$, the \emph{free individual noise Laplacian} is defined to be
	\begin{align}\label{eqn:free_laplacian}
		\Theta_{\cA}\, U(t,X) :=&\  \sum_{l=1}^{d}{\rm Hess}\, U_{\mathcal{B}}(t,\iota\, X)\big[S^l \mathbf{e}_{\cA}^l, S^l \mathbf{e}_{\cA}^l\big], \quad X\in L^2(\cA)_{sa}^d, 
	\end{align}
	where $\mathbf{e}_{\cA}^l\in L^2(\cA)_{sa}^d$ has the algebra unit in the $l$-th component and zero in the other components.
	
	This definition is motivated by the classical fact that the Laplacian of a function on $\bR^d$ can be expressed as
	\[
	\Delta u(x) =  \mathbb{E}[\ip{\text{Hess} u(x)Z,Z}],
	\]
	where $Z$ is a standard Gaussian random vector in $\bR^d$.  A similar approach is used in \cite[\S 4.3]{jekel2022tracial} and \cite{jekel2023martingale}, while many previous works gave a more explicit definition of the Laplacian for non-commutative polynomials, power series, and the like, in terms of non-commutative derivative operations and traces; see Appendix \ref{apx:free_laplacian} for more detail.  We point out that such explicit computations are not necessarily possible for tracial $\mathrm{W}^*$-functions in general since a function might be smooth with respect to the Wasserstein distance but \emph{not} continuous with respect to (weak-$*$) convergence in non-commutative law, and hence unable to be approximated by trace polynomials.  Hence, it is necessary for us to define the Laplacian directly in terms of free semi-circulars.
	
	\begin{remark} ~
		\begin{enumerate}[label=\roman*.]
			\item Note that for  $(U_{\cA})_{\cA\in \bW}\in \mathcal{X}_\Sigma$, \eqref{eqn:free_laplacian} does \emph{not} depend on the choice of von Neumann algebra $\cB$ or the $\mathrm{W}^*$-embedding $\iota: \cA \to \cB$. {Indeed, the common noise Laplacian and free individual noise Laplacian are tracial $\mathrm{W}^*$-functions.}
			\item With the tuple from the definition \eqref{eqn:free_laplacian}, $\iota\, X+ \sqrt{t-s} \, \beta_F\, S$ has the same non-commutative law as $\iota X+ \beta_F\, (S_{t}- S_{s})$ where  $S_t-S_s$ is freely independent from $\iota X$. This holds because $S_t-S_s$ has the same law as $\sqrt{t-s} {S}$, and both are freely independent of $\iota X$. This fact will be used in the proof of the mixed-It\^{o} formula, Lemma \ref{lem:Ito}.
		\end{enumerate}
	\end{remark}

	\subsection{Intrinsic Viscosity Solution} \label{subsec: intrinsic viscosity}
	
	For a metric space $\mathcal{Y}$, let $\textup{USC}(\mathcal{Y})$ denote the upper-semicontinuous and bounded above functions, and $\textup{LSC}(\mathcal{Y})$ denote the lower-semicontinuous and bounded below functions. We say that a function $\Phi\in \textup{LSC}(\mathcal{Y})$ \emph{touches} a function $U\in \textup{USC}(\mathcal{Y})$ \emph{from above (below)} at $y\in \mathcal{Y}$ if $\Phi(y)=U(y)$ and $\Phi(y')\geq U(y')$ (resp.\ $\Phi(y')\leq U(y')$) for all $y'\in \mathcal{Y}$.
	
	\begin{definition}\label{def:WassSpaceViscosity} Suppose that $(U_{\cA})_{\cA\in \bW}\in \textup{USC}([0,T]\times \Sigma_d^2)$ (resp.\ $\textup{LSC}([0,T]\times \Sigma_d^2)$). We say that $(U_{\cA})_{\cA\in \bW}$ is a \emph{free viscosity sub(super)solution} of (\ref{eqn:non_commutative}) if 
		\begin{enumerate}
			\item $U_{\cA}(T,x)\leq (\geq) g_{\cA}(x)$ for any $\cA\in \bW$ and $x\in L^2(\cA)_{sa}^d$.
			\item Whenever $(\Phi_{\cA})_{\cA\in \bW}\in \mathcal{X}_\Sigma$ touches $(U_{\cA})_{\cA\in \bW}$ from above (below) at $(t_0,\lambda_0)\in [0,T)\times \Sigma_d^2$,
			\begin{align*}
				\sup_{\cA\in \bW, x_0\in L^2(\cA)_{sa}^d,\lambda_{x_0}=\lambda_0}\Big\{&-\partial_t \Phi_{\cA}(t_0,x_0) + H_{\cA}\big(x_0, -\nabla \Phi_\cA(t_0,x_0)\big)\\
				&\ - \frac{\beta_C^2}{2}\, \Delta_{\cA} \Phi(t_0,x_0)- \frac{\beta_F^2}{2}\, \Theta_{\cA} \Phi(t_0,x_0)\Big\}\leq (\geq) 0.
			\end{align*}
		\end{enumerate}
	\end{definition}

	We say that  a continuous function $(U_{\cA})_{\cA\in \bW} $ on $[0,T]\times \Sigma_d^2$ is a 
	\emph{free viscosity solution} if it is a free viscosity subsolution \emph{and} supersolution.

	~

	We first establish an It\^{o} formula, which will then be used to establish a dynamic programming principle. The dynamic programming principle is then used to prove that the value function, defined in \eqref{eqn:value}, is a free viscosity subsolution.
	
	\subsubsection{Mixed-It\^{o} formula}
	
	We essentially need two versions of Taylor's theorem.
	The following lemma provides the first-order Taylor's theorem.
	\begin{lemma} For any $(U_{\cA})_{\cA \in \bW} \in \mathcal{X}_\Sigma$, $X,Y\in L^2(\cA)_{sa}^d$ and $t\in [0,T]$,
		$$
		\big|U_{\cA}(t,Y)-U_{\cA}(t,X) -\langle \nabla U_{\cA}(t,X),Y-X\rangle_{L^2(\cA)}\big|\leq \|U_{\cA}\|_{C^{1,1}}\, \|Y-X\|_{L^2(\cA)}^2.
		$$
	\end{lemma}
	
	\begin{proof}
		This follows directly from $U_{\cA}\in C^{1,1}([0,T]\times L^2(\cA)_{sa}^d)$.
	\end{proof}
	
	The next lemma provides the second order Taylor's theorem.
	\begin{lemma}  \label{second taylor}
		For any $(U_{\cA})_{\cA \in \bW} \in \mathcal{X}_\Sigma$, $t\in [0,T]$ and $X,Y\in L^2(\cA)_{sa}^d$ with $X-Y\in L^\infty(\cA)^d_{sa}$, 
		\begin{align*}
			&\ \big|U_{\cA}(t,Y)-U_{\cA}(t,X) -\langle \nabla U_{\cA}(t,X),Y-X\rangle_{L^2(\cA)}-\frac{1}{2}{\rm Hess}\, U_{\cA}(t,X)[Y-X, Y-X]\big|\\
			\leq&\ \frac{K}{2} \|Y-X\|_{L^2(\cA)}\, \|Y-X\|_{L^2(\cA)}\,  \|Y-X\|_{L^\infty(\cA)},
		\end{align*}
		where $K>0$ is a constant from \eqref{hess}.
	\end{lemma}
	
	\begin{proof}
		Fix $X\in L^2(\cA)_{sa}^d$, $t\in [0,T]$  and let
		$$
		\phi(H): = U_{\cA}(t,X+H) - U_{\cA}(X)- \langle \nabla U_{\cA}(t,X),H\rangle_{L^2(\cA)} - \frac{1}{2}{\rm Hess}\, U_{\cA}(t,X)[H,H].
		$$
		We have that, for $H\in L^2(\cA)_{sa}^d$ and  $B\in L^\infty(\cA)_{sa}^d$, 
		\begin{align*}
			\langle \nabla \phi(H), B\rangle_{L^2(\cA)} =&\ \langle\nabla U_{\cA}(t,X+H) - \nabla U_{\cA}(t,X), B\rangle_{L^2(\cA)} - {\rm Hess}\, U_{\cA}(t,X)[H, B]\\
			=&\ \int_0^1 \Big( {\rm Hess}\, U_{\cA}(t,X + sH)[H,B] - {\rm Hess}\, U_{\cA}(t,X)[H, B]\Big) ds.
		\end{align*}
		By  the condition \eqref{hess},
		$$
		\big|\langle\nabla \phi(H), B\rangle_{L^2(\cA)}\big| \leq K\, \|H\|_{L^2(\cA)}\, \|H\|_{L^2(\cA)}\, \|B\|_{L^\infty(\cA)}.
		$$
		It follows that
		\begin{align*}
			\phi(H)-\phi(0) =&\ \int_0^1 \langle \nabla\phi(t\, H), H\rangle_{L^2(\cA)}\, dt\\
			\leq&\ \frac{K}{2}\|H\|_{L^2(\cA)}\, \|H\|_{L^2(\cA)}\, \|H\|_{L^\infty(\cA)},
		\end{align*}
		yielding  the second order Taylor's theorem. 
	\end{proof}
	~
	
	We are now ready to state a free version of the mixed-It\^o formula.

	\begin{lemma}[Mixed It{\^o} formula] \label{lem:Ito}
		Suppose that Assumption \textbf{A} holds.    For  $\cA\in \bW$, $t_0\in [0,T)$, $x_0\in L^2(\cA)_{sa}^d$, and   $\widetilde{\alpha}\in {\mathbb{A}}_{\cA,x_0}^{t_0,T}$, let $(X_t)_{t\in [t_0,T]}$ be the corresponding solution to \eqref{eqn:common_noise}.  Then for any $(U_\cA)_{\cA\in \bW}\in \mathcal{X}_\Sigma$ and $t\in [t_0, T]$,
		\begin{align}\label{eqn:Ito}
			\mathbb{E}\big[U_{\cA}(t,X_t)\big] =&\ U_{\cA}(t_0,x_0)+\mathbb{E}\Big[\int_{t_0}^t\Big(\partial_t U_{\cA}(s,X_s) +  \big\langle \nabla U_{\cA}(s,X_s),b_{\cA}(X_s,\alpha_s)\big\rangle_{L^2(\cA)} \\
			&\ +  \frac{\beta_C^2}{2}\Delta_{\cA} U(s,X_s) + \frac{\beta_F^2}{2}\Theta_{\cA} U(s,X_s)\Big) ds\Big].\nonumber
		\end{align}
	\end{lemma}
	\begin{proof}

		Fix $U = (U_\cA)_{\cA\in \bW}\in \mathcal{X}_\Sigma$, $t\in [t_0,T]$ and $\delta>0$. Then we take a partition $t_0<t_1<\ldots<t_N=t$ with $t_{i+1}-t_i<\delta$ for $i\in \{0,\ldots, N-1\}$. 
		We define 
		\begin{align*}
			Y_{t_{i}}^1 &:= X_{t_i} + \beta_C\mathbbm{1}_{\cA}\big(W_{t_{i+1}}^0-W_{t_i}^0\big),  \\
			Y^2_{t_{i}} &:=  Y_{t_{i}}^1 + \beta_F \big(S_{t_{i+1}} - S_{t_i}\big),
		\end{align*}
		and for $r\in [t_i,t_{i+1})$,
		$$
		Y^3_r = Y_{t_{i}}^2 + \int_{t_i}^r b_{\cA}(X_s,\alpha_s)ds.
		$$
		Then,
		\begin{align} \label{211}
			\mathbb{E}\Big[\|Y^3_r - X_r\|^2_{L^2(\cA)}\Big] &=   \mathbb{E}\Big[\|\beta_C\mathbbm{1}_{\cA}\big(W_{t_{i+1}}^0-W_{r}^0\big) +  \beta_F \big(S_{t_{i+1}} - S_{}\big)  \|^2_{L^2(\cA)}\Big] \nonumber   \\
			&\leq d\, \delta\, (\beta_F^2+\beta_C^2).
		\end{align}
		We write 
		\begin{align}  \label{several}
			U_{\cA}(t_{i+1},X_{t_{i+1}}) &- U_{\cA}(t_{i},X_{t_i})  \nonumber\\
			&= \Big(U_{\cA}(t_{i+1},X_{t_{i+1}}) - U_{\cA}(t_{i},Y^2_{t_{i}})\Big)\nonumber\\
			&+ \Big(U_{\cA}(t_{i},Y^2_{t_{i}}) - U_{\cA}(t_{i},Y^1_{t_{i}})\Big)+ \Big(U_{\cA}(t_{i},Y^1_{t_{i}}) - U_{\cA}(t_{i},X_{t_i})\Big).
		\end{align}
		To control the first term of RHS above, by the fundamental theorem of Calculus and the chain rule,
		\begin{align*}
			U_{\cA}(t_{i+1},X_{t_{i+1}}) - U_{\cA}(t_{i},Y^2_{t_{i}})  =&\ \int_{t_i}^{t_{i+1}}\Big(\partial_t U_{\cA}(s,Y^3_s) + \big\langle \nabla U_{\cA}(s, Y_s^3), b_{\cA}(X_s,\alpha_s)\big\rangle_{L^2(\cA)}\Big)ds.
		\end{align*}
		Note that we have, for any $\epsilon>0$,
		\begin{align*}
			\mathbb{E}\Big[\big|\langle\nabla U_{\cA}(s,Y_s^3),&  b_{\cA}(X_s,\alpha_s)\rangle_{L^2(\cA)} - \langle \nabla U_{\cA}(s,X_s), b_{\cA}(X_s,\alpha_s)\rangle_{L^2(\cA)}\big|\Big]\\
			&\leq\ \|U_{\cA}\|_{C^{1,1}}\, \sqrt{\mathbb{E}\big[\| Y^3_s-X_s\|_{L^2(\cA)}^2\big]}\,  \sqrt{\mathbb{E}\big[\| b_{\cA}(X_s,\alpha_s)\|_{L^2(\cA)}^2\big]}\\
			&\overset{\eqref{211}}{\le}\ \|U_{\cA}\|_{C^{1,1}}\, \sqrt{d\, \delta\, (\beta_F^2+\beta_C^2)}\big(\frac{1}{2\epsilon} + \frac{\epsilon}{2} \mathbb{E}\big[\|b_{\cA}(X_s,\alpha_s)\|_{L^2(\cA)}^2\big]\big).
		\end{align*}
		and
		\begin{align*}
			\mathbb{E}\Big[\big|\partial_t U_{\cA}(s,Y^3_s) - \partial_t U_{\cA}(s,X_s)\big|\Big]
			&\leq  \|U_{\cA}\|_{C^{1,1}}\, \sqrt{\mathbb{E}\big[\| Y^3_s-X_s\|_{L^2(\cA)}^2\big]}\\
			&\overset{\eqref{211}}{\le}  \|U_{\cA}\|_{C^{1,1}}\, \sqrt{d\, \delta\, (\beta_F^2+\beta_C^2)}.
		\end{align*}
		Hence  for any $i\in \{0,\ldots, N-1\}$
		\begin{align*}
			&\ \Big|\mathbb{E}\Big[U_{\cA}(t_{i+1}, X_{t_{i+1}}) - U_{\cA}(t_i,Y_{t_i}^2) - \int_{t_i}^{t_{i+1}}\Big(\partial_t U_{\cA}(s,X_s) + \big\langle \nabla U_{\cA}(s, X_s), b_{\cA}(X_s,\alpha_s)\big\rangle_{L^2(\cA)}\Big)ds\Big]\Big|\\
			\leq&\  \|U_{\cA}\|_{C^{1,1}}\, \sqrt{d\, \delta\, (\beta_F^2+\beta_C^2)}\Big( (t_{i+1}-t_i)\Big(1+\frac{1}{2\epsilon}\Big) + \frac{\epsilon}{2} \int_{t_i}^{t_{i+1}}\mathbb{E}\big[\|b_{\cA}(X_s,\alpha_s)\|_{L^2(\cA)}^2\big]ds\Big).
		\end{align*}
		
		To control the last term of RHS in \eqref{several}, by the second order Taylor's theorem (Lemma \ref{second taylor}), 
		\begin{align*}
			&\ \Big|\mathbb{E}\Big[U_{\cA}(t_{i}, Y^1_{t_{i}})-U_{\cA}(t_{i},X_{t_i}) -\langle \nabla U_{\cA}(t_{i},X_{t_i}),Y_{t_{i}}^1-X_{t_i}\rangle_{L^2(\cA)}\\
			&\ -\frac{1}{2}{\rm Hess}\, U_{\cA}(t_{i},X_{t_i})\big[Y_{t_{i}}^1-X_{t_i}, Y_{t_{i}}^1-X_{t_i}\big]\Big]\Big|\\
			& \qquad \qquad \leq \mathbb{E}\Big[\frac{K}{2} \|Y_{t_{i}}^1-X_{t_i}\|_{L^2(\cA)}\, \|Y_{t_{i}}^1-X_{t_i}\|_{L^2(\cA)}\, \|Y_{t_{i}}^1-X_{t_i}\|_{L^\infty(\cA)}\Big]\\
			& \qquad \qquad  \leq K\, (d\, \beta_C^2)^{\frac{3}{2}}\, \sqrt{\delta}\, (t_{i+1}-t_i).
		\end{align*}
		By independence of $W_{t_{i+1}}^0-W_{t_i}^0$ and $\mathcal{F}_{t_i}$,
		$$
		\mathbb{E}\big[\langle \nabla U_{\cA}(t_{i},X_{t_{i}}),Y_{t_{i}}^1-X_{t_i}\rangle_{L^2(\cA)}\big]=0,
		$$
		and
		$$
		\mathbb{E}\Big[\frac{1}{2}{\rm Hess}\, U_{\cA}(t_{i},X_{t_i})\big[Y_{t_{i}}^1-X_{t_i}, Y_{t_{i}}^1-X_{t_i}\big]\Big] = (t_{i+1}-t_i)\mathbb{E}\Big[\frac{\beta_C^2}{2}\Delta_{\cA} U(t_{i},X_{t_i})\Big].
		$$
		In addition, using   \eqref{hess} and (\ref{eqn:time_bound_L2}) from Theorem \ref{thm:well_posedness} with the same definition of $M$ and $N$, for $s \in [t_i, t_{i+1}]$,
		$$
		\Big|\mathbb{E}\Big[\frac{\beta_C^2}{2}\Delta_{\cA} U(t_{i},X_{t_i})\Big] - \mathbb{E}\Big[\frac{\beta_C^2}{2}\Delta_{\cA} U(s,X_{s})\Big]\Big| \leq \frac{K\, \beta_C^2}{2}\,  \sqrt{\tilde{C}_{M,N}\, \delta}.
		$$
		Thus, 
		\begin{align*}
			\Big|\mathbb{E}\Big[U_{\cA}(t_{i}, Y^1_{t_{i}})-U_{\cA}(t_{i},X_{t_i})  &- \int_{t_i}^{t_{i+1}}\frac{\beta_C^2}{2}  \Delta_{\cA} U(s,X_{s})ds \Big]\Big| \\
			&\leq  \big(K(d\, \beta_C^2)^{\frac{3}{2}} + \frac{K\, \beta_C^2}{2}\,  \sqrt{\tilde{C}_{M,N}}\big)\, \sqrt{\delta}\,   (t_{i+1}-t_i).
		\end{align*}
		
		To control the middle term of RHS in \eqref{several},  
		let $\cB\in \bW$ and  $\iota:\cA\rightarrow \cB$ be a tracial $W^*$-embedding such that $\cB$ contains a $d$-variable semi-circular element ${S} = (S^1,\cdots,S^d)$ freely independent of $\iota ( \cA)$. The tuples $Y_{t_i}^2$ and $\iota Y_{t_i}^1 + \sqrt{t_{i+1}-t_i}\, \beta_F\, {S}$ have the same non-commutative law, so, by the tracial property of $(U_\cA)_{\cA \in \bW}$, 
		$$
		U_{\cA}(t_{i}, Y^2_{t_{i}}) - U_{\cA}(t_{i}, Y^1_{t_{i}}) = U_{\cB}(t_{i}, \iota\, Y^1_{t_{i}} + \sqrt{t_{i+1} - t_i}\, \beta_F\, S) - U_{\cB}(t_{i}, \iota\, Y^1_{t_{i}}).
		$$
		By free independence of $S$ and $\iota\, Y^1_{t_i}$,
		$$
		\langle \nabla  U_{\cB}(t_{i}, \iota\, Y^1_{t_{i}}), \beta_F\, S\rangle_{L^2(\cB)} = 0.
		$$
		Thus, using the second order Taylor's theorem  along with the fact  $\|S\|_{L^\infty(\cA)} = 2\|S\|_{L^2(\cA)} = 2$,  
		\begin{align*}
			\Big|U_{\cB}(t_{i}, \iota\, Y^1_{t_{i}} + \sqrt{t_{i+1} - t_i}\, \beta_F\, S) - U_{\cB}(t_{i}, \iota\, Y^1_{t_{i}}) - \frac{t_{i+1}-t_i}{2} &{\rm Hess}\, U_{\cB}(t_{i}, \iota\, Y^1_{t_{i}})[\beta_FS, \beta_FS] \Big|\\
			\leq&\ K (d\, \beta_F^2)^{\frac{3}{2}}\, \sqrt{\delta} \, (t_{i+1}-t_i).
		\end{align*}
		Note that, for $j_1\not= j_2$  in $\{1,\cdots,d\}$,
		$$
		{\rm Hess}\, U_{\cB}(t_{i}, \iota\, Y^1_{t_{i}})[\beta_FS^{j_1}\, \mathbf{e}^{j_1}_{\cA}, \beta_FS^{j_2}\, \mathbf{e}^{j_2}_{\cA}] = 0,
		$$
		by free independence of $S^{j_1},S^{j_2}$ and $\iota\, Y^1_{t_i}$ along with the tracial property of $(U_\cA)_{\cA \in \bW}$. Thus, again using \eqref{hess} and  (\ref{eqn:time_bound_L2}) from Theorem \ref{thm:well_posedness},
		\begin{align*}
			\mathbb{E}\Big[\big| \frac{t_{i+1}-t_i}{2}{\rm Hess}\, & U_{\cB}(t_{i}, \iota\, Y^1_{t_{i}})[\beta_FS, \beta_FS] - \int_{t_i}^{t_{i+1}} \frac{\beta_F^2}{2}\Theta_{\cA}\, U(s, X_s)ds\big|\Big]\\
			\leq&\  \frac{K\, \beta_F^2}{2}\, \sqrt{\tilde{C}_{M,N}}\,  \sqrt{\delta}\, (t_{i+1} - t_i).
		\end{align*}
		
		Putting the above estimates together and taking the summation over $i=0,\cdots,N-1$, the difference of LHS and RHS in (\ref{eqn:Ito}) is bounded by
		\begin{align*}
			C\sqrt{\delta}\Big(\Big(1+\frac{1}{\epsilon}\Big)(t-t_0) + \epsilon\, \int_{t_0}^t \mathbb{E}\big[\|b_\cA(X_s,\alpha_s)\|_{L^2(\cA)}^2\big]ds\Big),
		\end{align*}
		where $C$ is a  constant that does not depend on $\delta$ or $\epsilon$.
		Using bounds from Theorem \ref{thm:well_posedness} and the bounds on $b_{\cA}$,
		$$
		\mathbb{E}\big[\|b_\cA(X_s,\alpha_s)\|_{L^2(\cA)}^2\big] \leq C,
		$$
		where $C$ only depends on $M=\|x_0\|_{L^2(\cA)}$ and 
		$
		N = \mathbb{E}\Big[\int_{t_0}^T \|\alpha_s\|_{L^2(\cA)}^2ds\Big].
		$
		We can then set $\epsilon : = \delta^{\frac{1}{4}}$ and sending $\delta\rightarrow^+0$ proves the theorem.
	\end{proof}

	\subsubsection{Dynamic Programming Principle} 
	We now show that the solution to  \eqref{eqn:common_noise} satisfies a Markov property analogous to the classical result in \cite[Lemma 3.2]{yong1999stochastic}.
	
	\begin{lemma}\label{lem:Markov_property}
		Consider $\cA\in \bW$, $x_0\in L^2(\cA)_{sa}^d$ and $\widetilde{\alpha}\in \mathbb{A}_{\cA, x_0}^{t_0,T}$. Then for any $0\leq t_0\leq t_1\leq T$, there exist  a collection of control policies $\widetilde{\alpha}^{\bar{\omega}}\in \mathbb{A}_{\cA, X_{t_1}[\widetilde{\alpha}](\bar{\omega})}^{t_1,T}$ for $\bar{\omega}\in \Omega$ such that 
		\begin{enumerate}
			\item 
			For $\mathbb{P}$-a.e.\ $\bar{\omega}\in \Omega$,
			\begin{align}\label{eqn:Markov}
				&\ \mathbb{E}\Big[\int_{t_1}^T L_{\cA}\big(X_t[\widetilde{\alpha}], \alpha_t\big)dt + g_{\cA}\big(X_T[\widetilde{\alpha}]\big) \ \Big\vert \  \mathcal{F}_{t_1}\Big](\bar{\omega}) \nonumber \\
				=&\  \mathbb{E}\Big[\int_{t_1}^T L_{\cA}\big(X_t[t_1,X_{t_1}[\widetilde{\alpha}](\bar{\omega}), \widetilde{\alpha}^{\bar{\omega}}], \alpha^{\bar{\omega}}_t\big)dt + g_{\cA}\big(X_T[t_1,X_{t_1}[\widetilde{\alpha}](\bar{\omega}), \widetilde{\alpha}^{\bar{\omega}}]\big)\Big];
			\end{align}
			\item 
			For  every  $\bar{\omega}\in \Omega$, $(\alpha_t^{\bar{\omega}})_{t\in [t_1,T]}$ is independent of $\mathcal{F}_{t_1}$.
		\end{enumerate}
	\end{lemma}

	\begin{proof}
		For $s,t\in [0,T]$ we let $s\wedge t := \min\{t,s\}$. Define $\mathcal{B}_t = \mathcal{B}(\{u(s\wedge t)_{0\leq s\leq T}: u\in C([0,T])\})$ to be the Borel $\sigma$-algebra of paths stopped at time $t$ (and continued by a constant). We then let $\mathcal{B}_{t+}=\cup_{s>t}\mathcal{B}_s$ be the right limit.
		By Theorem 2.10 in \cite[Chapter 2]{yong1999stochastic}, there exists $\eta:[0,T]\times C([0,T])\rightarrow L^2(\cA)^d_{sa}$, which is progressively measurable with respect to $(\mathcal{B}_{t+})_{0\leq t\leq T}$, such that for $t\in [t_0,T]$
		$$
		\alpha_t (\omega)= \eta\big(t,(W^0_{s\wedge t}(\omega))_{0\leq s\leq T}\big).
		$$ 
		
		
		For each $\omega,\bar{\omega}\in \Omega$ and $t\in [t_1,T]$, we now define
		$$
		\alpha_t^{\bar{\omega}}(\omega) := \eta\big(t, (W^0_{s\wedge t_1}(\bar{\omega}) + W^0_{s\wedge t}(\omega) - W^0_{s\wedge t_1}(\omega))_{0\leq s\leq T}\big).
		$$
		Using this, we define the control policy $\widetilde{\alpha}^{\bar{\omega}}\in \mathbb{A}_{\cA, X_{t_1}[\widetilde{\alpha}](\bar{\omega})}^{t_1,T}$ with the same free semicircular process (we take the increment $S_t' = S_t - S_{t_1}$ for $t' \in [t_1,T]$) and free filtration as in $\widetilde{\alpha}$. This is a valid control policy because it inherits the progressive measurability from $\eta$ by composition with a continuous process. In particular, for $\mathbb{P}$-a.e.\ $\bar{\omega}\in \Omega$,  $(\alpha_t^{\bar{\omega}})_{t_1\leq t\leq T}$ is progressively measurable with respect to the filtration $(\mathcal F_t)_{t\in [t_1,T]}.$ Since $\alpha_t^{\bar{\omega}}$ is a function of the increments $W^0_s-W^0_{t_1}$ for $s\in [t_1,t]$, which are independent from $\mathcal{F}_{t_1}$, we deduce that $(\alpha_t^{\bar{\omega}})_{t\in [t_1,T]}$ is independent of $\mathcal{F}_{t_1}$.
		
		To show (\ref{eqn:Markov}), we argue that the free stochastic differential equation for $t\in [t_1,T]$ may be viewed on the probability space $\widehat{\Omega} := \{u\in C([0,T]): u(s) = 0,\ s\in [0,t_1]\}$. We define $\phi:\Omega \rightarrow \widehat{\Omega}$ by $\phi(\omega) = (W_{s}^0 (\omega) - W^{0}_{s\wedge t_1} (\omega) )_{0\leq s\leq T}$. We let $\widehat{\mathbb{P}} := \phi_\#\mathbb{P}$ and consider the filtration $(\widehat{\mathcal{F}}_t)_{t_1\leq t\leq T}$ generated by the canonical stochastic process $\widehat{W}^{0}_t := u(t)$. We claim this makes $(\widehat{W}_t^0)_{t_1\leq t\leq T}$ a standard Brownian motion on $(\widehat{{\Omega}}, (\widehat{\mathcal F}_t)_{t_1\leq t\leq T}, \widehat{\mathbb{P}})$.
		Let $f \in C(\bR^{k+n})$ and consider real numbers $0 \leq b_k< \cdots< b_0< t_1\leq a_1<a_2< \cdots<a_n$. Defining $\widehat f: \widehat \Omega \to \bR$ by 
		\begin{equation}\label{eq:dec29.2024.1}
			\widehat f(\widehat \omega)=f\Big(\widehat \omega(b_k), \cdots, \widehat \omega(b_1), \widehat \omega(a_1), \cdots, \widehat \omega(a_n) \Big), \qquad \hbox{for all } \widehat \omega \in \widehat \Omega,
		\end{equation}
		we have that 
		\begin{equation}
			\widehat f \circ \phi (\omega)= f\Big(0, \cdots, 0, W^0_{a_1} (\omega)-W^0_{t_1}(\omega), \cdots, W^0_{a_n}(\omega)-W^0_{t_1}(\omega) \Big). \label{eq:dec28.2024.1}
		\end{equation}
		Exploiting \eqref{eq:dec28.2024.1}, we compute $\int_{\Omega} \widehat f \circ \phi(\omega) \bP(d \omega)$ which by definition is nothing but $\int_{\widehat \Omega} \widehat f(\widehat \omega) \widehat{\mathbb{P}}(d \widehat \omega)$. We use the arbitrariness of $\widehat f$ to 
		deduce that $\widehat{\mathbb{P}}$ is the restriction of Wiener measure to $\widehat \Omega$.
		
		We denote by $\mathbb{P}^{\bar{\omega}}$ (resp. $\mathbb{E}^{\bar{\omega}}$)  the conditional probability measure  (resp. conditional expectation) given $\mathcal{F}_{t_1}$ at $\bar{\omega}$, i.e., for any random variable $Y\in L^2(\Omega,\mathcal{F}_T,\mathbb{P})$, 
		$$
		\mathbb{E}^{\bar{\omega}}[Y] = \int_{\Omega} Y(\omega)\mathbb{P}^{\bar{\omega}}(d\omega) = \mathbb{E}[Y\mid \mathcal{F}_{t_1}](\bar{\omega}).
		$$
		We continue to assume that $\widehat f$ is as in \eqref{eq:dec29.2024.1}. Since the increment $(W^0_s - W^0_{t_1})_{s\in [t_1,T]}$  is independent of $\mathcal{F}_{t_1}$, we exploit \eqref{eq:dec28.2024.1} to  check that 
		\begin{equation}\label{eq:dec28.2024.3} 
			\int_{\hat{\Omega}} \widehat f(\widehat \omega) ( \phi_\# \bP^{\bar{\omega}})(d\widehat \omega)= \int_{\hat{\Omega}} \widehat f(\widehat \omega) \widehat{\mathbb{P}}(d\widehat \omega).
		\end{equation}
		By the arbitrariness of $\widehat f$, we infer  $\phi_\#\mathbb{P}^{\bar{\omega}} = \widehat{\mathbb{P}}.$ For $t\in [t_1,T]$, we let $\mathcal{F}^{\bar{\omega}}_t$ be the completion of the $\sigma$-algebra generated by $\{W^0_s-W^0_{t_1}: s\in[t_1,t]\}$ with respect to the measure  $\mathbb{P}^{\bar{\omega}}$. Since $\mathbb{P}^{\bar{\omega}} \ll \mathbb{P}$ for $\mathbb{P}$-a.s. $\bar{\omega}$, this completion results in a larger filtration (despite having restricted the time interval) where $\mathcal{F}_t\subset \mathcal{F}_t^{\bar{\omega}}$ for each $t\in [t_1,T]$.
		
		For each $\bar{\omega}\in \Omega$ and $t\in [t_1,T]$, we define a control policy on $\widehat{\Omega}$ by
		$$
		\widehat{\alpha}_t^{\bar{\omega}}(u) := \eta\big(t, (W^0_{s\wedge t_1}(\bar{\omega}) + u_{s\wedge t})_{0\leq s\leq T}\big),\qquad u\in \widehat{\Omega}.
		$$
		We have that
		\begin{equation}\label{eq:dec31.2024.1}
			\widehat{\alpha}_t^{\bar{\omega}}\big(\phi(\omega)\big) = \alpha_t^{\bar{\omega}}(\omega) 
			\hbox{ for all $w \in \Omega$ and} \; t \in [t_1, T]
		\end{equation}
		and
		\begin{equation}\label{eq:jan01.2025.1}
			\widehat{\alpha}_t^{\bar{\omega}}\big(\phi(\omega)\big) = \alpha_t(\omega) = \alpha^{\bar \omega}_t(\omega) \hbox{ for }\ \mathbb{P}^{\bar{\omega}} \hbox{almost surely } \omega, \; \text{and for all} \;\; t \in [t_1, T].
		\end{equation}
		The second inequality follows as $W^0_{s\wedge t_1}(\omega) = W^0_{s\wedge t_1}(\bar{\omega})$ for all $s\in [0,t_1]$ for  $\mathbb{P}^{\bar{\omega}}$-almost every $\omega$. 
		Furthermore,
		\begin{equation}\label{eq:jan02.2025.1}
			\alpha_\tau(\omega) = \alpha_\tau(\bar \omega) \quad \hbox{ for }\ \mathbb{P}^{\bar{\omega}}\ \hbox{almost surely } \omega, \; \text{and for all} \;\; \tau \in [0, t_1],
		\end{equation}
		which implies
		\begin{equation}\label{eq:jan03.2025.1}
			X_{t_1}[\widetilde \alpha](\omega)=
			X_{t_1}[\widetilde \alpha](\bar\omega)\quad \hbox{ for }\ \mathbb{P}^{\bar{\omega}}\ \hbox{almost surely } \omega
		\end{equation}

		We now consider the control policy $\widetilde{{\hat{\alpha}}^{\bar{\omega}}}\in \hat{\mathbb{A}}_{\cA,X_{t_1}[\widetilde{\alpha}](\bar{\omega})}^{t_1,T}$, using the same  free filtration as for $\widetilde{\alpha}$ and the free Brownian motion $S'_t := S_t - S_{t_1}$ for $t \in [t_1,T]$.
		We let $\big(\widehat{X}_t\big[\widetilde{{\hat{\alpha}}^{\bar{\omega}}}\big]\big)_{t_1\le t\le T}$ denote the solution to the following SDE  on the probability space $(\widehat{\Omega},(\widehat{\mathcal{F}}_t)_{t_1\leq t\leq T},\widehat{\mathbb{P}})$ with Brownian motion $(\widehat{W}^0_t)_{t_1\leq t\leq T}$ and the initial condition  $X_{t_1}[\widetilde \alpha](\bar \omega)$:
		\begin{align}\label{eqn:coupled_dynamics}
			d\widehat{X}_t = b_{\cA}(\widehat{X}_t, \widehat{\alpha}_t)\, dt + \beta_C\, \mathbbm{1}_{\cA} d\widehat{W}_t^0 + \beta_F\, dS_t'.
		\end{align}
		Theorem \ref{thm:well_posedness} applies with this probability space, showing well-posedness. {{Using \eqref{eq:dec31.2024.1}, we see that}}
		$(\widehat{X}_t[\widetilde{{\hat{\alpha}}^{\bar{\omega}}}] \circ \phi)_{t_1\leq t\leq T}$ is a solution of (\ref{eqn:common_noise}) on the original probability space $\Omega$ with the control policy $\widetilde{\alpha}^{\bar{\omega}}$.
		It follows that for all $t\in [t_1,T]$ and $\mathbb{P}$ almost surely
		\begin{equation}\label{eq:composeFlow1}
			\widehat{X}_t\Big[\widetilde{{\hat{\alpha}}^{\bar{\omega}}}\Big] \circ \phi = X_t\Big[t_1,X_{t_1}[\widetilde{\alpha}](\bar{\omega}), \widetilde{\alpha}^{\bar{\omega}}\Big]
		\end{equation}
		and
		\begin{align} \label{111}
			&\ \widehat{\mathbb{E}}\bigg[\int_{t_1}^T L_{\cA}\Big(\widehat{X}_t\Big[\widetilde{{\hat{\alpha}}^{\bar{\omega}}}\Big], {{\widehat\alpha}_t^{\bar \omega}}\Big)dt + g_{\cA}\Big(
			\widehat{X}_T\Big[\widetilde{{\hat{\alpha}}^{\bar{\omega}}}\Big]\Big)\bigg]\nonumber \\
			=&\ \mathbb{E}\bigg[\int_{t_1}^T L_{\cA}\Big(X_t\Big[t_1,X_{t_1}[\widetilde{\alpha}](\bar{\omega}), \widetilde{{\alpha}}^{\bar{\omega}}\Big], {{\alpha_t^{\bar \omega}}}\Big)dt + g_{\cA}\Big(
			X_T\Big[t_1,X_{t_1}[\widetilde{\alpha}](\bar{\omega}), \widetilde{{\alpha}}^{\bar{\omega}}\Big]
			\Big)\bigg].
		\end{align}
		Note that $(W^0_t-W^0_{t_1})_{t_1\leq t \leq T}$ is a Brownian motion with filtration $(\cF_t^{\bar{\omega}})_{t_1 \leq t \leq T}$ on $(\Omega,  \cF^{\bar{\omega}}_T,\mathbb{P}^{\bar{\omega}})$. Let $X_t^{\bar{\omega}}[\widetilde{\alpha}]$ denote the solution to  \eqref{eqn:common_noise} on the  probability space $(\Omega,\mathcal{F}^{\bar{\omega}}_T,  (\cF^{\bar{\omega}}_t)_{t_1\leq t\le T},\mathbb{P}^{\bar{\omega}})$ with the initial condition  $X_{t_1}[\widetilde \alpha](\bar \omega)$:
			\begin{align}\label{eqn:omega_bar}
				X_t^{\bar{\omega}}[\widetilde{\alpha}](\omega)=X_{t_1}[\tilde \alpha](\bar \omega) + &
				\int_{t_1}^tb_{\cA}(
				X_s^{\bar{\omega}}[\widetilde{\alpha}](\omega), \alpha_s(\omega))\, ds \nonumber \\
				&+ \beta_C\, \mathbbm{1}_{\cA} ({W}_t^0(\omega)-W^0_{t_1}(\omega)) + \beta_F\, (S_t-S_{t_1}).
			\end{align}
			This holds $\mathbb{P}^{\bar{\omega}}$ almost everywhere, and $(X_t^{\bar{\omega}}[\widetilde{\alpha}])_{t_1\leq t\leq T}$ is $(\mathcal{F}^{\bar{\omega}}_t)_{t_1\leq t\leq T}$ progressively measurable.
			Since $\mathbb{P}^{\bar{\omega}}\ll \mathbb{P}$ and $\mathcal{F}_t\subset \mathcal{F}_t^{\bar{\omega}}$, for $\mathbb{P}$  a.s.\ $\bar{\omega}$ we also have that $X_t[\tilde{\alpha}]$ satisfies (\ref{eqn:omega_bar}) $\mathbb{P}^{\bar{\omega}}$ a.s. and is $(\mathcal{F}^{\bar{\omega}}_t)_{t_1\leq t\leq T}$ progressively measurable. Thus, using \eqref{eq:jan03.2025.1} and the uniqueness of solutions we deduce that for $\mathbb{P}$ a.s.\ $\bar{\omega}$, we have  $X_t^{\bar{\omega}}[\widetilde{\alpha}](\omega) = X_t[\widetilde{\alpha}](\omega)$ for $\mathbb{P}^{\bar{\omega}}$ almost every $\omega$. By the definition of the conditional probability
			$$
			\mathbb{E}\Big[\int_{t_1}^T L_{\cA}\big(X_t[\widetilde{\alpha}], \alpha_t\big)dt + g_{\cA}\big(X_T[\widetilde{\alpha}]\big)\Big| \mathcal{F}_{t_1}\Big](\bar{\omega}) = \mathbb{E}^{\bar{\omega}}\Big[\int_{t_1}^T L_{\cA}\big(X^{\bar{\omega}}_t[\widetilde{\alpha}], \alpha_t\big)dt + g_{\cA}\big(X_T^{\bar{\omega}}[\widetilde{\alpha}]\big)\Big].
			$$

			Finally, we have that $\mathbb{P}^{\bar{\omega}}$ almost surely
			$$
			\widehat{X}_t[\widetilde{{\hat{\alpha}}^{\bar{\omega}}}] \circ \phi = X_t^{\bar{\omega}}[\widetilde{\alpha}],
			$$
			and thus recalling $\phi_\#\mathbb{P}^{\bar{\omega}} = \widehat{\mathbb{P}}$,
			\begin{align*}
				&\ \mathbb{E}^{\bar{\omega}}\Big[\int_{t_1}^T L_{\cA}\big(X^{\bar{\omega}}_t[\widetilde{\alpha}], \alpha_t\big)dt + g_{\cA}\big(X_T^{\bar{\omega}}[\widetilde{\alpha}]\big)\Big]\\
				&= \widehat{\mathbb{E}}\Big[\int_{t_1}^T L_{\cA}\big(\widehat{X}_t[\widetilde{{\hat{\alpha}}^{\bar{\omega}}}], \widehat{\alpha}_t^{\bar{\omega}}\big)dt + g_{\cA}\big(\widehat{X}_T[\widetilde{{\hat{\alpha}}^{\bar{\omega}}}]\big)\Big]\\
				&\overset{\eqref{111}}{=} \mathbb{E}\Big[\int_{t_1}^T L_{\cA}\big(X_t[t_1,X_{t_1}[\widetilde{\alpha}](\bar{\omega}), \widetilde{\alpha}^{\bar{\omega}}], \alpha_t^{\bar{\omega}}\big)dt + g_{\cA}\big(X_T[t_1,X_{t_1}[\widetilde{\alpha}](\bar{\omega}), \widetilde{\alpha}^{\bar{\omega}}]\big)\Big].
			\end{align*}
			Therefore we conclude the proof.
			
	\end{proof}

	\begin{proposition}\label{prop:subdynamic_programming}
		Suppose that Assumption \textbf{A} holds.  The value function $\overline{V}$, defined in (\ref{eqn:value}), satisfies the sub-dynamic programming principle: For any $\cA\in \bW$, $x_0\in L^2(\cA)_{sa}^d$ and  $0\leq t_0\leq t_1\leq T$,
		\begin{align} \label{350}
			\overline{V}_{\cA}(t_0,x_0) \leq  \inf_{\widetilde{\alpha}\in \bA_{\cA,x_0}^{t_0,t_1}}\Big\{ \bE\Big[\int_{t_0}^{t_1} L_{\cA}(X_s[\widetilde{\alpha}], \alpha_s)ds +\overline{V}_{\cA}(t_1,X_{t_1}[\widetilde{\alpha}])\Big]: X_{t_0}[\widetilde{\alpha}]=x_0\Big\}.
		\end{align}
		Furthermore, $\overline{V}$ is a free viscosity subsolution in the sense of Definition \ref{def:WassSpaceViscosity}.
	\end{proposition}

	\begin{proof}
		First note that if $\bA_{\cA,x_0}^{t_0,t_1}$ is empty, i.e., $\cA$ does not support a free Brownian motion independent of $x_0$, then RHS of (\ref{350}) is $+\infty$ and the inequality holds.  Hence let us assume that  $\cA$  supports a free Brownian motion on $[t_0,t_1]$ independent of $x_0$. 
		
		Given an arbitrary control policy $\widetilde{\alpha}\in \bA_{\cA,x_0}^{t_0,t_1}$, we will construct a larger space $\cB \in \bW$ with a tracial $W^*$-embedding $\iota:\cA\rightarrow \mathcal{B}$ and control policy $\widetilde{\alpha}'\in  \bA_{\mathcal{B},\iota\, x_0}^{t_0, T}$ which is $\epsilon$-optimal on $[t_1,T]$, i.e., 
		\begin{align} \label{nearoptimal}
			\bE\Big[ \overline{V}_{\cB}(t_1,X_{t_1}[\widetilde{\alpha}'])\Big] \geq -\epsilon + \bE\Big[\int_{t_1}^T L_{\mathcal{B}}(X_t[\widetilde{\alpha}'],\alpha'_t)dt + g_{\mathcal{B}}(X_T[\widetilde{\alpha}'])\Big].
		\end{align}
		Furthermore, the control policy will be consistent on $[t_0,t_1]$ in the sense that $\alpha_t' = \iota\, \alpha_t$ and $X_t[\widetilde{\alpha}'] = \iota\, X_t[\widetilde{\alpha}]$  on $[t_0,t_1]$. Indeed, given this consistency, we have
		\begin{align*}
			\overline{V}_{\cA}(t_0,x_0) &\le  \bE\Big[\int_{t_0}^{T} L_{\mathcal{B}}(X_t[\widetilde{\alpha}'], \alpha'_t)dt +g_{\mathcal{B}}(X_T[\widetilde{\alpha}'])\Big]\\
			& = \bE\Big[\int_{t_0}^{t_1} L_{\mathcal{B}} (X_t[\widetilde{\alpha}'], \alpha'_t)dt\Big] + \bE\Big[\int_{t_1}^{T}  L_{\mathcal{B}} 
			(X_t[\widetilde{\alpha}'], \alpha'_t)dt +g_{\mathcal{B}}(X_T[\widetilde{\alpha}'])\Big]\\
			& \overset{\eqref{nearoptimal}}{\le} \bE\Big[\int_{t_0}^{t_1} L_{\mathcal{A}}(X_t[\widetilde{\alpha}], \alpha_t)dt \Big]+ \bE\Big[ \overline{V}_{\cA}(t_1,X_{t_1}[\widetilde{\alpha}])\Big] + \epsilon,
		\end{align*}
		where the last inequality follows from the tracial property of the value function $
		(\overline{V}_\cA)_{\cA \in \bW}$. Taking $\epsilon \rightarrow 0^+$ proves the sub-dynamic programming principle.
		
		~
		
		In order to construct a tracial $W^*$-embedding $\iota:\cA\rightarrow \mathcal{B}$ along with a consistent policy $\widetilde{\alpha}'\in  \bA_{\mathcal{B},\iota\, x_0}^{t_0, T}$ satisfying \eqref{nearoptimal}, we reduce to taking a finite collection of points in the von Neumann algebra that approximate $X_{t_1}[\widetilde{\alpha}]$ and apply Lemma \ref{lem: amalgamation of Brownian motions} to amalgamate spaces with nearly optimal control policies for each point. Let $(\cA_t)_{t\in[t_0,T]}$ be the free filtration  and $(S_t^0)_{t\in [t_0,T]}$ be the compatible free Brownian motion of $\widetilde{\alpha}$.
		
		We recall that by adaptedness, $X_{t_1}[\widetilde{\alpha}]\in L^2(\cA_{t_1})_{sa}^d$. By separability of the space $L^2(\cA_{t_1})_{sa}^d$, there is a countable dense collection $\{z_k\}_{k=1}^\infty$ in $L^2(\cA_{t_1})_{sa}^d$. For a positive integer $N$,  let $K^{N,r} := \cup_{k=1}^N B_r(z_k)$ be the union of balls of radius $r>0$ around the first $N$ points. Then, for any $r>0,$
		\begin{align} \label{411}
			\lim_{N\rightarrow \infty}    \mathbb{P}(X_{t_1}[\widetilde{\alpha}]\notin K^{N,r})=  \mathbb{P}(X_{t_1}[\widetilde{\alpha}]\notin \cup_{k=1}^\infty  B_r(z_k) ) = 0 .
		\end{align} 
		We claim that for any $r>0$ and $\epsilon>0$, there exists sufficiently large $N$
		such that
		\begin{align} \label{351}
			\mathbb{E}\Big[\overline{V}_{\cA}(t_1,X_{t_1}[\widetilde{\alpha}])\Big] \geq - \epsilon + \mathbb{E}\Big[\mathbbm{1}_{X_{t_1}[\widetilde{\alpha}]\in K^{N,r}}\, \overline{V}_{\cA}(t_1,X_{t_1}[\widetilde{\alpha}])\Big].
		\end{align}
		Indeed, by H\"older inequality,  
		\begin{align*}
			\mathbb{E}\Big[ \big\vert \mathbbm{1}_{X_{t_1}[\widetilde{\alpha}]\notin K^{N,r}} \overline{V}_{\cA}(t_1,X_{t_1}[\widetilde{\alpha}])\big\vert \Big] &\le  \sqrt{\mathbb{P}( X_{t_1}[\widetilde{\alpha}]\notin K^{N,r}) } \sqrt{ \mathbb{E} |\overline{V}_{\cA}(t_1,X_{t_1}[\widetilde{\alpha}])|^2} \\
			&\overset{\eqref{eqn:global_bound}}{\le}  C\sqrt{\mathbb{P}( X_{t_1}[\widetilde{\alpha}]\notin K^{N,r}) } \sqrt{1 +  \mathbb{E} \norm{ X_{t_1}[\widetilde{\alpha}]}_{L^2(\cA)}^2}.
		\end{align*}
		By \eqref{eqn:time_bound_L2} in  Theorem \ref{thm:well_posedness}, 
		\begin{align*}
			\mathbb{E} \big[\norm{ X_{t_1}[\widetilde{\alpha}]}_{L^2(\cA)}^2\big] \le   2\bE \big[\big\| X_{t_1}[\widetilde{\alpha}]-x_0\big\|_{L^2(\cA)}^2\big] + 2
			\|x_0\big\|_{L^2(\cA)}^2 \leq \widetilde{C}_{M,N}\, (t_1-t_0) +  2
			\|x_0\big\|_{L^2(\cA)}^2 .
		\end{align*}
		This together with  \eqref{411} verifies \eqref{351}.

		Take a partition $K^{N,r} = \sqcup_{k=1}^N K_k$ such that $K_k\subset B_r(z_k)$. We choose
		\begin{align} \label{236}
			r = \frac{\epsilon}{4\, C_2(1+T)\sqrt{\widetilde{C}}},
		\end{align}
		using the Lipschitz coefficient of Proposition \ref{prop:continuity} so that $\overline{V}(t_1,\cdot)$ does not oscillate by more than $\epsilon$ in a ball of radius $r$. Setting $c_k:= \bP (X_{t_1}[\widetilde{\alpha}]\in K_k),$
		\begin{align} \label{352}
			\mathbb{E}\Big[\mathbbm{1}_{X_{t_1}[\widetilde{\alpha}]\in K^{N,r}}\, \overline{V}_{\cA}(t_1,X_{t_1}[\widetilde{\alpha}])\Big] &= \sum_{k=1}^N   \mathbb{E}\Big[\mathbbm{1}_{X_{t_1}[\widetilde{\alpha}]\in K_k}\, \overline{V}_{\cA}(t_1,X_{t_1}[\widetilde{\alpha}])\Big]  \nonumber   \\
			& \ge \sum_{k=1}^N 
			\mathbb{E}\Big[\mathbbm{1}_{X_{t_1}[\widetilde{\alpha}]\in K_k}\, 
			\big(\overline{V}_{\cA}(t_1,z_k) - \epsilon\big)\Big]  \nonumber  \\
			&  \ge  - \epsilon + \sum_{k=1}^N c_k\, \overline{V}_{\cA}(t_1,z_k).
		\end{align}
		Thus combining this with \eqref{351}, we obtain
		\begin{align} \label{237}
			\sum_{k=1}^N c_k\, \overline{V}_{\cA}(t_1,z_k) \le    \mathbb{E}\big[ \overline{V}_{\cA}(t_1,X_{t_1}[\widetilde{\alpha}])\big] + 2\epsilon.
		\end{align}

		We now apply Lemma \ref{lem: arranging initial filtration}, for each $k\in\{1,\cdots,N\}$, to infer the existence of a tracial $W^*$-algebra $\mathcal{B}^k$, a tracial $W^*$-embedding $\iota_k:\cA \rightarrow \mathcal{B}^k$, a control policy $\widetilde{\alpha}^k\in \bA^{t_1,T}_{\mathcal{B}^k, \iota_k z_k}$ with filtration $(\mathcal{B}_t^k)_{t\in [t_1,T]}$ and compatible $d$-variable free Brownian motion $(S_t^k)_{t\in [t_1,T]}$, such that $\iota_k(\cA)\subset \mathcal{B}_{t_1}^k$ and
		\begin{align} \label{353}
			\overline{V}_{\cA}(t_1,z_k) \geq&\ - \epsilon +  \bE\Big[\int_{t_1}^T L_{\mathcal{B}^k} (X_t[t_1,\iota_k z_k,\widetilde{\alpha}^k], \alpha_t^k)dt + g_{\mathcal{B}^k}(X_T[t_1,\iota_k z_k,\widetilde{\alpha}^k])\Big] .
		\end{align}
		{Furthermore, we may assume that $({\alpha}^k_t)_{t\in [t_1,T]}$ is independent of $\mathcal{F}_{t_1}$. To accomplish this, we first take $\bar{\omega}\in \Omega$ such that
			\begin{align*}
				&\bE\Big[\int_{t_1}^T L_{\mathcal{B}^k} (X_t[t_1,\iota_k z_k,\widetilde{\alpha}^k], \alpha_t^k)dt + g_{\mathcal{B}^k}(X_T[t_1,\iota_k z_k,\widetilde{\alpha}^k])\Big] \\
				& \ge \bE\Big[\int_{t_1}^T L_{\mathcal{B}^k} (X_t[t_1,\iota_k z_k,\widetilde{\alpha}^k], \alpha_t^k)dt + g_{\mathcal{B}^k}(X_T[t_1,\iota_k z_k,\widetilde{\alpha}^k])  \ \Big\vert \ \mathcal{F}_{t_1}\Big](\bar{\omega})  .
			\end{align*}
			By Lemma \ref{lem:Markov_property} with $t_0=t_1$,
			redefining \(\alpha_t^k := (\alpha_t^k)^{\bar{\omega}}\), we arrive at a process, satisfying \eqref{353}, that is \(\mathcal{F}_{t_1}\)-independent for all \(t \in [t_1,T]\), as required.}
		
		By construction we have $\iota_k(\cA_{t_1})\subset \cB_{t_1}^k$, so we may apply the amalgamation Lemma \ref{lem: amalgamation of Brownian motions} to find a larger algebra $\cB$ with tracial ${\rm W}^*$ embeddings $\iota:\cA\rightarrow \cB$ and $\widetilde{\iota}_k:\cB^k\rightarrow \cB$, a filtration $(\mathcal{B}_t)_{t\in [t_0,T]}$, and a compatible free Brownian motion $(S_t)_{t\in [t_0,T]}$ such that $\iota\, (S^0_t - S^0_{t_0}) = S_t-S_{t_0}$ for $t\in [t_0,t_1]$ and $\tilde{\iota}_k(S^k_t-S^k_{t_1}) = S_t-S_{t_1}$ for each $k$ and $t\in [t_1,T]$.  We define a control policy $\widetilde{\alpha}'_t\in \mathbb{A}^{t_0,T}_{\cB, \iota\, x_0}$ using the filtration $(\mathcal{B}_t)_t$, free Brownian motion $(S_t)_{t}$ and the control  $(\alpha'_t)_t$ defined as follows:
		{$\alpha_t':= \iota\, \alpha_t$ for $t\in [t_0,t_1)$, and on  $[t_1,T)$
			\begin{align} \label{234}
				\alpha_t' := \begin{cases}
					\widetilde  \iota_k \alpha_t^k \quad  &  \text{on the event $X_{t_1}[\widetilde{\alpha}] \in K_k$ for some $k=1,\cdots,N$}, \\
					0 &   \text{otherwise.}
				\end{cases}
			\end{align}
		}
		Note that the properties of Lemma \ref{lem: amalgamation of Brownian motions} imply that $(\alpha_t')_{t\in [t_0,T]}$ is adapted to $(\cB_t)_{t\in [t_0,T]}$. {Also, $(\alpha_t')_{t\in [t_0,T]}$ is adapted to $(\mathcal{F}_{t})_{t\in [t_0,T]}$.}
		
		For $t\in [t_0,t_1]$, both $\iota\, X_t[t_0,x_0,\widetilde{\alpha}]$ and $X_t[t_0,\iota\, x_0,\widetilde{\alpha}']$ solve the same SDE and thus agree by Theorem \ref{thm:well_posedness}. Hence they are consistent, i.e. $X_t[\widetilde{\alpha}'] = \iota\, X_t[\widetilde{\alpha}]$  (we omit the initial time and position for the sake of readability). Similarly, for $t\in [t_1,T]$,  on the event $X_{t_1}[\widetilde{\alpha}]\in K_k$,
		\begin{align*}
			d\, \widetilde{\iota}_k\, X_t[t_1,\iota_k\, X_{t_1}[\widetilde{\alpha}],\widetilde{\alpha}^k]=&\ \widetilde{\iota}_k\, b_{\cB_k}( X_t[t_1,\iota_k\, X_{t_1}[\widetilde{\alpha}],\widetilde{\alpha}^k], \alpha^k_t)dt +\beta_C\, \widetilde{\iota}_k\, \mathbbm{1}_{\cB_k}dW^0_t + \beta_F\,  \widetilde{\iota}_k\, dS_t^k\\
			=&\   b_{\cB}(\widetilde{\iota}_k\, X_t[t_1,\iota_k\, X_{t_1}[\widetilde{\alpha}],\widetilde{\alpha}^k], \alpha'_t)dt +\beta_C\, \mathbbm{1}_{\cB}dW^0_t +\beta_F\, dS_t,
		\end{align*}
		recalling that the free SDE formulation is shorthand for an integral form where $\int_{t_1}^t\beta_F\, \widetilde{\iota}_k dS_t^k = \beta_F\, \widetilde{\iota}_k (S_t^k-S_{t_1}^k)$ for $t\in [t_1,T]$.
		Thus it follows that $X_t[t_1,\iota\, X_{t_1}[\widetilde{\alpha}],\widetilde{\alpha}'] =\ \widetilde{\iota}_k\, X_t[t_1,\iota_k X_{t_1}[\widetilde{\alpha}],\widetilde{\alpha}^k]$,  
		since they solve the same SDE and thus agree by Theorem \ref{thm:well_posedness}.

		Then we bound RHS of \eqref{nearoptimal} as follows: Noting that  $X_t[ \widetilde \alpha'] = X_t[t_1,\iota\, X_{t_1}[\widetilde{\alpha}], \widetilde{\alpha}']$,
		\begin{align} \label{233}
			\bE & \Big[\int_{t_1}^T  L_{\mathcal{B}} (X_t[t_1,\iota\, X_{t_1}[\widetilde{\alpha}], \widetilde{\alpha}'], \alpha_t')dt + g_{\mathcal{B}}(X_T[t_1,\iota\, X_{t_1}[\widetilde{\alpha}], \widetilde{\alpha}'])\Big] \nonumber\\
			& =  \bE \Big[ \bE\Big[\int_{t_1}^T L_{\mathcal{B}} (X_t[t_1,\iota\, X_{t_1}[\widetilde{\alpha}], \widetilde{\alpha}'], \alpha_t')dt + g_{\mathcal{B}}(X_T[t_1,\iota\, X_{t_1}[\widetilde{\alpha}], \widetilde{\alpha}']) \ \Big\vert \  \mathcal{F}_{t_1}\Big] \Big]\nonumber \\
			& \le \sum_{k=1}^N    \bE \Big[ \bE\Big[\int_{t_1}^T L_{\mathcal{B}} (X_t[t_1,\iota\, X_{t_1}[\widetilde{\alpha}], \widetilde{\alpha}'], \alpha_t')dt + g_{\mathcal{B}}(X_T[t_1,\iota\, X_{t_1}[\widetilde{\alpha}], \widetilde{\alpha}']) \ \Big\vert \ 
			\mathcal{F}_{t_1} \Big] \mathbbm{1}_{X_{t_1}[\tilde{\alpha}]\in K_k} \Big] \nonumber \\
			&\qquad \qquad +C \mathbb{E} \Big[1+ \norm{ X_{t_1}[\widetilde{\alpha}]}_{L^2(\cA)}^2 +  \int_{t_1}^T  \norm{  {\alpha_t}'}_{L^2(\cA)}^2 dt  \Big]  \mathbb{P}(X_{t_1}[\widetilde{\alpha}]\notin K^{N,r}),
		\end{align}
		where we used the upper bound condition \eqref{eqn:lower_and_upper_bounds}  and  \eqref{eqn:time_bound_L2} in Theorem \ref{thm:well_posedness} to obtain the last term. Observe that by the Lipschitz condition \eqref{eqn:Lipschitz_bounds} along with \eqref{eqn:initial_condition_continuity_2}, for any $\cA' \in \bW$, $y_1,y_2 \in L^2(\cA')^d_{sa}$ and a control $\beta = (\beta_t)_t,$
		\begin{align} \label{238}
			&\Big\vert  \mathbb E\Big[ \int_{t_1}^T L_{\cA'} (X_t[t_1,y_1, \beta],\beta_t)dt + g_{\cA'}(X_T[t_1,y_1, \beta]) \Big] \nonumber \\
			&\qquad-   \mathbb E \Big[\int_{t_1}^T L_{\cA'} (X_t[t_1,y_2, \beta],\beta_t)dt + g_{\cA'}(X_T[t_1,y_2, \beta])\Big] \Big\vert  \le  C \norm{y_1-y_2}_{L^2(\cA')}.
		\end{align}
		Now recall the definition of $(\alpha'_t)_t$ in \eqref{234} and $\norm{X_{t_1}[\widetilde{\alpha}] - z_k}_{L^2(\cA)} \le r$ on the event $X_{t_1}[\widetilde{\alpha}]\in K_k$. {The independence of $({\alpha}^k_t)_{t\in [t_1,T]}$ and $\mathcal{F}_{t_1}$ also implies independence of $(X_t[t_1,\iota_k\, z_k, \tilde{\alpha}^k])_{t\in [t_1,T]}$  and $\mathcal{F}_{t_1}$.} We get that the first term in \eqref{233}  is bounded by
		\begin{align*}
			&\sum_{k=1}^N    \bE \Big[ \bE\Big[\int_{t_1}^T L_{\mathcal{B}^k} (X_t[t_1,\iota_k\, X_{t_1}[\widetilde{\alpha}], \widetilde{\alpha}^k], \alpha^k_t)dt + g_{\mathcal{B}^k}(X_T[t_1,\iota_k\, X_{t_1}[\widetilde{\alpha}],  \widetilde{\alpha}^k]) \ \Big\vert \ 
			\mathcal{F}_{t_1}\Big] \mathbbm{1}_{X_{t_1}[\widetilde{\alpha}]\in K_k} \Big]  \\
			&\overset{\eqref{238}}{\le} \sum_{k=1}^N    \bE \Big[  \Big ( \bE\Big[\int_{t_1}^T L_{\mathcal{B}^k} (X_t[t_1,\iota_k z_k, \widetilde{\alpha}^k], \alpha^k_t)dt + g_{\mathcal{B}^k}(X_T[t_1,\iota_k z_k , \widetilde{\alpha}^k]) \Big]  + Cr \Big) \mathbbm{1}_{X_{t_1}[\widetilde{\alpha}]\in K_k} \Big]\\
			& \overset{\eqref{353}}{\le}  \sum_{k=1}^N    \bE \Big[  
			\big ( \overline V_{\cA} (t_1, z_k) + \epsilon \big)\mathbbm{1}_{X_{t_1}[\widetilde{\alpha}]\in K_k}  \Big] +  Cr \\
			&\le  \sum_{k=1}^N   c_k\overline V_{\cA} (t_1, z_k)  + \epsilon + Cr \\
			&\overset{\eqref{237}}{\le}  \bE \big[\overline V_{\cA} (t_1, X_{t_1}[\widetilde{\alpha}])\big] +   3\epsilon  + Cr  \overset{\eqref{236}}{=}  \bE \big[\overline V_{\cA} (t_1, X_{t_1}[\widetilde{\alpha}])\big] +   C'\epsilon  .
		\end{align*}
		
		Next, the second term in \eqref{233} converges to 0, due to \eqref{411} and the boundedness of $\mathbb{E} \big[\norm{ X_{t_1}[\widetilde{\alpha}]}_{L^2(\cA)}^2\big]$ (see \eqref{eqn:time_bound_L2} in Theorem \ref{thm:well_posedness}).
		Thus noting $\overline{V}_{\cB}(t_1,\iota\, X_{t_1}[\widetilde{\alpha}]) = \overline{V}_{\cA}(t_1,X_{t_1}[\widetilde{\alpha}])$, we obtain \eqref{nearoptimal} and thus establish the sub-dynamic programming principle.
		
		~
		
		We now prove the viscosity subsolution property.  Suppose that $(\Phi_{\cA})_{\cA\in \bW}\in \mathcal{X}_\Sigma$ touches $(\overline{V}_\cA)_{\cA\in \bW}$ from above at   $(t_0,\lambda_0)\in [0,T)\times \Sigma_d^2$.  Let $\cA\in \bW$ and $x_0 \in L^2(\cA)_{sa}^d$ such that $\lambda_0 = \lambda_{x_0}$.  Take a larger von Neumann algebra $\cB$, along with a tracial $W^*$-embedding $\iota:\cA\rightarrow \cB$, which contains a free semi-circlular process $(S_t)_t$ freely independent of $\iota(\cA)$. Then for any $\bar{\alpha}\in \mathbb{A}_{\cA}$, consider {$\widetilde{\alpha} = ((\alpha_t)_{t},(\cA_t)_{t}, (S_t)_t
			) \in \bA_{\cB, \iota\, x_0}^{t_0,t_1}$} such that ${\alpha}_t=\iota\, \bar{\alpha}$ for $t\in [t_0,T)$, and then let $X_t[\widetilde{\alpha}]$ be  a solution to \eqref{eqn:common_noise}  on $L^2(\cB)_{sa}^d$  with an initial condition $X_{t_0}[\widetilde{\alpha}] = \iota\, x_0$. As the non-commutative law of $\iota\, x_0$, as an element in $\cB$, is $\lambda_0,$ for  any $0\leq t_0 < t_1\leq T$,
		\begin{align*}
			\Phi_{\cB}(t_0,\iota x_0) = \overline{V}_\cB (t_0,\iota x_0) &  \overset{\eqref{350}}{\le}  \  \bE\Big[\int_{t_0}^{t_1} L_{\cB}(X_s[\widetilde{\alpha}], \alpha_s)ds +\overline{V}_{\cB}(t_1,X_{t_1}[\widetilde{\alpha}])\Big].\\
			& \leq \  \bE\Big[\int_{t_0}^{t_1} L_{\cB}(X_s[\widetilde{\alpha}], \alpha_s)ds +\Phi_{\cB}(t_1,X_{t_1}[\widetilde{\alpha}])\Big].
		\end{align*}
		Using Lemma \ref{lem:Ito},  abbreviating   $X_{t}[\widetilde{\alpha}]$ to $X_{t}$, 
		\begin{align*}
			0 \leq &\  \bE\Big[\int_{t_0}^{t_1}\Big( L_{\cB}(X_s, \iota \, \bar{\alpha}) + \partial_t \Phi_{\cB}(s,X_s)\\
			&\  + \langle \nabla \Phi_{\cB}(s,X_s),b_{\cB}(X_s, \iota  \, \bar{\alpha})\rangle_{L^2(\cB)} + \frac{\beta_C^2}{2}\Delta_{\cB} \Phi (s,X_s) + \frac{\beta_F^2}{2}\Theta_{\cB} \Phi (s,X_s)\Big)ds\Big].
		\end{align*}
		Dividing by $t_1-t_0$ and then taking the limit $t_1\rightarrow^+t_0$, we find that 
		\begin{align} \label{355}
			-\partial_t \Phi_{\cB}(t_0,\iota\, x_0) &-\langle \nabla \Phi_{\cB}(t_0,\iota\, x_0),b_{\cB}(\iota\, x_0, \iota\, \bar{\alpha})\rangle_{L^2(\cB)} - L_{\cB}(\iota\, x_0,\iota\, \bar{\alpha})   \nonumber \\
			&-\frac{\beta_C^2}{2}\Delta_{\cB} \Phi(t_0,\iota\, x_0)-\frac{\beta_F^2}{2}\Theta_{\cB} \Phi(t_0,\iota\, x_0) \leq 0.
		\end{align}
		Note that as $\bar{\alpha}\in L^2(\mathcal{A})_{sa}^d$,
		$$
		\langle \nabla \Phi_{\cB}(\iota\, x_0, \iota\, \bar{\alpha}),b_{\cB}(\iota\, x_0, \iota\, \bar{\alpha})\rangle_{L^2(\cB)} = \langle \nabla \Phi_{\cA}(t_0, x_0),b_{\cA}(x_0,\bar{\alpha})\rangle_{L^2(\cA)}.
		$$
		Hence since all the functions in \eqref{355} are tracial $W^*$-functions, taking the supremum over all $\bar{\alpha} \in \bA_\cA$, by \eqref{eqn:Hamiltonian} in Lemma \ref{lem:Hamiltonian},
		$$
		-\partial_t \Phi_{\cA}(t_0,x_0) +H_{\cA}\big(x_0,-\nabla \Phi_{\cA} (t_0,x_0)\big) - \frac{\beta_C^2}{2}\Delta_{\cA} \Phi(t_0,x_0)-\frac{\beta_F^2}{2}\Theta_{\cA} \Phi(t_0, x_0)\leq 0.
		$$
		As this holds for arbitrary $\cA\in \bW$ which contains a non-commutative law $\lambda_0$,  one can take a supremum over all such  $\cA\in \bW$ and thus establish the free viscosity subsolution property of $\overline{V}$.
	\end{proof}
	
	We contrast this with the proof that the value function on a given von Neumann algebra satisfies the superdynamic programming principle and is a supersolution. 
	
	\begin{proposition}\label{prop:dynamic_programming}
		Suppose that Assumption \textbf{A} holds. The value function  $\overline{V}$, defined in (\ref{eqn:value}), satisfies the super-dynamic programming principle: For any $\cA\in \bW$, $x_0\in L^2(\cA)_{sa}^d$ and  $0\leq t_0\leq t_1\leq T$,
		$$
		\overline{V}_{\cA}(t_0,x_0) \geq \inf_{\iota:\cA\rightarrow \cB}\inf_{\widetilde{\alpha}\in \bA_{\cB, \iota\, x_0}^{t_0,t_1}}\Big\{ \bE\Big[\int_{t_0}^{t_1} L_{\cB}(X_s[\widetilde{\alpha}], \alpha_s)ds +\overline{V}_{\cB}(t_1,X_{t_1}[\widetilde{\alpha}])\Big]: X_{t_0}[\widetilde{\alpha}]=\iota\, x_0\Big\}.
		$$
		Furthermore, $\overline{V}$ is a free viscosity supersolution in the sense of Definition \ref{def:WassSpaceViscosity}.
	\end{proposition}
	\begin{proof}
		For any  $t_0\in [0,T]$, $x_0\in L^2(\cA)_{sa}^d$  and  $\epsilon>0$, there exist $\cB\in \bW$ along with a tracial $W^*$-embedding $\iota: \cA \rightarrow \cB$ and  a control policy $\widetilde{\alpha}\in \bA_{\cB, \iota\, x_0}^{t_0,T}$ such that  for the corresponding  solution $(X_t[\widetilde{\alpha}])_t $ to \eqref{eqn:common_noise} on $L^2(\cB)_{sa}^d$  with $X_{t_0}[\widetilde{\alpha}] = \iota\, x_0$,  \begin{align*} 
			\overline{V}_{\cA}(t_0,x_0) &\geq   -\epsilon + \bE\Big[\int_{t_0}^{T} L_{\cB}(X_s[\widetilde{\alpha}], \alpha_s)ds +g_{\cB}(X_{T}[\widetilde{\alpha}])\Big] \\
			&=  -\epsilon + \bE\Big[\int_{t_0}^{t_1}  L_{\cB}(X_s[\widetilde{\alpha}], \alpha_s)ds  + \bE \Big [ \int_{t_1}^{T}  L_{\cB}(X_s[\widetilde{\alpha}], \alpha_s)ds   +g_{\cB}(X_{T}[\widetilde{\alpha}]) \ \Big\vert \  \mathcal F_{t_1} \Big] \Big] \nonumber\\
			&=  -\epsilon + \bE\Big[\int_{t_0}^{t_1}  L_{\cB}(X_s[\widetilde{\alpha}], \alpha_s)ds  \\
			& \qquad \qquad +  \bE \Big [ \int_{t_1}^{T}  L_{\cB}(X_s[t_1, X_{t_1}[\widetilde \alpha],  \widetilde \alpha], \alpha_s)ds   +g_{\cB}(X_T[t_1, X_{t_1}[\widetilde \alpha],  \widetilde \alpha ] ) \ \Big\vert \   \mathcal F_{t_1} \Big] \Big] 
		\end{align*}
		By Lemma \ref{lem:Markov_property}, for $\mathbb P$-a.s.  $\bar{\omega},$ there exists   a control policy $\widetilde{\alpha}^{\bar{\omega}}$ such that 
		\begin{align*}
			&\ \bE \Big [ \int_{t_1}^{T}  L_{\cB}(X_s[t_1, X_{t_1}[\widetilde \alpha],  \widetilde \alpha], \alpha_s)ds   +g_{\cB}(X_T[t_1, X_{t_1}[\widetilde \alpha],  \widetilde \alpha ] ) \ \Big\vert \ 
			\mathcal F_{t_1} \Big] (\bar{\omega})\\
			=&\ \mathbb{E}\Big[\int_{t_1}^T L_{\cA}\big(X_t[t_1,X_{t_1}[\widetilde{\alpha}](\bar{\omega}), \widetilde{\alpha}^{\bar{\omega}}], \alpha^{\bar{\omega}}_t\big)dt + g_{\cA}\big(X_T[t_1,X_{t_1}[\widetilde{\alpha}](\bar{\omega}), \widetilde{\alpha}^{\bar{\omega}}]\big)\Big]\\
			\geq&\ \bar{V}_{\cB}(t_1,X_{t_1}[\tilde{\alpha}](\bar{\omega})).
		\end{align*}
		Combining the above inequalities, we obtain 
		$$
		\overline{V}_{\cA}(t_0,x_0)\geq  -\epsilon + \bE\Big[\int_{t_0}^{t_1} L_{\cB}(X_s[\widetilde{\alpha}], \alpha_s)ds +\overline{V}_{\cB}(t_1,X_{t_1}[\widetilde{\alpha}])\Big].
		$$
		As $\epsilon>0$ is arbitrary,  we establish the super-dynamic programming principle. 
		
		~
		
		Next, in order to show that  $\overline{V}$ is a free viscosity supersolution,  we  assume that $(\Phi_{\cA})_{\cA\in \bW}\in \mathcal{X}_\Sigma$ touches $(\overline{V}_\cA)_{\cA\in \bW}$ from below at  $(t_0,\lambda_0)\in [0,T)\times \Sigma_d^2$.  Let $\cA\in \bW$ and $x_0 \in L^2(\cA)_{sa}^d$ such that $\lambda_0 = \lambda_{x_0}$.  For $t_1\in (t_0,T]$ and $\epsilon>0$, we choose $\cB\in \bW$ along with a tracial $W^*$-embedding $\iota: \cA \rightarrow \cB$ and $\widetilde{\alpha}\in \bA_{\cB,\iota\, x_0}^{t_0,t_1}$ such that 
		\begin{align*}
			\epsilon\, (t_1 - t_0) \geq&\ \bE\Big[\int_{t_0}^{t_1} L_{\cB}(X_t[\widetilde{\alpha}], \alpha_t)ds - \overline{V}_{\cA}(t_0,x_0)  + \overline{V}_{\cB}(t_1,X_{t_1}[\widetilde{\alpha}])\Big]\\
			\geq&\ \bE\Big[\int_{t_0}^{t_1} L_{\cB}(X_t[\widetilde{\alpha}], \alpha_t)ds - \Phi_{\cB}(t_0,\iota\, x_0)  + \Phi_{\cB}(t_1,X_{t_1}[\widetilde{\alpha}])\Big].
		\end{align*}
		Abbreviating $X_{t}[\widetilde{\alpha}]$ to $X_{t}$, by Lemma \ref{lem:Ito}, 
		\begin{align*}
			\epsilon \geq&\ \frac{1}{t_1-t_0}\bE\Big[\int_{t_0}^{t_1} \Big(L_{\cB}(X_t , \alpha_t) +\partial_t \Phi_{\cB}(t,X_t) + \langle\nabla \Phi_{\cB}(t,X_t ), b_\cB (X_t,\alpha_t) \rangle_{L^2(\cB)} \\
			& \qquad \qquad \qquad + \frac{\beta_C^2}{2}\Delta_{\cB} \Phi(t,X_t)+ \frac{\beta_F^2}{2}\Theta_{\cB} \Phi(t,X_t)\Big)dt \Big]\\
			\geq&\ \frac{1}{t_1-t_0}\bE\Big[\int_{t_0}^{t_1} \Big(\partial_t \Phi_{\cB}(t,X_t)  - H_{\cB}\big(X_t ,-\nabla\Phi_{\cB}(t,X_t )\big)+ \frac{\beta_C^2}{2}\Delta_{\cB} \Phi(t,X_t)+ \frac{\beta_F^2}{2}\Theta_{\cB} \Phi(t,X_t)\Big)dt \Big].
		\end{align*}
		Taking $\epsilon\rightarrow^+ 0$ and then $t_1\rightarrow^+ t_0$, setting $y_0 := \iota x_0$, 
		$$
		-\partial_t \Phi_{\cB}(t_0, y_0)  + H_{\cB}\big( y_0,-\nabla\Phi_{\cB}(t_0, y_0)\big)- \frac{\beta_C^2}{2}\Delta_{\cB} \Phi (t_0, y_0)- \frac{\beta_F^2}{2}\Theta_{\cB} \Phi (t_0, y_0) \geq 0.
		$$
		This shows that $\overline{V}$ is a free viscosity supersolution.
	\end{proof}
	
	By Propositions \ref{prop:subdynamic_programming}  and \ref{prop:dynamic_programming}, we establish the following theorem.
	\begin{theorem}
		Suppose that Assumption \textbf{A} holds. Then the value function  $\overline{V}$, defined in (\ref{eqn:value}), is a free viscosity solution.
	\end{theorem}
	
	\subsection{Hilbert Space Viscosity Solution} \label{subsec: Hilbert viscosity}
	
	An alternative definition is to make use of the Hilbert spaces in the following way. As will be explained below, this definition only applies when there is \emph{no} free individual noise, and thus throughout this section we assume that $\beta_F=0$.  Recall that for a function $V:[0,T]\times \Sigma_{d}^2 \rightarrow \bR$,  for each $\cA\in \bW$,  $V_\cA$ is  defined as in \eqref{defv}. 
	\begin{definition}\label{def:Hilbert_space_subsolution}
		We say that a $V \in \textup{USC}([0,T]\times \Sigma_{d}^2)$ is a \emph{Hilbert space viscosity subsolution} to (\ref{eqn:non_commutative}) if for every Von Neumann algebra $\cA\in \bW$, $V_\cA$ is a viscosity subsolution (in a Hilbert sense) to (\ref{eqn:von_Neumann}) on $L^2(\cA)_{sa}^d$. 
	\end{definition}
	{Recall that the free individual noise Laplacian \eqref{eqn:free_laplacian} is  defined for a tracial $W^*$-function, \emph{not} for a function on a fixed von Neumann algebra. Thus the above definition makes only sense in the absence of the free individual noise, i.e. $\beta_F=0.$}

	\begin{definition}\label{def:Hilbert_space_supersolution}
		We say that a $V \in \textup{LSC}([0,T]\times\Sigma_d^2)$ is a \emph{Hilbert space viscosity supersolution} to (\ref{eqn:non_commutative})  if for every $(t,\lambda)\in [0,T)\times \Sigma_d^2$, there exist a Von Neumann algebra $\cA\in \bW$ and $x\in L^2(\cA)_{sa}^d$ with $\lambda = \lambda_x$, such that there is a viscosity supersolution $U_\cA$ (in a Hilbert sense) to (\ref{eqn:von_Neumann}) on $L^2(\cA)_{sa}^d$ with $U_{\cA}(t,x) = V(t,\lambda)$.
	\end{definition}
	
	We say $V\in C([0,T]\times \Sigma_d^2)$ is a \emph{Hilbert space viscosity solution} if it is both a Hilbert space viscosity subsolution and supersolution. We will omit some details from the proof that the value function is a Hilbert space viscosity solution, as they can be found in the literature on stochastic control in Hilbert spaces and are mostly repetitive with our results for the intrinsic viscosity solution.
	
	\begin{proposition}\label{prop:ValueFunctionHilbertSpaceViscositySolution}
		Suppose that Assumption \textbf{A} holds. If there is no free noise (i.e. $\beta_F=0$), then the value function $\overline{V}$, defined in (\ref{eqn:value}), is a Hilbert space viscosity solution.
	\end{proposition}
	
	\begin{proof}
		We first establish that $\overline{V}$ is a Hilbert space viscosity subsolution.  This follows from the subdynamic programming principle (Proposition \ref{prop:subdynamic_programming}), along with an It\^{o} formula for the test functions $\Phi\in \chi_{\cA}$. These results are standard and also very similar to the arguments above.  
		
		To establish that $\overline{V}$ is a Hilbert space  viscosity supersolution, consider that, by a standard Hilbert space theory, for any  $\cA \in \bW$, the value function $\widetilde{V}_{\cA}$, defined in \eqref{eqn:l2_value}, is a viscosity solution on $L^2(\cA)_{sa}^d$. We fix $(t_0,\lambda_0)\in [0,T)\times \Sigma_d^2$ and consider a minimizing sequence of Von Neumann algebras, $\cA_i\in \bW$, states $x_i\in L^2(\cA)_{sa}^d$ such that $\lambda_{x_i} = \lambda_0$, and controls, $\tilde{\alpha}^i\in \mathbb{A}_{\cA,x_i}^{t_0,T}$ for the variational problem  \eqref{def:bar}.  By considering the free product of these Von Neumann algebras, $\cA = *(\cA_i)_{i=1}^\infty$, and then taking  $x_0\in \cA$ such that $\lambda_{x_0}=\lambda_0$,  we can restrict the control policies in $\cA$ to obtain
		$$
		\overline{V}(t_0,\lambda_0) =\widetilde{V}_{\cA}(t_0,x_0)= \inf_{\widetilde{\alpha}\in \bA_{\cA,x_0}^{t_0,T}}\Big\{ \bE\Big[\int_t^T L_{\cA}(X_s[\widetilde{\alpha}], \alpha_s)ds + g_{\cA}(X_T[\widetilde{\alpha}])\Big]\Big\}.
		$$
		Since {$\widetilde{V}_{\cA}$ is a viscosity supersolution of (\ref{eqn:von_Neumann}) on $L^2(\cA)_{sa}^d$ and touches $\bar{V}_{\cA}$ from below at $(t,\lambda_0)$, it follows that $\overline{V}_{\cA}$ is supersolution. This holds for all $(t_0,\lambda_0)$ so $\overline{V}$ is a  Hilbert space viscosity supersolution of (\ref{eqn:non_commutative}).}
	\end{proof}
	
	Using the comparison principle for the viscosity solutions to Hamilton-Jacobi equations on the (infinite-dimensional) Hilbert spaces (see \cite{crandall1985hamilton,crandall1986hamilton,lions1988viscosity,crandall1990viscosity,crandall1991viscosity,ishii1993viscosity,fabbri2017stochastic}), one can deduce  the comparison principle for the Hilbert space viscosity subsolution and supersolution.
	\begin{theorem}\label{thm:comparison}
		Suppose that Assumption \textbf{A} holds. If $V^1$ is a Hilbert space viscosity subsolution to (\ref{eqn:non_commutative}) in the sense of Definition \ref{def:Hilbert_space_subsolution} and $V^2$ is a Hilbert space viscosity supersolution to (\ref{eqn:non_commutative}) in the sense of Definition \ref{def:Hilbert_space_supersolution}, then $V^1(t,\lambda)\leq V^2(t,\lambda)$ for all $(t,\lambda)\in [0,T]\times \Sigma_{d}^2$.
		
		In particular, the value function $\overline{V}$,  defined in  \eqref{eqn:value}, is a unique Hilbert space viscosity solution to (\ref{eqn:non_commutative}).
	\end{theorem}
	
	\begin{proof}
		We first prove that for any fixed $\cA\in \bW$,  viscosity solutions on $L^2(\cA)_{sa}^d$ satisfy the comparison principle. We do this by checking the assumptions F6-F9 from \cite[Section 6]{ishii1993viscosity}. The result from \cite{ishii1993viscosity} allows for more general term and singular drifts and infinite-dimensional diffusions, which complicate the matter. The equation is expressed using a perturbation equation with a stronger norm, which we may take to be $h(x) = \frac{1}{2}\|x\|^2_{L^2(\cA)}$. The perturbed equation is 
		$$
		F_{\cA,\delta}^\pm \big(x, u(t,x), \nabla u(t,x),{\rm Hess}\, u(t,x)\big) = H_{\cA}\big(x, -\nabla u(t,x) \mp \delta\, x\big) - \frac{\beta_C^2}{2}\Delta_{\cA}u(t,x) \mp \frac{\beta_C^2\, d\, \delta}{2}.
		$$ 
		Unlike most cases considered in \cite{ishii1993viscosity}, this function is continuous at $\delta=0$, and we can readily check all the conditions for $\delta=0$. 
		\begin{itemize}
			\item Assumption F6 states monotonicity of dependence of $F_{\cA,\delta}^{\pm}$ on $u$. This assumption always holds as we have no dependence of $F_{\cA,\delta}^{\pm}$ on $u$ itself.
			\item Assumption F7 combines both a structural assumption on the diffusion and a quantitative estimate of the Hamiltonian. Since these terms are separate in our equation we may address them separately. For the diffusion, fix $\theta>1$, and for $A, B\in \text{BL}(L^2(\cA)_{sa}^d)$ such that for all $a,b\in L^2(\cA)_{sa}^d$ and some $\gamma>1$
			$$
			\langle a,A\, a\rangle_{L^2(\cA)} + \langle b,B\, b\rangle_{L^2(\cA)} \leq \theta\, \gamma\, \|a-b\|_{L^2(\cA)}^2.
			$$ 
			To satisfy the hypothesis we must show the monotonicity of the equation (up to an error, which is not needed).
			This follows from
			\begin{align*}
				-\frac{\beta_C^2}{2}\langle\mathbbm{1}_{\cA}, A\, \mathbbm{1}_{\cA}\rangle_{L^2(\cA)} - \frac{\beta_C^2}{2}\langle \mathbbm{1}_{\cA}, B\, \mathbbm{1}_{\cA}\rangle_{L^2(\cA)}\geq -\frac{\beta_C^2\, \theta\, \gamma}{2}\, \|\mathbbm{1}_\cA-\mathbbm{1}_\cA\|_{L^2(\cA)}^2=0.
			\end{align*}
			
			For the Hamiltonian, for $\gamma>1$ and $X,Y\in L^2(\cA)_{sa}^d$, we must show that
			$$
			H_{\cA}(X, -\gamma(X-Y)) - H_{\cA}(Y,-\gamma(X-Y))\geq -\omega_1\Big(\gamma\, \|X-Y\|^2_{L^2(\cA)} + \frac{1}{\gamma}\Big)
			$$
			for a function $\omega_1$ on $[0,\infty)$ with $\lim_{y\rightarrow^+0}\omega_1(y)=0$.
			This has been shown in Lemma \ref{lem:Hamiltonian} with
			$$
			\omega_1(y) = \max \Big\{\bar{C} + C_2^2, \frac{1}{2}\Big\}\, y.
			$$

			\item Hypothesis F8 states uniform continuity of the Hamiltonian and the diffusion coefficient on bounded sets. The uniform continuity of the Hamiltonian is shown in Lemma \ref{lem:Hamiltonian}, and our diffusion coefficient is constant. 
			\item This assumption F9 states convergence with respect to finite-dimensional approximations of the Hessian term. For the common noise, it is entirely in the one-dimensional space spanned by $\mathbf{1}_\cA$, so the one-dimensional approximation is exact, satisfying F9. 
		\end{itemize}
		Having met the hypotheses, we have the comparison principle for each fixed $\cA\in \bW$.
		

		~
		
		Now let us prove that $V^1 \le V^2.$
		By Definition \ref{def:Hilbert_space_supersolution}, for each $(t,\lambda)\in [0,T]\times \Sigma_{d}^2$, there exist a Von Neumann algebra $\cA\in \bW$, $x\in L^2(\cA)^d_{sa}$ with $\lambda_x=\lambda$, and a supersolution $U^2_{\cA}$ of (\ref{eqn:von_Neumann}) on  $L^2(\cA)^d_{sa}$ such that $U^2_{\cA}(t, x) = V^2(t,\lambda)$.  Also by Definition \ref{def:Hilbert_space_subsolution},  for such $\cA\in \bW$, the function $V^1_\cA$, defined as in \eqref{defv} through a function $V^1$, is a viscosity subsolution of (\ref{eqn:von_Neumann}).  Therefore, by the comparison principle  on $L^2(\cA)_{sa}^d$,  we have $V^1_\cA(t,x) \leq U^2_{\cA}(t,x)$ and thus
		\begin{align*}
			V^1(t,\lambda) = V^1_\cA(t,x) \leq U^2_{\cA}(t,x) = V^2(t,\lambda).
		\end{align*} 
		
		The comparison principle and Proposition \ref{prop:ValueFunctionHilbertSpaceViscositySolution} immediately imply the existence and  uniqueness of Hilbert space viscosity solutions.
	\end{proof}

	\section{Examples}\label{sec:examples}
	\subsection{Linear-quadratic setting}\label{sec:linear_quadratic}
	
	We first consider the setting of Assumption \textbf{B} when $L_{\cA}$ and  $g_{\cA}$ are both quadratic.  These examples admit exact solutions, although they show less non-commutativity due to the cyclic property of the trace.  
	
	We consider the case where
	\begin{align*}
		\bA_{\cA} = L^2(\cA)_{sa}^d,\quad b_{\cA}(X,\alpha) = \alpha,\quad  L_{\cA}(X,\alpha)= \frac{1}{2}\|\alpha\|_{L^2(\cA)}^2.
	\end{align*}
	For coefficients $g_{ij}^0,g_{ij}^1\in \bR$ such that  $g_{ij}^0 = g_{ji}^{0}$ and $g_{ij}^1 = g_{ji}^{1}$, let
	$$
	g_{\cA}(X) := \sum_{i=1}^d\sum_{j=1}^d g_{ij}^0\tau(X^i\, X^j)+ \sum_{i=1}^d\sum_{j=1}^d g_{ij}^1\tau(X^i)\tau(X^j),$$
	showcasing the two types of quadratic terms that may arise. More general linear-quadratic terms can be handled similarly. 
	
	The Hamilton-Jacobi equation has the form
	\begin{align} \label{hj1}
		-\partial_t U_{\cA}(t,X) + \frac{1}{2}\|\nabla U_{\cA}(t,X)\|^2_{L^2(\cA)} - \frac{\beta_C^2}{2}\Delta_\cA U(t,X) - \frac{\beta_F^2}{2}\Theta_\cA U(t,X) = 0,
	\end{align}
	with the terminal condition
	$$
	U_{\cA}(T,X) = g_{\cA}(X).
	$$
	
	We begin with a candidate solution of the form 
	$$
	{U}_{\cA}(t,X) = e(t) + \sum_{i=1}^d\sum_{j=1}^d a_{ij}^0(t)\tau(X^i\, X^j)+\sum_{i=1}^d\sum_{j=1}^d a_{ij}^1(t)\tau(X^i)\tau(X^j)
	$$
	such that $a_{ij}^0, a_{ij}^1\in \bR$ and $a_{ij}^0(t) = a_{ji}^{0}(t)$ and $a_{ij}^1 (t)= a_{ji}^{1}(t)$.
	Then we have 
		\begin{align*}
			\nabla^j U_{\cA}(t,X) = 2\sum_{i=1}^d (a_{ij}^0(t))\, X^i + 2\sum_{i=1}^d  (a_{ij}^1(t))\, \tau(X^i)\, \mathbf{1}_\cA.
		\end{align*}
		This yields
		$$
		\frac{1}{2}\|\nabla U_{\cA}(t,X)\|_{L^2(\cA)}^2 = 2\sum_{i=1}^d\sum_{j=1}^d\sum_{k=1}^d \Big({a_{ik}^0}(t){a_{kj}^0}(t)\tau(X^i\, X^j)+\big({a_{ik}^0}(t) + {a_{ik}^1}(t)\big)\, {a_{kj}^1}(t)\tau(X^i)\tau(X^j) \Big).
		$$
		The Hessian of $U_{\cA}$ is given by
		$$
		{\rm Hess}\, U_{\cA}(t,X)[A, B] = 2\, \sum_{i=1}^d\sum_{j=1}^d \Big({a_{ij}^0}(t)\tau(A^i\, B^j) + {a_{ij}^1}(t)\tau(A^i)\tau(B^j)\Big),
		$$
		implying that the common noise Laplacian is  given by
		$$
		\Delta_{\cA} U(t,X) =   2\,  \sum_{i=1}^d\sum_{j=1}^d\big( {a_{ij}^0}(t) + {a_{ij}^1}(t)\big).
		$$
		The free individual noise Laplacian is  
		$$
		\Theta_{\cA} U(t,X) =   2\,  \sum_{i=1}^da_{ii}^0(t).
		$$

		Computing the coefficient of $\tau(X^i\, X^j)$ term  in \eqref{hj1},
		$$
		-(a_{ij}^0)'(t) + 2\, \sum_{k=1}^d {a_{ik}^0}(t){a_{kj}^0}(t)=0,
		$$
		with $a_{ij}^0(T) = g_{ij}^0$. Computing the coefficient of 
		$\tau(X^i)\, \tau(X^j)$ term  in \eqref{hj1},
		$$
		-(a_{ij}^1)'(t) + 2\, \sum_{k=1}^d \big({a_{ik}^0}(t)+{a_{ik}^1}(t)\big){a_{kj}^1}(t)=0,
		$$
		with $a_{ij}^1(T) = g_{ij}^1$.  The constant term  in \eqref{hj1} reads as
		$$
		-e'(t) -   \beta_C^2\, \sum_{i=1}^d\sum_{j=1}^d \big({a_{ij}^0}(t) + {a_{ij}^1}(t)\big) -   \beta_F^2\, \sum_{i=1}^da_{ii}^0(t)=0,
		$$
		and $e(T)=0$.

		We can also easily express this solution on the space of non-commutative laws. 
		Recall that we consider $\lambda$ to be a map from the space of polynomials of $d$ complex variables to $\bC$. For $i=1,\cdots,d$, define  $q^i$ to be the canonical polynomial $q^i(X):= X^i$.
		
		The value function of non-commutative laws is then given by
		$$
		V(t,\lambda) := e(t) + \sum_{i=1}^d \sum_{j=1}^d\Big(a_{ij}^0(t)\, \lambda\big(q^i\, q^j\big) + a_{ij}^1(t)\, \lambda\big(q^i\big)\, \lambda\big(q^j\big)\Big).
		$$

		\subsection{Eikonal example}
		
		Next, we consider the free Eikonal equation to illustrate a nonsmooth solution. We set 
		\begin{align*}
			\bA_{\cA} = \big\{\alpha\in L^2(\cA)_{sa}^d:\|\alpha\|_{L^2(\cA)}\leq 1\big\},\quad b_{\cA}(X,\alpha) = \alpha,\quad L_{\cA}(X,\alpha)= 1,\quad \beta_C=\beta_F = 0.
		\end{align*}
		We consider the terminal cost of 
		$$
		g_{\cA}(X) = d_W(\bar{\lambda},\lambda_X),
		$$
		where $\bar{\lambda}$ is a given non-commutative law and $d_W$ denotes the non-commutative $L^2$-Wasserstein distance (recall \eqref{wass} for the definition). This example satisfies Assumption \textbf{A}. However, the terminal condition is not $E$-convex, violating Assumption \textbf{B}.
		
		
		{We claim that the following  function $\overline{V}$ is a  free viscosity solution:}
		\begin{align} \label{555}
			\overline{V}(t,\lambda) := \max\{T-t, d_W(\bar{\lambda},\lambda)\}.
		\end{align}
		In order to deduce that $\overline{V}$ is a free viscosity subsolution, we may observe that the maximum of two free viscosity subsolutions is a free viscosity subsolution and check the terms individually. For $\cA \in \bW$ and $X\in \cA$, set
		\begin{align*}
			U^1_{\cA}(t,X) := T-t,\quad U^2_{\cA}(t,X): = d_W(\bar{\lambda},\lambda_X),\quad U_\cA := \max \{U^1_{\cA},U^2_{\cA}\}. 
		\end{align*}
		Firstly, for any  $\cA \in \bW$, $U^1_{\cA}$ 
		satisfies $U^1_{\cA}(T,X) = 0\leq g_{\cA}(X)$ 
		and 
		$$
		-\partial_t U^1_{\cA}(t,X) + \|\nabla U^1_{\cA}(t,X)\|_{L^2(\cA)} - 1 = 0. 
		$$
		{{Since $U^1_\cA$ is a smooth function, we verify that  $(U^1_{\cA})_{\cA \in \bW}$  is a viscosity solution.}}
		
		Now we show that $(U^2_{\cA})_{\cA \in \bW}$  is a free viscosity subsolution.  Since $X\mapsto d_W(\bar{\lambda},\lambda_X)$ is 1-Lipschitz, {{if $(\phi_\cA)_{\cA \in \bW}$ is a tracial function which touches $(U^2_\cA)_{\cA \in \bW}$ from above at $(t_0,\lambda_0)$, then   for all $\cA \in \bW$, $t \ge 0$ and $X, X_0 \in L^2(\cA)^d_{\rm sa}$ with $\lambda_0=\lambda_{X_0},$
				\begin{align*}
					\phi_\cA(t_0, X_0)-\phi_\cA(t, X) \leq  d_W(\bar{\lambda},\lambda_{X_0})-d_W(\bar{\lambda},\lambda_X) 
					\le \|X-X_0\|_{L^2(\cA)}.
				\end{align*}
				This implies that $\partial_t \phi_\cA(t_0, X_0)=0$ and $\|\nabla \phi_\cA(t_0,X_0)\|_{L^2(\cA)}\leq 1$, verifying that $(U^2_{\cA})_{\cA \in \bW}$ is a viscosity subsolution.}}

		To show that $\overline{V}$ is a viscosity supersolution, let $\lambda_1$  be an arbitrary non-commutative law in $\Sigma_d^2$. Suppose that $\phi=(\phi_\cA)_{\cA \in \bW}$ is a tracial function which touches $ (U_\cA)_{\cA \in \bW}$  from below at $(t_1, \lambda_1)$  with $t_1 \in [0,T)$. Let  $\cA$ be a von Neumann algebra such that there exist $X_1,\bar{X} \in L^2(\cA)_{sa}^d$ satisfying $\lambda_{\bar{X}}=\bar{\lambda}$, $\lambda_{X_1}=\lambda_1$ and $d_W(\bar{\lambda},\lambda_1) = \|\bar{X}-X_1\|_{L^2(\cA)}$.

		In the case $T-t_1 \ge d_W( \bar \lambda, \lambda_1)$, we have
		$\phi_{\cA}(t_1,X_1)  =  T-t_1$ and $\phi_{\cA}(t,X_1) \le T-t$. Thus $\partial_t \phi_{\cA}(t_1,X_1) = -1$, implying
		\begin{equation}\label{eq:june30.2024.1}
			-\partial_t \phi_{\cA}(t_1,X_1) + \|\nabla \phi_{\cA}(t_1,X_1)\|_{L^2(\cA)} - 1  \ge  0.
		\end{equation}
		
		In the other case $T-t_1 < d_W( \bar \lambda, \lambda_1)$, we have $U_\cA = U^2_\cA$  near $(t_1, X_1)$. Thus   $\phi_{\cA}(t_1,X_1) = d_W( \bar \lambda, \lambda_1) $ and $\phi_{\cA}(t,X_1) \le  d_W( \bar \lambda, \lambda_1)$ for any $t$ near $t_1$, implying that $\partial_t \phi_{\cA}(t_1,X_1) = 0.$ Also note that  $\lambda_1 \not =\bar \lambda$ and so, 
		$\|\bar X-\cdot\|_{L^2(\cA)}$
		is differentiable at $X_1$. Since $\phi_\cA (t_1,\cdot)  -\|\bar X - \cdot\|_{L^2(\cA)}$ has the maximum at $X_1$, {{its  derivative exists at $X_1$ and vanishes there. As the $L^2(\cA)$-norm of a derivative of   $\|\bar X -  \cdot\|_{L^2(\cA)}$ is equal to 1 at any differentiable point, we have   $ \|\nabla \phi_{\cA} (t_1,X_1)\|_{L^2(\cA)} = 1,$ implying \eqref{eq:june30.2024.1}.
				Therefore we conclude that $\overline{V}$ is a viscosity supersolution.}}
		
		~

		
		To illustrate this example further, suppose that $d=1$ and $\bar{\lambda}$ has a spectral measure $\mu\in \mathcal{P}_2(\bR)$, where $\mathcal{P}_2(\bR)$ denotes the set of probability measures on $\bR$ having a finite second moment.  Then, on $\cA=\bC$, for any real number $x$ (i.e. a self-adjoint element in $\bC$), we have 
		$$
		\overline{V}_{\bC}(t,x) =  \max\Big\{T-t,\inf_{f_{\sharp}Leb[0,1] = \mu}\int_0^1 |x-f(\omega)|d\omega\Big\}=  \max\Big\{T-t,\int_{\bR}|x-y|d\mu(y)\Big\}.
		$$
		We get
		$$
		\partial_x\, \int_{\bR}|x-y|\mu(y) = -\mu\big( (x,\infty)\big) + \mu\big((-\infty,x)\big).
		$$
		Clearly, we have $-\partial_t \overline{V}_{\bC}(t,x)+|\partial_x\overline{V}_{\bC}(t,x)|\leq 1$, but the supersolution property does not hold everywhere so long as $\mu$ is not supported on a point.
		{This particular example shows that, although $\overline{V}$ defined in \eqref{555} is always a \emph{free} viscosity solution, for a \emph{fixed} von Neumann algebra $\cA,$ the function $\overline{V}_\cA$ constructed from $\overline{V}$ via \eqref{defv} may {not} be a viscosity supersolution on $L^2(\cA)_{sa}$ in a usual sense.}

		\subsection{Controlled von Neumann equation}\label{sec:schrodinger}
		
		We consider the dynamics corresponding to the evolution of $d$-quantum system, where the state plays the role of a density matrix (i.e. the one satisfying $\tau(X)=1$ along with $X\geq 0$) and the control plays the role of the Hamiltonian.  This is given by the drift, with the commutator is applied element-wise,
		$$
		b_{\cA}(X,\alpha) = {\rm i}[X,\alpha].
		$$
		Assume that there is neither common noise nor free individual noise, and consider the quadratic cost
		$$
		L_{\cA}(X,\alpha) = \frac{1}{2}\|\alpha\|_{L^2(\cA)}^2
		$$
		along with a general terminal cost $g_{\cA}(X)$.
		
		The drift $b_\cA$ defined above is a tracial vector-field and satisfies $E$-linearity in the control and \eqref{eq:june28.2024.1} for $\bar C=0$, but does not satisfy the uniform continuity or Lipschitz continuity conditions in  Assumption \textbf{A}.  Instead, it would make sense to restrict a state space to $L^\infty(\cA)_{sa}^d$, which is preserved under the dynamics.
		
		A seemingly remarkable property is that the optimal control $(\alpha_t)_t$ turns out be constant along trajectories. The Hamiltonian is given by
		\begin{align*}
			H_{\cA}(X,P) =&\ \sup_{\alpha\in L^2(\cA)_{sa}^d}\big\{\langle {\rm i}[X,\alpha], P\rangle - \frac{1}{2}\|\alpha\|_{L^2(\cA)}^2 \big\}\\
			=&\ \sup_{\alpha\in L^2(\cA)_{sa}^d}\big\{\langle {\rm i}[P, X], \alpha\rangle - \frac{1}{2}\|\alpha\|_{L^2(\cA)}^2 \big\}\\
			=&\ \frac{1}{2} \|[P, X]\|^2_{L^2(\cA)}.
		\end{align*}
		
		
		Thus, the minimization can be done of the terminal cost in the form of a non-commutative Hopf-Lax formula:
		\begin{align}\label{eqn:nc_Hopf_Lax}
			\widetilde{V}_{\cA}(t,X) = \inf_{\alpha\in L^2(\cA)^d_{sa}}\Big\{\frac{(T-t)}{2}\|\alpha\|_{L^2(\cA)}^2 + g_{\cA}\big(e^{-{\rm i}\alpha (T-t)}\, X\, e^{{\rm i}\alpha (T-t)}\big)\Big\}.
		\end{align}
		
		The following theorem partly verifies that the optimal control is constant in time if the minimizer exists.
		
		\begin{theorem}
			Suppose that $\cA\in \bW$, $    \bA_{\cA} = L^2(\cA)_{sa}^d$ and $b_{\cA}(X,\alpha) = {\rm i}[X,\alpha]$ with neither common noise nor free individual noise. Let $(\alpha_t)_{t\in [0,T]}\in\mathbb{A}_{\cA}$ be a minimizer of $\widetilde V_\cA$ in \eqref{eqn:l2_value} with the quadratic Lagrangian $
			L_{\cA}(X,\alpha) = \frac{1}{2}\|\alpha\|_{L^2(\cA)}^2,
			$ where $(X_t)_t\in L^\infty([0,T];L^\infty(\cA)^d_{sa})\cap H^1([0,T];L^2(\mathcal{A})^d_{sa})$ is determined by the dynamics (\ref{eqn:common_noise}). Then $(\alpha_t)_{t\in [0,T]}$ is constant in time.
		\end{theorem}
		\begin{proof}
			We consider variations corresponding to a curve $\xi\in C^1([0,T];L^\infty(\mathcal{A})^d_{sa})$ with $\xi_0=\xi_T=0$.
			We will show that for any $s\in \mathbb{R}$, setting
			$$
			\beta_{t,s} := \int_0^1 e^{-{\rm i}\, r\, s\, \xi_t}\, \xi_t'\, e^{{\rm i} \,r\, s\, \xi_t}dr,\qquad t \in [0,T],
			$$
			the trajectory 
			$$
			t\mapsto  \widehat{X}_{t,s} := e^{-{\rm i}\, s\, \xi_t}X_t\, e^{{\rm i}\, s\, \xi_t}.
			$$
			is a solution to
			\begin{align}\label{eqn:variation}
				\frac{d}{dt}\widehat{X}_{t,s}  = -{\rm i}\, [s\, \beta_{t,s}+e^{-{\rm i}\, s\, \xi_t}\alpha_t\, e^{{\rm i}\, s\, \xi_t}, \widehat{X}_{t,s}].
			\end{align}
			In other words, we have variations of the control (parametrized by $s\in \mathbb{R}$) of the form
			$$
			t\mapsto \widehat{\alpha}_{t,s} := s\, \beta_{t,s}+e^{-{\rm i}\, s\, \xi_t}\alpha_t\, e^{{\rm i}\, s\, \xi_t}
			$$
			that result in solutions with $\widehat{X}_{0,s}=X_0$ and $\widehat{X}_{T,s} = X_T$ for any $s\in \mathbb{R}$.
			
			To verify that $(\widehat{X}_{t,s})_t$ is a solution to \eqref{eqn:variation}, first we use Duhamel's formula to compute 
			\begin{align*}
				\frac{d}{dt} e^{-{\rm i}\, s\, \xi_t} =&\ \int_0^1 e^{-{\rm i} \, r\, s\, \xi_t}(-{\rm i}\, s\, \xi_t')e^{-{\rm i} (1-r)\, s\, \xi_t}dr\\
				=&\ -{\rm i}\, s\, \beta_{t,s}\, e^{{-\rm i}\, s\, \xi_t}
			\end{align*}
			and
			\begin{align*}
				\frac{d}{dt} e^{{\rm i}\, s\, \xi_t} =&\ \int_0^1 e^{{\rm i} r\, s\, \xi_t}({\rm i}\, s\, \xi_t')e^{{\rm i} (1-r)\, s\, \xi_t}dr\\
				=&\ \int_0^1 e^{{\rm i}\, (1-u)\, s\, \xi_t}({\rm i}\, s\, \xi_t')e^{{\rm i} \, u\, s\, \xi_t}du\\
				=&\ {\rm i}\, s\, e^{{\rm i}\, s\, \xi_t} \beta_{t,s}.
			\end{align*}
			We then have that
			\begin{align*}
				\frac{d}{dt} \widehat{X}_{t,s} = -{\rm i}\, s\, \beta_{t,s}\, e^{{-\rm i}\, s\, \xi_t} X_t\, e^{{\rm i}\, s\, \xi_t} + e^{{-\rm i}\, s\, \xi_t} (-{\rm i}[\alpha_t,X_t])\, e^{{\rm i}\, s\, \xi_t} + {\rm i}\, s\,  e^{{-\rm i}\, s\, \xi_t} X_t\, e^{{\rm i}\, s\, \xi_t}\, \beta_{t,s},
			\end{align*}
			which agrees with (\ref{eqn:variation}) seeing as
			$$
			e^{{-\rm i}\, s\, \xi_t} (-{\rm i}[\alpha_t,X_t])\, e^{{\rm i}\, s\, \xi_t} = -{\rm i}[e^{{-\rm i}\, s\, \xi_t} \alpha_t\, e^{{\rm i}\, s\, \xi_t},e^{{-\rm i}\, s\, \xi_t} X_t\, e^{{\rm i}\, s\, \xi_t}].
			$$

			Now the variation of the cost at $s=0$ is given by
			\begin{align*}
				&\ \frac{d}{ds}\Big|_{s=0}\int_0^T \frac{1}{2}\|s\, \beta_{t,s} + e^{-{\rm i}\, s\, \xi_t}\alpha_t\, e^{-{\rm i}\, s\, \xi_t}\|_{L^2(\cA)}^2dt\\
				=&\ \frac{d}{ds}\Big|_{s=0}\int_0^T\Big( \frac{s^2}{2}\|\beta_{t,s}\|_{L^2(\cA)}^2 + s\, \langle \beta_{t,s},   e^{-{\rm i}\, s\, \xi_t}\alpha_t\, e^{-{\rm i}\, s\, \xi_t}\rangle_{L^2(\cA)} + \frac{1}{2}\|\alpha_t\|_{L^2(\cA)}^2\Big)dt\\
				=&\ \int_0^T\langle \xi_t',\alpha_t\rangle_{L^2(\cA)} dt.
			\end{align*}
			By the optimality of $(\alpha_t)_{t\in [0,T]}$, the above quantity is zero for all $\xi\in C^1([0,T];L^2(\cA)^d_{sa})$ with $\xi_0=\xi_T=0$, which implies that $(\alpha_t)_{t\in [0,T]}$ is constant in time.
		\end{proof}
		
		\appendix
		
		\section{Differential equations} \label{subsec: diff eq appendix}
		
		\subsection{Differentiation of vector-valued functions}
		Throughout this section, we interchangeably use the notations $\norm{}_2$ and $\norm{}_{L^2(\cA)}$ for a tracial von Neumann algebra $\cA$.
		\begin{lemma}[Facts about vector-valued $W^{1,2}$ space]
			Let $\cA$ be a tracial von Neumann algebra.  For any function $A: [0,T] \to L^2(\cA)$, define
			\[
			m(A) := \sup_{0 = t_0 < t_1 < \dots < t_n = T} \sum_{j=1}^n \norm{A_{t_j} - A_{t_{j-1}}}_2^2.
			\]
			If $\alpha: [0,T] \to L^2(\cA)$ is Bochner measurable with $\int_0^T \norm{\alpha_t}_2^2\,dt < \infty$, then $A_t := \int_0^t \alpha_s\,ds$ is well-defined and continuous (in $t$) which satisfies $m(A) = \int_0^T \norm{\alpha_t}_2^2\,dt$.  Conversely, if $t\mapsto A_t$ is continuous with $m(A) < \infty$ and $A_0 = 0$, then there exist a measurable $\alpha: [0,T] \to L^2(\cA)$   such that $A_t = \int_0^t\alpha_s\,ds$ for $t\in [0,T]$ and $\int_0^T \norm{\alpha_t}_2^2\,dt = m(A)$.
		\end{lemma}
		
		\begin{proof}
			First, let $\alpha: [0,T] \to L^2(\cA)$ be given with $\int_0^T \norm{\alpha_t}_2^2\,dt < \infty$.  Let $0 \leq a < b \leq T$ and $f \in L^2(\cA)$.  Then
			\begin{align*}
				\int_a^b |\ip{\alpha_t,f}_2|\,dt \le  \int_a^b \norm{\alpha_t}_2 \norm{f}_2\,dt & \le  \left( \int_a^b \norm{\alpha_t}_2^2\,dt \right)^{1/2} \left( \int_a^b \norm{f}_2^2\,dt \right)^{1/2} \\
				&= (b - a)^{1/2} \left( \int_a^b \norm{\alpha_t}_2^2\,dt \right)^{1/2} \norm{f}_2.
			\end{align*}
			Thus, $f \mapsto \int_a^b \ip{\alpha_t,f}_2\,dt$ defines a bounded linear functional on $L^2(\cA)$ and hence there is a unique vector, which we denote by $\int_a^b \alpha_t\,dt$, such that
			\[
			\ip*{\int_a^b \alpha_t\,dt, f}_2 = \int_a^b \ip{\alpha_t,f}\,dt \quad \text{ for any } f \in L^2(\cA).
			\]
			Furthermore,
			\[
			\norm*{\int_a^b \alpha_t\,dt}_2^2 = \sup_{\norm{f}_2 \leq 1} \left|\int_a^b \ip{\alpha_t,f}_2\,dt \right|^2 \leq (b - a) \int_a^b \norm{\alpha_t}_2^2\,dt.
			\]
			Let $A_t := \int_0^t \alpha_s\,ds$.  Fix a partition $0 = t_0 < t_1 < \dots < t_n = T$, and note that
			\[
			\sum_{j=1}^n \frac{1}{t_j-t_{j-1}} \norm{A_{t_j} - A_{t_{j-1}}}_2^2 \leq \sum_{j=1}^n \frac{1}{t_j-t_{j-1}} \norm*{\int_{t_{j-1}}^{t_j} \alpha_t\,dt}_2^2 \leq \sum_{j=1}^n \int_{t_{j-1}}^{t_j} \norm{\alpha_t}_2^2\,dt.
			\]
			and so $m(A) \leq \int_0^T \norm{\alpha_t}_2^2\,dt$.  To show the opposite inequality, since $\alpha_t$ is in the Bochner $L^2$ space, it can be approximated by linear combinations of indicator functions times fixed vectors, and we can also arrange approximations by indicator functions of intervals.  Hence, given $\epsilon > 0$, there exists some partition $0 = t_0 < t_1 < \dots < t_n = T$ and  vectors $f_1$, \dots, $f_n \in L^2(\cA)$ such that
			\[
			\left( \int_0^T \norm{\alpha_t - \beta_t}_2^2\,dt \right)^{1/2} < \epsilon, \text{ where } \beta_t := \sum_{j=1}^n \mathbf{1}_{(t_{j-1},t_j]}(t) f_j.
			\]
			Let $B_t := \int_0^t \beta_s\,ds$. Note that
			\[
			m(B) \geq \sum_{j=1}^n \frac{1}{t_j-t_{j-1}} \norm*{\int_{t_{j-1}}^{t_j} \norm{\beta_t}\,dt}_2^2 = \sum_{j=1}^n (t_j-t_{j-1}) \norm{f_j}^2 = \int_0^T \norm{\beta_t}_2^2\,dt.
			\]
			It is not hard to check that $m(A)^{1/2}$ is a seminorm (allowing the value $+\infty$) on functions $[0,T] \to L^2(\cA)$.  Hence,
			\begin{align*}
				m(A)^{1/2} \geq m(B)^{1/2} - m(A-B)^{1/2} &\geq \norm{\beta}_{L^2([0,T],L^2(\cA))} - \norm{\alpha - \beta}_{L^2([0,T],L^2(\cA))} \\
				&\geq \norm{\alpha}_{L^2([0,T],L^2(\cA))} - 2 \norm{\alpha - \beta}_{L^2([0,T],L^2(\cA))} \\
				&\geq \norm{\alpha}_{L^2([0,T],L^2(\cA))} - 2 \epsilon.
			\end{align*}
			Since $\epsilon>0$ was arbitrary, we have $m(A) = \int_0^T \norm{\alpha_t}_2^2\,dt$ as desired.

			~

			Next, suppose we are given $A$ with $m(A) < \infty$.  We will construct $\alpha$ again using the self-duality of Hilbert spaces.  Let $H_0 \subseteq L^2([0,T],L^2(\cA))$ (Bochner $L^2$ space) be the linear span of functions of the form $\mathbf{1}_{(a,b]}(t) f$ for $f \in L^2(\cA)$.  We define a linear map $\ell: H_0 \to \bC$ by
			\[
			\ell: \sum_{j=1}^n \mathbf{1}_{(s_j,t_j]} f_j \mapsto \sum_{j=1}^n \ip{A_{t_j} - A_{s_j},f_j}_2;
			\]
			It is straightforward to check that this is well-defined, i.e.\ independent of the decomposition of the given function on $H_0$.  To check that the map  $\ell$ is bounded, we consider a given function $f$ in $H_0$; let $\{t_0,\dots,t_n\}$ be a partition that contains all the endpoints of intervals in the given decomposition of $f$, and express $f$ as $\sum_{j=1}^n \mathbf{1}_{(t_{j-1},t_j]} f_j$.  Then
			\begin{align*}
				|\ell(f)| &\leq \sum_{j=1}^n |\ip{A_{t_j} - A_{t_{j-1}},f_j}_2| \\ &\leq \sum_{j=1}^n \frac{1}{(t_j-t_{j-1})^{1/2}} \norm{A_{t_j} - A_{t_{j-1}}}_2 (t_j - t_{j-1})^{1/2} \norm{f_j}_2 \\
				&\leq \left( \sum_{j=1}^n \frac{1}{t_j - t_{j-1}} \norm{A_{t_j} - A_{t_{j-1}}}_2^2 \right)^{1/2} \left( \sum_{j=1}^n (t_j - t_{j-1}) \norm{f_j}_2^2 \right)^{1/2} \\
				&\leq m(A)^{1/2} \norm{f}_{L^2([0,T],L^2(\cA))}.
			\end{align*}
			Since $\ell$ is bounded, it extends from $H_0$ to all of $L^2([0,T],L^2(\cA))$, and there exists $\alpha \in L^2([0,T],L^2(\cA))$ with $\ell(f) = \ip{f,\alpha}_{L^2([0,T],L^2(\cA))}$, and $\norm{\alpha}_{L^2([0,T],L^2(\cA)} \leq m(A)^{1/2}$.  Then to check that $A_t = \int_0^t \alpha_s\,ds$, note that for $f \in L^2(\cA)$, we have
			\[
			\ip*{\int_0^t \alpha_s\,ds,f}_{L^2(\cA)} = \ip{\alpha, \mathbf{1}_{[0,t]} f}_{L^2([0,T],L^2(\cA))} = \ell(f) = \ip{A_t - A_0, f}_{L^2(\cA)}.
			\]
		\end{proof}
		
		\begin{lemma}
			The following statements hold.

			1. Let $\alpha: [0,T] \to L^2(\cA)$ be Bochner measurable  and define $A_t := \int_0^t \alpha_s\,ds$ for $t\in [0,T]$.  Then $t \mapsto \norm{\alpha_t}_{L^\infty(\cA)}$ is measurable and
			\[
			\norm{\alpha}_{L^\infty([0,T],L^\infty(\cA))} = \sup_{0 \leq s < t \leq T} \frac{\norm{A_t - A_s}_{L^\infty(\cA)}}{t-s}.
			\]
			(Here both sides may be infinite.)  
			
			2. For any $A: [0,T] \to L^2(\cA)$,
			\[
			m(A)^{1/2} \leq T^{1/2} \sup_{0 \leq s < t \leq T} \frac{\norm{A_t - A_s}_{L^\infty(\cA)}}{t-s},
			\]
			so in particular if $A$ is Lipschitz in $\norm{\cdot}_{L^\infty(\cA)}$, then there exists a compatible $\alpha: [0,T] \to L^2(\cA)$  (i.e. $A_t := \int_0^t \alpha_s\,ds$).
		\end{lemma}
		
		\begin{proof}
			Note that
			\[
			\norm{x}_{L^\infty(\cA)} = \sup_{a, b \in \cA, \norm{a}_2, \norm{b}_2 \leq 1} |\tau_{\cA}(axb)|.
			\]
			This shows that $x \mapsto \norm{x}_{L^\infty(\cA)}$ is a lower semi-continuous function on $L^2(\cA)$, both in the norm topology and the weak-$*$ topology.  In particular, it is a Borel-measureable function, and hence $t \mapsto \norm{\alpha_t}_{L^\infty(\cA)}$ is measureable.
			
			Now note that for $s < t$, we have
			\[
			\norm{A_t - A_s}_{L^\infty(\cA)} = \norm*{\int_s^t \alpha_u\,du }_{L^\infty(\cA)} \leq \int_s^t \norm{\alpha_u}_{L^\infty(\cA)}\,du \leq (t - s) \norm{\alpha}_{L^\infty([0,T],L^\infty(\cA))}.
			\]
			For the opposite direction, first note that the image of $\alpha$ is contained in $L^2(\cA_0)$ for some $L^2$-separable von Neumann subalgebra of $\cA$.  Furthermore, choose some countable set $C \subseteq \cA_0$ that is dense in the unit ball of $L^2(\cA)$, so that
			\[
			\norm{\alpha_t}_{L^\infty(\cA)} = \sup_{a, b \in C} |\tau_{\cA}(a \alpha_t b)|
			\]
			Note that for $a, b \in C$,
			\[
			\int_s^t \tau_{\cA}(a \alpha_u b)\,du = \int_s^t \ip{\alpha_u, (ba)^*}\,du = \tau_{\cA}(a(A_t - A_s)b).
			\]
			Write $L$ for the Lipschitz norm of $t \mapsto A_t$ with respect to $\norm{\cdot}_{L^\infty(\cA)}$.  By the Lebesgue differentiation theorem, we have for a.e. $t$ that
			\[
			|\tau_{\cA}(a \alpha_t b)| = \lim_{\epsilon \to 0} |\tau_{\cA}(a(A_{t+\epsilon} - A_t) b)| \leq L \norm{a}_{L^2(\cA)} \norm{b}_{L^2(\cA)} \leq L.
			\]
			Since $C$ is countable, this holds for all $a, b \in C$ simultaneously, for a.e.\ $t \in [0,T]$.  Hence, the $L^\infty[0,T]$-norm of $\norm{\alpha_t}_{L^\infty(\cA)}$ is bounded by $L$.
		\end{proof}
		
		\subsection{Stochastic integrals and differential equations} \label{apx:SDE_theory}
		
		First, we should define what the stochastic integral means.  For our purposes, it should be sufficient to assume \textbf{A}, and we really don't have to worry about any free stochastic integrals.

		To complete the proof of Theorem \ref{thm:well_posedness}, we should show the existence of an adapted solution using a Picard iteration scheme.
		A simple version of the proof uses the weighted norm on $C([t_0,t_1];L^2(\Omega,\mathcal{F},\mathbb{P};L^2(\cA)^d_{sa}))$ given by, for some constant $\gamma>0$ chosen later,
		$$
		\|X\|_{C_\gamma}^2 := \sup_{t\in [t_0,t_1]} e^{-2\gamma (t-t_0)} \bE [\|X_t\|^2_{L^2(\cA)}].
		$$

		For $X\in C([t_0,t_1];L^2(\Omega,\mathcal{F},\mathbb{P};L^2(\cA)^d_{sa}))$,  define $\Phi(X) = (\Phi_t(X))_{t \in [t_0,t_1]}$ as
		$$
		\Phi_t(X) := \int_{t_0}^tb_{\cA}(X_s,\alpha_s)ds + \beta_C\, \mathbbm{1}_{\cA}  W^0_t + \beta_F\, S_t.
		$$
		Then for $X,Y\in C([t_0,t_1];L^2(\Omega,\mathcal{F},\mathbb{P};L^2(\cA)^d_{sa}))$,
		$$
		\Phi_t(X)-\Phi_t(Y) = \int_{t_0}^t\big(b_{\cA}(X_s,\alpha_s) - b_{\cA}(Y_s,\alpha_s)\big)ds.
		$$
		From this we obtain
		\begin{align*}
			&\ \frac{d}{dt}e^{-2\gamma (t-t_0)}\mathbb{E}\big[\|\Phi_t(X)-\Phi_t(Y)\|_{L^2(\cA)}^2\big]\\
			=&\ e^{-2\gamma(t-t_0)}\mathbb{E}\Big[2\langle b_{\cA}(X_t,\alpha_t) - b_{\cA}(Y_t,\alpha_t), \Phi_t(X)-\Phi_t(Y)\rangle_{L^2(\cA)} -2\gamma \|\Phi_t(X)-\Phi_t(Y)\|^2_{L^2(\cA)}\Big]  \\
			\leq&\ e^{-2\gamma (t-t_0)}\, \mathbb{E}\Big[\frac{\bar{C}^2}{\gamma}\, \|X_{t}-Y_{t}\|_{L^2(\cA)}^2 -\gamma\,  \|\Phi_t(X)-\Phi_t(Y)\|_{L^2(\cA)}^2\Big].
		\end{align*}
		We end up with 
		\begin{align*}
			\|\Phi(X)-\Phi(Y)\|_{C_\gamma}^2 \leq&\ \frac{\bar{C}^2}{\gamma^2} \|X-Y\|_{C_\gamma}^2,
		\end{align*}
		making $\Phi$ a contraction map w.r.t. $C_\gamma$-norm when $\gamma>\bar{C}$. At each time $t$, the fixed-point iteration convergences in $L^2(\cA_t)$, making the solution freely adapted.
		
		We may also directly calculate that for almost every $t$,
		\begin{align*}
			\frac{d}{dt}\mathbb{E}\big[\|X_t - x_0\|_{L^2(\cA)}^2\big] =&\ \lim_{h\rightarrow^+0}h^{-1}\mathbb{E}\big[\|X_{t+h} - x_0\|_{L^2(\cA)}^2 - \|X_{t} - x_0\|_{L^2(\cA)}^2\big]\\
			=&\ \lim_{h\rightarrow^+0}h^{-1}\mathbb{E}\big[\|X_{t+h} - X_t\|_{L^2(\cA)}^2 + 2\langle X_{t+h} - X_t, X_t-x_0\rangle_{L^2(\cA)}\big] \\
			=&\ \lim_{h\rightarrow^+0}\mathbb{E}\big[ h^{-1}\beta_C^2\, d\, (W^0_{t+h}-W^0_t)^2 + \beta_F^2 \|S_{t+h} - S_t\|^2_{L^2(\cA)}  \\ &\ \quad  \quad \quad +  2\langle h^{-1}\, \int_t^{t+h}b_{\cA}(X_s,\alpha_s)ds, X_t-x_0\rangle_{L^2(\cA)}\big]\\
			=&\ d\, (\beta_C^2 + \beta_F^2) + \mathbb{E}\big[\langle b_{\cA}(X_t,\alpha_t), X_t-x_0\rangle_{L^2(\cA)}\big]. 
		\end{align*}
		
		\section{Computation of Free Laplacian} \label{apx:free_laplacian}
		
		Here we describe the free Laplacian for special test functions (such as polynomials) in terms of Voiculescu's non-commutative derivatives.  These computations are well-established in free probability.  See e.g. \cite[\S 3]{driver2013large}, \cite[\S 14.1]{jekel2020evolution}, \cite[\S 4.3]{jekel2022tracial}.
		
		We recall that $\operatorname{NCP}_d$ denotes the $*$-algebra of non-commutative polynomials in $d$ self-adjoint variables.  Because Voiculescu's free difference quotient maps $\operatorname{NCP}_d$ into the tensor product space $\operatorname{NCP}_d \otimes \operatorname{NCP}_d$, we recall the following definitions.
		
		\subsection{Tensor products}
		
		For any $*$-algebras $A$ and $B$, the algebraic tensor product $A \otimes B$ is also a $*$-algebra with the multiplication and $*$-operation satisfying
		$$
		(X\otimes Y)(W\otimes Z) = (XW) \otimes (ZY)
		$$
		and
		$$
		(X\otimes Y)^* = X^*\otimes Y^*.
		$$
		Moreover, in the case of $A \otimes A$, for $X, Y, Z \in A$, we write
		\[
		(X \otimes Y) \# Z = XZY,
		\]
		and extend this function of $X \otimes Y$ linearly to the tensor product $A \otimes A$.
		
		Given a tracial von Neumann algebra $\mathcal{A} = (A,\tau)$, the algebraic tensor product $A \otimes A$ can be equipped with the trace $\tau \otimes \tau$ and then completed to a tracial von Neumann algebra.  The inner product is then given by
		\[
		\ip{X \otimes Y, Z \otimes W} = \tau(X^*Z) \tau(Y^*W).
		\]
		The von Neumann algebraic tensor product with completion is also denoted by, for example, $\cA \otimes \cA$.
		
		\subsection{The free difference quotient}
		
		Voiculescu's $j$th \emph{free difference quotient} is the map $\partial_{x_j}: \NCPd \to \NCPd \otimes \NCPd$ given, for a monomial, by
		$$
		\partial_{x_j} x_{i_1}x_{i_2}\ldots x_{i_m} := \sum_{i_k=j} (x_{i_1}\ldots x_{i_{k-1}})\otimes (x_{i_{k+1}} \ldots x_{i_m}).
		$$
		In the cases where $i_k=j$ for $k=1$, we take the left part of the tensor to be the identity, and when $k=m$, we take the right part to be the identity.  This operation is extended linearly to $\NCPd$.  Note that if $\cA$ is a tracial $\mathrm{W}^*$-algebra and $x \in \cA_{\operatorname{sa}}^d$, then $\partial_{x_j} p$ can be evaluated at $x$ to produce an element in $\cA {\otimes} \cA$.  
		
		As motivation for this definition, we recall the following fact; see e.g.\ \cite[Lemma 14.1.3]{jekel2020evolution}, \cite[Lemma 3.17]{jekel2022tracial}.
		
		\begin{lemma} \label{lem: NCP differentiability}
			Let $p \in \NCPd$.  Fix a tracial von Neumann algebra $\cA = (A,\tau)$.  Then $p$ defines a map $A_{\operatorname{sa}}^d \to A$ which is Fr{\'e}chet differentiable and satisfies
			\[
			\frac{d}{dt} \biggr|_{t=0} p(x + ty) = \sum_{j=1}^d \partial_{x_j} p(x) \# y_j.
			\]
		\end{lemma}
		
		The non free difference quotionts also satisfy a chain rule; recall the multiplication in the tensor product space defined earlier. 
		
		\begin{lemma}[Chain rule]
			Let $f \in \NCPd$ and let $g = (g_1,\dots,g_d) \in \operatorname{NCP}_{d'}^d$.  Then for $i\in \{1,\ldots, d'\}$,
			$$
			\partial_{x_i}(f\circ g)(x) = \sum_{j=1}^d \partial_{x_j} f(g(x)) \partial_{x_i} g_j(x).
			$$
		\end{lemma}
		
		\begin{remark}  
			In the case of a single variable, the non-commutative derivative $\partial_x$ relates closely to difference quotients. Indeed, $\partial_x$ maps $\bC[x]$ into $\bC[x] \otimes \bC[x]$, and there is a canonical algebra isomorphism $\Phi: \bC[x] \otimes \bC[x] \to \bC[x,y]$ such that $\Phi(x \otimes 1) = x$ and $\Phi(1 \otimes x) = y$.  Note
			\[
			\partial_x[x^k] = x^{k-1} \otimes 1 + x^{k-2} \otimes x + \dots + 1 \otimes x^{k-1},
			\]
			and thus
			\[
			\Phi[\partial_x[x^k]] = x^{k-1} + x^{k-2} y + \dots + y^{k-1} = \frac{x^k - y^k}{x - y} \in \bC[x,y].
			\]
			Hence, by linearity, for all $p \in \bC[x]$,
			\[
			\Phi[\partial_x p(x)] = \frac{p(x) - p(y)}{x - y} \in \bC[x,y].
			\]
			In a classical probabilistic setting, the space $\mathbb{C}[x,y]$ will often be equipped with a product measure representing identical copies of independent random variables. Suppose that $\tau$ is the trace of von Neumann algebra and $X\in L^2(\cA)_{sa}$ has spectral measure $\mu\in \mathcal{P}_2(\mathbb{R})$, i.e., for $p\in \mathbb{C}[x]$,
			$$
			\tau\big(p(X)\big) = \int_{\mathbb{R}}p(x)\mu(dx).
			$$
			This identification extends to the tensor product using the isomorphism $\Phi$. For $A\in C[x]\otimes C[x]$ we have
			$$
			(\tau\otimes \tau) \big(A(X)\big) = \int_{\mathbb{R}}\int_{\mathbb{R}} \Phi[A](x,y)\mu(dx)\mu(dy).
			$$
		\end{remark}
		
		An example that has driven a lot of research in free probability is the case of resolvents (see \cite{voiculescu2000coalgebra}).  Although resolvents are not polynomials, they can be expressed through power series, and so the definition of the non-commutative can be extended to them without much difficulty.  
		
		For resolvents, we may compute, for example,
		$$
		(1-\lambda\, x)^{-1} - (1-\lambda\, y)^{-1} = \lambda (1-\lambda\, x)^{-1}(x-y)(1-\lambda\, y)^{-1}
		$$
		and it follows that
		$$
		\partial_x (1-\lambda\, x)^{-1} = \lambda\, (1-\lambda\, x)^{-1}\otimes (1-\lambda\, x)^{-1}.
		$$
		
		For another example we make use of later, consider $\psi(x) = {\rm arctan}(x)$. We have for a single-variable that the derivative can be computed using the standard derivative of {\rm arctan},
		\begin{align*}
			\frac{d}{dt}{ \psi}(t\, X) =&\ X\, (1 + t^2 X^2)^{-1}\\
			=&\ \frac{1}{2}X\big[(1 + {\rm i}\, tX)^{-1} + (1 - {\rm i}\, tX)^{-1}\big]\\
			=&\ \frac{{\rm i}}{2t}\big[(1 + {\rm i}\, tX)^{-1} - (1 - {\rm i}\, tX)^{-1}\big].
		\end{align*}
		We find that we can write the arctan as an integral of resolvents
		$$
		{\rm arctan}(X) = \int_0^1 \frac{{\rm i}}{2t}\big[(1 + {\rm i}\, tX)^{-1} - (1 - {\rm i}\, tX)^{-1}\big] dt.
		$$
		The non-commutative derivative is then given by
		$$
		\partial_x \arctan(X) = \int_0^1 \frac{1}{2}\big[(1 + {\rm i}\, tX)^{-1}\otimes (1 + {\rm i}\, tX)^{-1} + (1 - {\rm i}\, tX)^{-1}\otimes(1 - {\rm i}\, tX)^{-1}\big] dt.
		$$
		
			

		\subsection{The cyclic derivative}
		The \emph{cyclic derivative} $\mathcal{D}_{x_j}^\circ: \NCPd \to \NCPd$ is defined for non-commutative monomials by
		$$
		\mathcal{D}_{x_j}^\circ x_{i_1}x_{i_2}\ldots x_{i_m} := \sum_{i_k=j} \big(x_{i_{k+1}} \ldots x_{i_m} x_{i_1}\ldots x_{i_{k-1}} \big), 
		$$
		and extended linearly to non-commutative polynomials.  Equivalently, $\mathcal{D}_{x_j}^\circ p$ is obtained from $\partial_{x_j} p$ by applying the map $X \otimes Y \mapsto YX$.  From this we see that $\tau(y \mathcal{D}_{x_j}^\circ p(x)) = \tau(\partial_{x_j} p(x) \# y)$ whenever $x$ is a self-adjoint tuple in a tracial von Neumann algebra and $y$ is an element of it.
		
		We remark that if $p$ is self-adjoint, then $\partial_{x_j} p$ is also self-adjoint.  We denote by $\mathcal{D}^\circ p$ the row vector with entries $\mathcal{D}_{x_j}^\circ p$.
		
		The cyclic derivative arises naturally from computing the gradient of the trace of a non-commutative polynomial as follows.
		
		\begin{lemma} \label{lem: cyclic gradient}
			Let $p \in \NCPd$.  Fix a tracial von Neumann algebra $\cA = (A,\tau)$.  Then $\tau \circ p$ defines a map $A_{\operatorname{sa}}^d \to \bC$ which is Fr{\'e}chet differentiable and satisfies
			\[
			\nabla[\tau(p(x))] = \mathcal{D}^\circ p(x).
			\]
		\end{lemma}
		
		\begin{proof}
			Differentiability follows immediately from Lemma \ref{lem: NCP differentiability} since $\tau$ is linear.  Moreover, for $y \in A_{\operatorname{sa}}^d$, we have
			\[
			\ip{y, \nabla p(x)}_{L^2(\cA)} = 
			\frac{d}{dt} \biggr|_{t=0} \tau [p(x + ty)] = \sum_{j=1}^d \tau(\partial_{x_j} p(x) \# y_j) = \sum_{j=1}^d \tau(y_j \mathcal{D}_{x_j}^\circ p(x)) = \ip{y, \mathcal{D}^\circ p(x)}_{L^2(\cA)}.
			\]
			Since $y$ is arbitrary,  
			{{$\nabla[\tau(p(x))]= \mathcal{D}^\circ p(x).$}}
		\end{proof}

		\subsection{Cylindrical test functions} \label{subsec: cylindrical}
		
		A useful class of test functions where one can compute the free Laplacian are cylindrical functions of the form
		$$
		U_{\cA}(X) = g\Big( \tau\big((\phi_1\circ\psi) (X)\big),\ldots,  \tau\big((\phi_m\circ \psi)(X)\big)\Big)
		$$
		for $g\in C^{2,1}(\bR^m;\bR)$, self-adjoint $\phi_1,\ldots,\phi_m\in {\rm NCP}_d$, and $\psi(x)={\rm arctan}(x)$ applied component-wise. Clearly, $(U_{\cA})_{\cA\in \bW}$ are tracial $W^*$-functions.  In fact, they fall into the category of non-commutative smooth functions studied in \cite{jekel2022tracial}, as one can see from \S 3.4 and \S 4.2 of that paper.  The free Laplacian computed here for cylindrical functions will thus be a special case of the free Laplacian in \cite[\S 4.3]{jekel2022tracial}, but here we want to give a shorter and more self-contained development for this concrete family.
		
		We use the notation $g_o$ to denote the partial derivative with respect to the $o$th component of $g$ and $g_{oq}$ to denote the second partial derivative with respect to the $o$ and $q$th components. We abbreviate by
		$$
		g_o = g_o\Big( \tau\big((\phi_1\circ\psi) (X)\big),\ldots,  \tau\big((\phi_m\circ \psi)(X)\big)\Big)
		$$
		and
		$$
		g_{oq} = g_{oq}\Big( \tau\big((\phi_1\circ\psi) (X)\big),\ldots,  \tau\big((\phi_m\circ \psi)(X)\big)\Big).
		$$ 
		Then by the chain rule we have
		$$
		\partial_{x_j}(\phi_m\circ \psi)(X) = (\partial_{x_j}\phi_m\circ \psi)(X)\, \partial_{x}\psi(X_j).
		$$
		Similarly, we may define the cyclic derivative $\mathcal{D}_{x_j}^\circ (\phi_m\circ \psi)(X)$.
		
		By Lemma \ref{lem: cyclic gradient} and the chain rule, we have that, for $j\in \{1,\ldots, d\}$,
		\begin{align*}
			(\nabla U_{\cA})^j(X) = \sum_{o=1}^m g_o\, \mathcal{D}_{x_j}^\circ  (\phi_o\circ \psi)(X).
		\end{align*}
		
		We define the non-commutative derivative, for $i,j\in \{1,\ldots, d\}$, naturally to be
		$$
		\partial_{x_i} (\nabla U_{\cA})^j(X) = \sum_{o=1}^m g_o\, \partial_{x_i}\mathcal{D}_{x_j}^o (\phi_o\circ\psi)(X)  \in L^2(\cA\otimes \cA,\tau\otimes \tau).
		$$
		
		\begin{proposition}\label{prop:cylindrical_test_functions}
			For  a cylindrical function $U$ as given above, we have $U\in \mathcal{X}_\Sigma$ and
			\begin{align*}
				\Theta_{\cA} U(X) = \sum_{i=1}^{d}(\tau\otimes \tau)\Big( \partial_{x_i} (\nabla U_{\cA})^i(X)\Big).
			\end{align*}
		\end{proposition}
		
		\begin{proof}
			There are two tensors arising in the calculation of the Hessian of $U_\cA$ that act in different ways. The first is the non-commutative derivative, $
			\partial_{x_i} (\nabla U_{\cA})^j(X)$ defined above. We also collect terms involving the second derivatives of $g$ by
			$$
			(\nabla^2 U_{\cA})^{ij}(X) = \sum_{o=1}^m\sum_{q=1}^m g_{oq}\, \mathcal{D}_{x_i}^o (\phi_o\circ\psi)(X)\otimes \mathcal{D}^o_{x_j} (\phi_q\circ\psi)(X).
			$$
			The Hessian of $U$ can be expressed as
			\begin{align}\label{eqn:cylindrical_hessian}
				&\ {\rm Hess}\, U_{\cA}(X)\big[A,B\big]\nonumber\\
				=&\ \sum_{i=1}^d\sum_{j=1}^d\Big( \langle(\nabla^2 U_\cA)^{ij}(X), A^i\otimes B^j\rangle_{L^2(\cA\otimes \cA)} + \langle \partial_{x_i} (\nabla U_\cA)^j(X)\# A^i,B^j\rangle_{L^2(\cA)}\Big).
			\end{align}
			We check condition (d) of the definition $\mathcal{X}_\Sigma$ as the other conditions are simpler.  We fix $\cA \in \bW$, $X,Y,A\in L^2(\cA)^d_{sa}$ and $B\in L^\infty(\cA)_{sa}^d$. Furthermore, we consider the case that $U_\cA(X) = \tau\big((\phi\circ \psi)(x)\big)$ for a monomial $\phi(x) = \prod_{k=1}^mx_{i_k} \in \text{NCP}_d$, which showcases the key calculation. We have
			\begin{align*}
				&\ \langle(\nabla U_{\cA})^j(X),A\rangle_{L^2(\cA)}\\
				=&\ \frac{1}{2}\int_0^1\sum_{i_k=j}\sum_{\rho=\pm1}\tau\big(\psi(X_{i_{1}})\ldots \psi(X_{i_{k-1}})(1+\rho\, {\rm i}\, t\, X_j)^{-1} A_j(1+\rho\, {\rm i}\, t\, X_j)^{-1}\, \psi(X_{i_{k+1}})\ldots \psi(X_{i_m})\big)dt.
			\end{align*}
			We have when computing the Hessian finitely many terms that include either, for $j,j'\in \{1,\ldots,d\}$ and $\rho,\rho'\in \{\pm 1\}$ and $i_k=j$ and $i_{k'}=1$,
			\begin{align*}
				&\ \tau\big(\psi(X_{i_{1}})\ldots \psi(X_{i_{k-1}})(1+\rho\, {\rm i}\, t\, X_j)^{-1} A_j(1+\rho\, {\rm i}\, t\, X_j)^{-1}\, \psi(X_{i_{k+1}})\\
				&\ \ldots \psi(x_{i_{k'-1}}) (1+\rho'\, {\rm i}\, t\, X_{j'})^{-1} B_{j'}(1+\rho'\, {\rm i}\, t\, X_{j'})^{-1}\, \psi(X_{i_{k'+1}})\ldots  \psi(X_{i_m})\big),
			\end{align*}
			or, when $j=j'$ and $k=k'$,
			\begin{align*}
				&\ -\rho\, {\rm i}\, t\, \tau\big(\psi(X_{i_{1}})\ldots \psi(X_{i_{k-1}})(1+\rho\, {\rm i}\, t\, X_j)^{-1} A_j(1+\rho\, {\rm i}\, t\, X_j)^{-1} B_j(1+\rho\, {\rm i}\, t\, X_j)^{-1}\, \psi(X_{i_{k+1}})\ldots  \psi(X_{i_m})\big).
			\end{align*}
			When we consider the terms appearing in $|{\rm Hess}\, U_\cA(X)\big[A,B\big] - {\rm Hess}\, U_\cA(Y)\big[A,B\big]|$ we may telescope and rearrange the terms into either the form, where $M_1,M_2,M_3$ have bounded $L^\infty$ norm (consider $\|\arctan(X_j)\|_\infty \leq \frac{\pi}{2})$,
			$$
			\tau\big([\psi(X_i)- \psi(Y_i)]M_1\, A_j\, M_2 B_{j'}\, M_3\big) \leq \big(\frac{\pi}{2}\big)^{m-2}\|X_i-Y_i\|_{L^2(\cA)} \|A_j\|_{L^2(\cA)}\, \|B_j'\|_{L^\infty(\cA)}
			$$
			or
			$$
			\tau\big([(1+\rho\, {\rm i}\, t\, X_i)^{-1}-(1+\rho\, {\rm i}\, t\, Y_i)^{-1}]\, A_j\, M_2 B_{j'}\, M_3\big) \leq \big(\frac{\pi}{2}\big)^{m-1}\|X_i-Y_i\|_{L^2(\cA)} \|A_j\|_{L^2(\cA)}\, \|B_j'\|_{L^\infty(\cA)}.
			$$
			It follows that condition (d) in the definition of $\mathcal{X}_{\Sigma}$ is satisfied with the constant $K$ depending only on the degree of the monomial because of the number of terms that arise in this calculation. For more general cylindrical functions the result holds with $K$ also depending on Lipschitz estimates for $g$.

			We now compute the free Laplacian for a general cylindrical function. Assume that  $\iota:\cA=(A,\tau)\rightarrow \cB=(B,\rho)$ is a  tracial $W^*$-embedding and $\cB$ contains a semi-circular   $S = (S^1,\cdots,S^d)\in B$  freely independent of $\iota(A)$. For the terms with the second-order derivatives of $g$, we see that, for example,  for $l=1,\cdots,d,$
			\begin{multline*}
				\langle(\nabla^2 U_{\cB})^{ll}(X), S^l\otimes S^l\rangle_{L^2(\cB\otimes \cB)}\\
				= \sum_{o=1}^m\sum_{q=1}^m g_{oq}\, \langle \mathcal{D}_{x_l}^o (\phi_o\circ\psi)(X), S^l\rangle_{L^2(\cB)}\, \langle \mathcal{D}^o_{x_l} (\phi_q\circ\psi)(X), S^l\rangle_{L^2(\cB)} = 0,
			\end{multline*}
			because $S^l$ is freely independent of $ \mathcal{D}^o_{x_l} (\phi_o\circ\psi)(X)$ and $S^l$ has a zero trace.

			For the term with the non-commutative second derivative, for each $i$ and $X$, we can decompose the non-commutative derivative into finitely many simple tensors (see the calculations above) like
			\begin{align*}
				\partial_{x_i} (\nabla U_{\cA})^i(X) = \sum_{j=1}^M a^j\otimes b^j.
			\end{align*}
			We then have that, using the free Wick law (which follows here directly from free independence) to evaluate,
			\begin{align*}
				{\rm Hess}\, U_{\cB}(X)[\mathbf{e}^iS, \mathbf{e}^iS]
				=&\ \big\langle \partial_{x_i} (\nabla U_\cB)^i(X)\# S^i, S^i\big\rangle_{L^2(\cB)}\\
				=&\ \sum_{j=1}^M\langle a^j\, S^i\, b^j, S^i\rangle_{L^2(\cB)}\\
				=&\ \sum_{j=1}^M\rho( S^i\, a^j\, S^i, b^j)\\
				=&\ \sum_{j=1}^M\tau(a^j)\tau(b^j)\\
				=&\ \sum_{j=1}^M(\tau \otimes \tau)(a^j\otimes b^j)\\
				=&\ (\tau\otimes \tau)\big(\partial_{x_i} (\nabla U_{\cA})^i(X)\big).
			\end{align*}
		\end{proof}

		
		
		
		We can also consider the special case of $d=1$ and relate it to the Wasserstein derivative of $U(\mu)$ where $\mu$ denotes the probability measure on $\mathbb{R}$ and the Wasserstein derivative is denoted by $\partial U(\mu)(x)$.  
		We will only describe this here as a formal computation.  We see that if $\cA$ is a tracial $W^*$-algebra and $X\in L^2(\cA)_{sa}$ with law $\mu$ then, since $\partial U(\mu)$ is a Borel function, the spectral theory of bounded normal operators allows us to define the operator $\partial U(\mu)(X)$. When $U$ is differentiable at $\mu$ and $\cU_\cA$ is differentiable at $X$, i.e., the cylindrical test functions above, one checks that the identity
		$$
		\partial U(\mu)(X) = \nabla U_{\cA}(X)
		$$
		holds (note the resemblance with the  identity in  \cite[Corollary 3.22]{gangbo2019}, which instead involves $\partial U(\mu)\circ X$).
		
		Using the canonical isomorphism $\Phi: \mathbb{C}[x] \otimes \mathbb{C}[x] \to \mathbb{C}[x,y]$ above we have
		$$
		\Phi\big[\partial_{x} \partial U(\mu)(x)\big] =  \frac{\partial U(\mu)(x) - \partial U(\mu)(y)}{x - y}.
		$$
		The free Laplacian can then be expressed using the identification of the trace tensor product with the double integral as
		\begin{align*}
			\Theta_{\cA}U(X) =&\ \int_{\mathbb{R}}\int_{\mathbb{R}} \frac{\partial U(\mu)(x) - \partial U(\mu)(y)}{x-y}\mu(dx)\mu(dy)\\
			=&\ 2\int_{\mathbb{R}}\Big(P.V.\int_{\mathbb{R}} \frac{\mu(dy)}{x-y}\Big) \partial U(\mu)(x)\mu(dx).
		\end{align*}
		The final form involves the Hilbert transform of the density, provided it exists.  For background and similar computations, see \cite[\S 5]{voiculescu1993analogues1}, \cite[Proposition 3.5]{Voiculescu1998analogues5}.

		\section{Amalgamated free products} \label{sec: AFP}
		
		\subsection{Properties of free products of von Neumann algebras}
		
		In this section, we record several useful properties of amalgamated free products and free independence.  The first property is a form of associativity for free independence, which is well-known (see e.g.\ \cite[Exercise 5.3.8]{anderson2010introduction}, \cite[Exercise 5.25]{nica2006lectures}, \cite[Example 5.22]{jekel2020operad}).  The proof is similar to, but much easier, than the proof of Lemma \ref{lem: AFP manipulation} below, so we omit it.
		
		\begin{lemma}[Associativity of free independence] \label{lem: associativity}
			Let $\mathcal{A}$, $\mathcal{B}$, and $\mathcal{C}$ be von Neumann subalgebras of a tracial von Neumann algebra $\mathcal{M}$.  Let $\mathcal{A} \vee \mathcal{B}$ be the von Neumann subalgebra generated by $\mathcal{A}$ and $\mathcal{B}$.  Then the following are equivalent:
			\begin{enumerate}
				\item $\mathcal{A}$, $\mathcal{B}$, and $\mathcal{C}$ are freely independent.
				\item $\mathcal{A}$ and $\mathcal{B}$ are freely independent, and $\mathcal{A} \vee \mathcal{B}$ and $\mathcal{C}$ are freely independent.
			\end{enumerate}
		\end{lemma}
		
		The next property relates free independence with restriction to subalgebras.
		
		\begin{lemma}[Free independence passes to subalgebras] \label{lem: free product embedding} ~ Let $J$ be any index set.
			\begin{enumerate}
				\item Suppose that $\cA$ is a non-commutative probability space containing $\cC$.  Let $\cC \subseteq \cB_j \subseteq \cA_j \subseteq \cA$.  If the $\cA_j$'s are freely independent over $\cB$, then the $\cB_j$'s are freely independent over $\cC$.
				\item Similarly, suppose non-commutative probability spaces $\cC \subseteq \cB_j \subseteq \cA_j$ are given.  Then there is a unique tracial $\mathrm{W}^*$-embedding $*_{\cC}  (\cB_j)_{j\in J} \to *_{\cC} (\cA_j)_{j\in J}$ that restricts to the given inclusions $\cB_j \to \cA_j$ for each $j$.
			\end{enumerate}
		\end{lemma}
		
		\begin{proof}
			(1) For the $\cB_j$'s to be freely independent over $\cC$ means that for all positive integers $n$ and $j:\{1,\ldots,n\}\rightarrow J$ such that $j(k)\not= j(k+1)$  for $k=1,\cdots,n-1,$
			$$
			E_{\cC}\bigg( \prod_{k=1}^{n}\big(b_k-E_{\cC}(b_k)\big)\bigg)=0, \qquad \hbox{for all }(b_1,\ldots,b_{n}) \in {\cB}_{j(1)}\times\ldots \times {\cB}_{j(n)}.
			$$
			Since ${\cB}_{j(1)}\times\ldots \times {\cB}_{j(n)} \subseteq {\cA}_{j(1)} \times \ldots \times {\cA}_{j(n)}$, the free independence of the $\cA_j$'s over $\cC$ implies free independence of the $\cB_j$'s over $\cC$.
			
			(2) This follows from claim (1), since we can apply Corollary \ref{cor: independence and embeddings} to the subalgebra of $*_{\cC} \cA_j$ generated by the images of $\cB_j$'s.
		\end{proof}
		
		Next, we consider the behavior of certain conditional expectations in free products.  The following fact is well-known in the study of von Neumann algebras.
		
		\begin{lemma} \label{lem: free product commuting square}
			Consider an amalgamated free product $\cA \median_{\cC} \cB$ of tracial von Neumann algebras, and let $\cC \subseteq \cA_0 \subseteq \cA$.  Let $\iota: \cA \to \cA \median_{\cC} \cB$ be the canonical inclusion.  Using the previous lemma, view $\cA_0 \median_{\cC} \cB$ as a subset of $\cA \median_{\cC} \cB$.  Then for $x \in \cA$, we have
			\[
			E_{\cA_0 \median_{\cC} \cB} \circ \iota(x) = \iota \circ E_{\cA_0}(x).
			\]
		\end{lemma}
		
		\begin{proof}
			We regard $\cA$ as a subset of $\cA \median_{\cC} \cB$ in this proof, and hence suppress $\iota$ in the notation.  Let $x' = x - E_{\cA_0}(x)$.  It suffices to show that $x'$ is orthogonal to $\cA_0 \median_{\cC} \cB$ inside $\cA \median_{\cC} \cB$.  We recall (see \cite[p. 684-685]{popa1993markov}) that an $L^2$-dense subset of $\cA_0 \median_{\cC} \cB$ is spanned by alternating products of elements $y_i$ from $\cA$ and $z_i$ from $\cB$ respectively that are orthogonal to $\cC$; for instance, an alternating product that starts with an element from $\cA$ and ends with an element from $\cB$ would be $y_1 z_1 \dots y_k z_k$.  Hence, it suffices to show that $x'$ is orthogonal to these elements.  Note
			\[
			x'(y_1 z_1 \dots y_k z_k) = (x'y_1)z_1 \dots y_k z_k
			\]
			is another alternating product of elements from $\cA$ and $\cB$ that are orthogonal to $\cC$; indeed,
			\[
			E_{\cC}[x'y_1] = E_{\cC} \circ E_{\cA_0}[x'y_1] = E_{\cC}[ E_{\cA_0}[x'] y_1] = 0,
			\]
			and by assumption $y_j \in \cA_0 \subseteq \cA$ and $z_j \in \cB$.  Therefore, by free independence of $\cA$ and $\cB$ over $\cC$, we have $E_{\cC}[(x'y_1)z_1 \dots y_k z_k] = 0$ as desired.  The other cases where the alternating string starts wtih an element of $\cB$ or ends with an element of $\cA$ are similar.
		\end{proof}
		
		The last and most substantial result is a free product manipulation for which several cases and related facts exist in the literature \cite[Proposition 4.1]{houdayer2007freeproducts}.  We remark that the free independence manipulation of \cite[Lemma 2.6]{shlyakhtenko2000entropy} which changes the base subalgebra is somewhat related but does not imply the statement here.
		
		\begin{lemma} \label{lem: AFP manipulation}
			Let $(\cA_j)_{j \in J}$ be a family of tracial $\mathrm{W}^*$-algebra containing a common subalgebra $\cB$, let $\cD \subseteq \cB$, and let $\cC$ be another tracial $\mathrm{W}^*$-algebra containing $\cD$.  For each $i \in J$, let
			\[
			\varphi_i: \cB \median_{\cD} \cC \to \cA_i \median_{\cD} \cC
			\]
			be the embedding given by Lemma \ref{lem: free product embedding} applied to the inclusion $\cB \to \cA_i$ and the identity $\cC \to \cC$.  Let $\median_{\cB \median_{\cD} \cC} \left(\cA_j \median_{\cD} \cC \right)_{j \in J}$ be the amalgamated free product over $\cB \median_{\cD} \cC$, where $\cB \median_{\cD} \cC$ is viewed as a subalgebra of $\cA_i \median_{\cD} \cC$ via $\varphi_i$.  Then we have a unique tracial $\mathrm{W}^*$-isomorphism
			\[
			\median_{\cB \median_{\cD} \cC} \left(\cA_j \median_{\cD} \cC \right)_{j \in J} \xrightarrow[\cong]{\Phi} \left( \median_{\cB} (\cA_j)_{j \in J} \right) \median_{\cD} \cC
			\]
			which behaves in the natural way on each copy of $\cA_j$ and each copy of $\cC$.  More precisely, there is a unique isomorphism such that the following diagrams commute for each $i \in J$:
			\[
			\begin{tikzcd}
				& \cA_i \arrow{dl}{\rho_i} \arrow{dr}{\lambda_i} & \\
				\cA_i \median_{\cD} \cC \arrow{d}{\tilde{\lambda}_i} & & \median_{\cB} (\cA_j)_{j \in J} \arrow{d}{\rho} \\
				\median_{\cB \median_{\cD} \cC} \left(\cA_j \median_{\cD} \cC \right)_{j \in J} \arrow{rr}{\Phi} & & \left( \median_{\cB} (\cA_j)_{j \in J} \right) \median_{\cD} \cC,
			\end{tikzcd}
			\]
			where $\rho_i$, $\lambda_i$, $\tilde{\lambda}_i$, and $\rho$ are the canonical inclusions associated to the free product constructions, and similarly,
			\[
			\begin{tikzcd}
				\cA_i \median_{\cD} \cC \arrow{d}{\tilde{\lambda}_i} & \cC \arrow{l}{\sigma_i} \arrow{d}{\sigma} \\
				\median_{\cB \median_{\cD} \cC} \left(\cA_j \median_{\cD} \cC \right)_{j \in J} \arrow{r}{\Phi} & \left( \median_{\cB} (\cA_j)_{j \in J} \right) \median_{\cD} \cC,
			\end{tikzcd}
			\]
			where $\sigma_i$ and $\sigma$ are the canonical inclusions of $\cC$ associated to the free product constructions.
		\end{lemma}
		
		\begin{proof}
			Let $\lambda_i: \cA_i \to \median_{j \in J} \cA_j$ be the embedding given by the free product construction.  By Lemma \ref{lem: free product embedding} (2), $\lambda_i$ and the identity map on $\cB$ induce an embedding
			\[
			\pi_i: \cA_i \median_{\cD} \cC \to (\median_{\cB} (\cA_j)_{j \in J}) \median_{\cD} \cC.
			\]
			Our goal is to show that the images $\pi_j(\cA_j \median_{\cD} \cC)$ are freely independent with amalgamation over $\cB \median_{\cD} \cC$ in $(\median_{\cB} (\cA_j)_{j \in J}) \median_{\cD} \cC$, and that they generate it.  Indeed, if we can show this, then the uniqueness of the free product with amalgamation over $\cB \median_{\cD} \cC$ up to canonical isomorphism (Lemma \ref{lem: uniqueness of free product}) will imply the existence and uniqueness of the isomorphism $\Phi$ with the desired properties.  Moreover, the fact that $(\median_{\cB} (\cA_j)_{j \in J}) \median_{\cD} \cC$ is generated by the images of $\pi_i$ is immediate because it is generated by the images of $\cA_j$ and $\cC$.  Thus, the main challenge is showing free independence over $\cB \median_{\cD} \cC$.
			
			In the entire rest of the argument, we work in $(\median_{\cB} (\cA_j)_{j \in J}) \median_{\cD} \cC$ as the ambient space and view $\cA_i \median_{\cD} \cC$ as a subset of it via the map $\pi_i$, and we therefore suppress $\pi_i$ in our notation.  To show free independence of $\cA_j \median_{\cD} \cC$ over $\cB \median_{\cD} \cC$, it suffices to prove free independence over $\cB \median_{\cD} \cC$ for weak-operator dense $*$-subalgebras of $\cA_j * \cC$; this is well known and follows from a straightforward limiting argument as in \cite[Proposition 2.5.7]{voiculescu1985symmetries}.  Therefore, given $j: \{1,\dots,N\} \to J$ and $X_i$ in some dense subalgebra of $\cA_{j(i)} *_{\cD} \cC$ with $E_{\cB *_{\cD} \cC}[X_i] = 0$, we seek to show that $E_{\cB *_{\cD} \cC}[X_1 \dots X_N] = 0$.  We rely on the decomposition of elements in the free product $\cA_j *_{\cD} \cC$ as linear combinations of strings in $\cA_j$ and $\cC$.  Consider the dense subalgebra of $\cA_j *_{\cD} \cC$ generated by $\cA_j$ and $\cC$; these elements can be expressed as linear combinations of products of elements from $x_i$ from $\cA_j$ and $y_i$ from $\cC$ respectively that are orthogonal to $\cD$ (for instance, an alternating product that starts with an element from $\cA_j$ and ends with an element from $\cC$ would be $x_1 y_1 \dots x_k y_k$). Thus our dense subalgebra can be expressed as
			\[
			\cC \oplus \bigoplus_{n=0}^\infty \cC \cdot (\cA_j^\circ \cdot \cC^\circ)^n \cdot\cA_j^\circ \cdot \cC,
			\]
			where $\cA_j^\circ = \{a \in \cA_j: E_{\cD}[a] = 0\}$ and $\cC^\circ = \{c \in \cC: E_{\cD}[c] = 0\}$.  These summands are, in fact, orthogonal to each other \cite[p. 684-685]{popa1993markov}.  Moreover, on each one, the inner product is given by
			\[
			\tau((y_0 x_1 y_1 \dots x_k y_k)^* y_0' x_1' y_1' \dots x_k' y_k') = \tau\left( y_k^* E_{\cD}[\dots x_1^*E_{\cD}[y_0^*y_0]x_1' \dots] y_k'\right).
			\]
			In other words, we have a bimodule decomposition,
			\[
			\cC \oplus \bigoplus_{n=0}^\infty \cC \otimes_{\cD} (\cA_j^\circ \otimes_{\cD} \cC^\circ)^{\otimes_{\cD}} \otimes_{\cD} \cA_j^\circ \otimes_{\cD} \cC.
			\]
			This follows for instance from the way that the amalgamated free product is constructed (see \cite[\S 4.7]{brown2008c*algebras} and use for the $\mathcal{H}_i$'s the $\cA_j$-$\cD$ correspondence generated by $\cA_j$ and $\cC$-$\cD$-correspondence generated by $\cC$).  In summary, we have a dense subalgebra
			\[
			\cA_j *_{\cD} \cC \supseteq \cC \oplus \bigoplus_{n=0}^\infty \cC \otimes_{\cD} (\cA_j^\circ \otimes_{\cD} \cC^\circ)^{\otimes_{\cD} n} \otimes_{\cD} \cA_j^\circ \otimes_{\cD} \cC.
			\]
			representing that each element of $\cA_j *_{\cD} \cC$ is expressed as a sum of products $c_1 a_0 c_2 \dots a_{n-1} c_n$ where $a_j \in \cA_j$ with $E_{\cD}[a_j] = 0$ and $c_j \in \cC$ with $E_{\cD}[c_j] = 0$ except for the first and last $c_j$'s, and the products for different lengths of strings are orthogonal to each other.  Now we further decompose
			\begin{align*}
				\cA_j^\circ &= \cB^\circ \oplus \cA_j^\bullet, \qquad \text{ where } \\
				\cB^\circ &= \{b \in \cB: E_{\cD}[b] = 0\}, \\ \cA_j^\bullet &= \{a \in \cA_j: E_{\cB}[a] = 0\}.
			\end{align*}
			Hence, we obtain a dense subalgebra
			\[
			\cA_j *_{\cD} \cC) \supseteq \cC \oplus \bigoplus_{n=0}^\infty \bigoplus_{H_\ell \in \{\cB^\circ, \cA_j^{\bullet} \}} \cC \otimes_{\cD} H_1 \otimes_{\cD} \cC^\circ \otimes_{\cD} \dots \otimes_{\cD} H_n \otimes_{\cD} \cC,
			\]
			where the summands are orthogonal.  When we restrict to the terms where all the $H_j$'s are $\cB^\circ$, that produces a dense subalegbra of $\cB *_{\cD} \cC$.  Thus, since $X_i$ is orthogonal to $\cB *_{\cD} \cC$, we see that $X_i$ can be expressed as a sum of terms in $\cC \otimes_{\cD} H_1 \otimes_{\cD} \cC^\circ \otimes_{\cD} \dots \otimes_{\cD} H_n \otimes_{\cD} \cC$ where at least one of the $H_\ell$'s is $\cA_j^{\bullet}$.  By taking linear combinations, we can assume without loss of generality that $X_i$ is an element of one of the spaces $L^2(\cC) \otimes_{\cD} H_1 \otimes_{\cD} L^2(\cC)^\circ \otimes_{\cD} \dots \otimes_{\cD} H_n \otimes_{\cD} L^2(\cC)$, so that
			\[
			X_i = y_{i,0} x_{i,1} y_{i,1} \dots x_{i,k_i} y_{i,n_i},
			\]
			where $y_{i,\ell} \in L^2(\cC)$ for all $\ell$, $y_{i,\ell} \in L^2(\cC)^\circ$ for $0 < \ell < n_i$, and $x_{i,\ell} \in L^2(\cB)^\circ$ or $x_{i,\ell} \in L^2(\cA_{j(i)})^{\bullet}$.
			
			Now we proceed to show that $X_1 \dots X_n$ is orthogonal to $\cB *_{\cD} \cC$ using induction on the total length $n_1 + \dots + n_i$.  Since the orthogonal complement of $\cB *_{\cD} \cC$ is a bimodule over $\cC$, we can factor out the first term $y_{1,0}$ and thus assume without loss of generality that $y_{1,0} = 1$.  Similarly, for $i > 1$, the term $y_{i,0}$ can be combined with $y_{i-1,n_{i-1}}$, so without loss of generality $y_{i,0} = 1$.  Next, each $y_{i,n_i}$ can be expressed as $E_{\cD}[y_{i,n_i}] + (y_{i,n_i} - E_{\cD}[y_{i,n_i}])$ and hence by linearity we can assume without loss of generality that each $y_{i,n_i}$ is either in $\cD$ or in $L^2(\cC)^{\circ}$.
			
			We want to apply the free product structure of $\cA *_{\cD} \cC$ where $\cA = *_{\cB} \cA_j$.  Note as above that there is a dense subalgebra
			\begin{align*}
				L^2(\cA *_{\cD} \cC) &\subseteq \cC \oplus \bigoplus_{n=0}^\infty \cC \otimes_{\cD} (\cA^\circ \otimes_{\cD} \cC^\circ)^{\otimes_{\cD} n} \otimes_{\cD} \cA^\circ \otimes_{\cD} \cC \\
				&\cong \cC \oplus \bigoplus_{n=0}^\infty \bigoplus_{H_\ell \in \{\cB^\circ, \cA^{\bullet} \}} \cC \otimes_{\cD} H_1 \otimes_{\cD} \cC^\circ \otimes_{\cD} \dots \otimes_{\cD} H_n \otimes_{\cD} \cC.
			\end{align*}
			Meanwhile,
			\[
			X_1 \dots X_n = (x_{1,1} y_{1,1} \dots x_{1,n_1} y_{i,n_i}) \dots (x_{N,1} y_{N,1} \dots x_{N,n_N} y_{N,n_N})
			\]
			It is important to keep track of what happens when we multiply $x_{i,n_i}$, $y_{i,n_i}$, and $x_{i+1,1}$ since $y_{i,n_i}$ is not necessarily in $\cC^{\circ}$.  We distinguish for each $i < N$ several cases:
			\begin{itemize}
				\item \textbf{Case 1:} $y_{i,n_i} \in \cD$, and $x_{i,n_i}$ and $x_{i+1,1}$ are in $\cB^\circ$.
				\item \textbf{Case 2:} $y_{i,n_i} \in \cD$, and $x_{i,n_i} \in \cA_{j(i)}^\bullet$ and $x_{i+1,1} \in \cB^\circ$.
				\item \textbf{Case 3:} $y_{i,n_i} \in \cD$, and $x_{i,n_i} \in \cB^\circ$ and $x_{i+1,1} \in \cA_{j(i+1)}^\bullet$.
				\item \textbf{Case 4:} $y_{i,n_i} \in \cD$, and $x_{i,n_i} \in \cA_{j(i)}^{\bullet}$ and $x_{i+1,1} \in \cA_{j(i+1)}^{\bullet}$.
				\item \textbf{Case 5:} $y_{i,n_i} \in \cC^\circ$.
			\end{itemize}
			We will first show that if there is any term in Cases 1-3, we can reduce the claim to a shorter string and apply the induction hypothesis, and moreover that if all the terms fall under Cases 4 and 5, we obtain the desired orthogonality from the free product structure of $\cA *_{\cD} \cC$.
			
			\textbf{Case 1:} Note that $x_{i,n_i} y_{i,n_i} x_{i+1,1} \in \cB$ and thus can be written as $d + b$ where $d \in \cD$ and $b \in \cB^\circ$.  Thus,
			\begin{align*}
				X_i X_{i+1} &= [x_{i,1} y_{i,1} \dots x_{i,n_i-1} (y_{i,n_i-1} d y_{i+1,1})] [x_{i+1,2} \dots x_{i+1,n_{i+1}} y_{i,n_{i+1}}] \\ & \quad + [x_{i,1} y_{i,1} \dots x_{i,n_i-1} y_{i,n_i-1}] [b y_{i+1,1} \dots x_{i+1,n_{i+1}} y_{i,n_{i+1}}],
			\end{align*}
			and hence $X_1 \dots X_n$ is a linear combination of two terms of the same form but with a lower total degree.  Hence, by induction hypothesis, it is orthogonal to $\cB *_{\cD} \cC$ as desired.
			
			\textbf{Case 2:}  Since $\cA_{j(i)}^{\bullet}$ is a bimodule over $\cB$, we have that $x_{i,n_i} y_{i,n_i} x_{i+1,1} \in L^2(\cA_{j(i)})^{\bullet}$.  Hence,
			\[
			X_i X_{i+1} = [x_{i,1} y_{i,1} \dots x_{i,n_i-1} y_{i-1,n_{i-1}} (x_{i,n_i} y_{i,n_i} x_{i+1,1}) y_{i+1,1}] [x_{i+1,2} y_{i+1,2} \dots x_{i+1,n_{i+1}} y_{i,n_{i+1}}],
			\]
			and so $X_1 \dots X_n$ has been expressed as a term of the same form but with a lower degree, and we may apply the induction hypothesis.
			
			\textbf{Case 3:} The argument is symmetrical to Case 2.
			
			At this point, we can assume that every index $i$ is in Case 4 or Case 5.  Besides, since we want to show that the element is orthogonal to $\cB *_{\cD} \cC$, we can factor out any terms at the front that come from $\cB *_{\cD} \cC$, and thus assume that $x_{1,1} \in \cA_{i(1)}^{\bullet}$.  We claim that $X_1 \dots X_N$ can be expressed as
			\begin{multline*}
				a_{1,1} d_{1,1} \dots a_{1,k_1-1} d_{1,k_1-1} a_{1,k_1}  b_{1,1} c_{1,1} \dots b_{1,\ell_1} c_{1,\ell_1} \\
				\dots \\
				a_{K,1} d_{K,1} \dots a_{K,k_K} d_{1,k_K} c_{1,K} b_{1,K} \dots c_{1,\ell_K} b_{1,\ell_K} c_{1,\ell_K+1},
			\end{multline*}
			where
			\begin{itemize}
				\item each $a_{i,1} d_{i,1} \dots a_{i,k_i} d_{i,k_i}$ is an alternating product with $a_{i,j}$ in one of the $\cA_t^{\bullet}$'s  and $d_{i,j} \in \cD$,
				\item each $c_{i,j} \in \cC^{\circ}$,
				\item each $b_{i,j} \in \cB^{\circ}$.
			\end{itemize}
			This factorization arises as follows.  The $d_{i,j}$ terms come from the $y_{i,n_i}$ terms that are in $L^2(\cD)$ under Case 4. The $c_{i,j}$ terms come from all the rest of the $y_{i,j}$ terms, i.e.\ those which are in $\cC^\circ$, and so in particular the junctures between $X_i$ and $X_{i+1}$ in Case 5 will produce $c_{i,j}$ terms.  Thus, each $a_{i,1} d_{i,1} \dots a_{i,k_i} d_{i,k_i}$ corresponds to consecutive occurrences of Case 5; here intermediate terms $d_{i,j} a_{i,j} d_{i,j+1}$ only occur when $a_{i,j}$ comes from some $X_{i'}$ of the form $x_{i',1} y_{i',1}$.  Now we notice that $a_{1,1} d_{1,1} \dots a_{1,k_1-1} d_{1,k_1-1} a_{1,k_1} \in \cA^{\bullet}$ since the $\cA_j$'s are freely independent over $\cB$.  Therefore, $X_1 \dots X_N$ is an element in
			\[
			\cC \otimes_{\cD} H_1 \otimes_{\cD} \cC^\circ \otimes_{\cD} \dots \otimes_{\cD} H_n \otimes_{\cD} \cC,
			\]
			where $H_\ell \in \{\cA^{\bullet}, \cB^{\circ}\}$, and there is at least one occurrence of $\cA^{\bullet}$.  Therefore, it follows from the direct sum decomposition that $X_1 \dots X_N$ is orthogonal to $\cB *_{\cD} \cC$ as desired.
		\end{proof}
		
		\subsection{Amalgamation of filtrations and free Brownian motions} \label{subsec: amalgamation of filtrations}
		
		\begin{lemma}[{Lemma \ref{lem: amalgamation of Brownian motions}}] \label{lem: amalgamation of Brownian motions appendix}
			Let $0 \leq t_0 \leq t_1 \leq T$ and $K$ be any index set.  Let $\cA$ be a tracial von Neumann algebra equipped with a filtration $(\cA_t)_{t \in [t_0,t_1]}$ and a compatible $d$-variable free Brownian motion $(S_t^0)_{t \in [t_0,t_1]}$. 
			Let $(\cB^k)_{k \in K}$ be a family of tracial von Neumann algebras and let $\iota_k: \cA \to \cB^k$ be a tracial $\mathrm{W}^*$-embedding.  Suppose each $\cB^{k}$ has a filtration $(\cB_t^k)_{t \in [t_1,T]}$ and a compatible $d$-variable free Brownian motion $(S_t^k)_{t \in [t_1,T]}$.  Assume that $\iota_k(\cA_{t_1}) \subseteq \cB_{t_1}^k$.
			
			Then there exists an algebra $\cB$, a tracial $\mathrm{W}^*$-embeddings $\iota: \cA \to \cB$ and $\widetilde{\iota}_k: \cB^k \to \cB$, a filtration $(\cB_t)_{t \in [t_0,T]}$, and a compatible  $d$-variable free Brownian motion $(S_t)_{t \in [t_0,T]}$ such that the following hold:
			\begin{enumerate}
				\item We have $\widetilde{\iota}_k \circ \iota_k|_{\cA_{t_0}} = \iota|_{\cA_{t_0}}$ for each $k \in K$.
				\item We have $\iota(S_t^0) = S_t$ for $t \in [t_0,t_1]$.
				\item We have $\widetilde{\iota}_k(S_t^k) = S_t - S_{t_1}$ for $t \in [t_1,T]$.
				\item We have $\iota(\cA_t) \subseteq \cB_t$ for $t \in [t_0,t_1]$.
				\item We have $\widetilde{\iota}_k(\cB_t^k) \subseteq \cB_t$ for $t \in [t_1,T]$.
			\end{enumerate}
			Furthermore, for any given $k_0 \in K$, the embedding $\iota$ can be taken to be $\iota = \widetilde{\iota}_{k_0} \circ \iota_{k_0}$.
		\end{lemma}
		
		\begin{proof}
			We want to construct $\cB$ as an amalgamated free product of the $\cB^k$'s over the subalgebra generated by $\cA_{t_1}$ and a free Brownian motion.
			
			First, let $\mathcal{Z}$ be the von Neumann algebra generated by a free Brownian motion $(Z_t)_{t \in [t_1,T]}$ with initial condition $Z_{t_1} = 0$.  The existence of such an algebra follows from Voiculescu's free Gaussian construction.  It is also well known in free probability that $(\mathcal{Z},(Z_t)_{t\in[t_1,T]})$ is unique up to isomorphism, meaning that if $(\widetilde{\mathcal{Z}},(\widetilde{Z}_t)_{t \in [t_1,T]})$ is another tracial von Neumann algebra generated by a free Brownian motion, then there is a unique trace-preserving $*$-isomorphism $\phi: \mathcal{Z} \to \widetilde{\mathcal{Z}}$ such that $\phi(Z_t) = \widetilde{Z}_t$ for $t \in [t_1,T]$.  We explain the proof briefly here for completeness.  Consider non-commutative polynomials in variables indexed by $t \in [t_1,T]$ and $j = 1$, \dots, $d$; this is a well-defined $*$-algebra where each polynomial only depends on finitely many variables, even though there are uncountably many variables overall.  We have $\tau(p(Z_t: t \in [t_1,T])) = \tau(p(\widetilde{Z}_t: t \in [t_0,T]))$ for every $p$ since the joint moments of the $S_t$'s are uniquely determined by it being a Brownian motion.  It follows that $\norm{p(Z_t: t \in [t_1,T])}_{L^2} = \norm{p(\widetilde{Z}_t: t \in [t_1,T])}_{L^2}$, hence there is an isometric map $L^2(\mathrm{W}^*(Z_t: t \in [t_1,T])) \to L^2(\mathrm{W}^*(\widetilde{Z}_t: t \in [t_1,T]))$.  Since the operator norm can be evaluated by testing the inner products against polynomials, we also obtain $\norm{p(Z_t^k: t \in [t_1,T])} = \norm{p(\widetilde{Z}_t: t \in [t_1,T])}$, and so this restricts to a isomorphism of tracial $\mathrm{W}^*$-algebras $\mathrm{W}^*(Z_t: t \in [t_1,T])) \to \mathrm{W}^*(\widetilde{Z}_t: t \in [t_1,T])$.

			Let $\widetilde{\cA}$ be the free product $\cA_{t_1} * \mathcal{Z}$, and denote the associated embeddings
			\[
			\lambda_1: \mathcal{A}_{t_1} \to \widetilde{\cA}, \qquad \lambda_2: \mathcal{Z} \to \widetilde{\mathcal{A}}.
			\]
			We claim that there is a unique tracial $\mathrm{W}^*$-embedding
			\[
			\varphi_k: \widetilde{\cA} \to \cB^k
			\]
			such that
			\begin{equation} \label{eq: October 4}
				\varphi_k \circ \lambda_1 = \iota_k|_{\cA_{t_1}}, \qquad \varphi_k \circ \lambda_2(Z_t) = S_t^k \text{ for } t \in [t_1,T].
			\end{equation}
			First, by uniqueness of the tracial $\mathrm{W}^*$-algebra generated by a free Brownian motion proved above, we obtain an isomorphism $\mathcal{Z} \to \mathrm{W}^*(S_t^k: t \in [t_1,T])$ that sends $Z_t$ to $S_t^k$.  Then because $\iota_k \circ \lambda_1(\cA_{t_1})$ is contained in $\cB_{t_1}^k$ and the Brownian motion $(S_t^k)_{t \in [t_1,T]}$ is compatible with the filtration $(\cB_t^k)_{t \in [t_1,T]}$, we have that $\iota_k \circ \lambda_1(\cA_{t_1})$ is freely independent of $\mathrm{W}^*(S_t^k: t \in [t_1,T])$.  Since the free product is unique up to canonical isomorphism, we have a unique isomorphism $\varphi_k$ from $\widetilde{\cA} = \mathcal{A}_{t_1} * \mathcal{Z}$ to the subalgebra of $\mathcal{B}^k$ generated by $\iota_k \circ \lambda_2(\cA_{t_1})$ and $\mathrm{W}^*(S_t^k: t \in [t_1,T])$ satisfying the desired conditions \eqref{eq: October 4}.
			
			Now let $\cB$ be the amalgamated free product
			\[
			\cB = \median_{\widetilde{\cA}} (\cB^k)_{k \in K}
			\]
			with respect to the inclusions $\varphi_k: \widetilde{\cA} \to \cB_k$ described above, and let $\widetilde{\iota}_k: \cB^k \to \cB$ be the canonical inclusion of $\cB^k$ into the amalgamated free product, so that $\psi := \widetilde{\iota}_k \circ \varphi_k: \widetilde{\cA} \to \cB$ is independent of $k$.  To obtain the embedding from $\cA$ into $\cB$, fix $k_0 \in K$ arbitrary and let $\iota: \cA \to \cB$ be given by $\iota = \widetilde{\iota}_k \circ \iota_k$.  (This embedding will not be unique since our requirements in the theorem statement only determine $\iota|_{\cA_{t_0}}$.)  The filtration $\cB_t$ is defined as follows:
			\[
			\cB_t = \psi \circ \lambda_1(\cA_t) \text{ for } t \in [t_0,t_1)
			\]
			and
			\[
			\cB_t = \mathrm{W}^*(\widetilde{\iota}_k(\cB_t^k): k \in K) \text{ for } t \in [t_1,T].
			\]
			This is a valid filtration since we assumed that $\iota_k(\cA_{t_1}) \subseteq \cB_{t_1}^k$.  The condition that $\iota(\cA_t) \subseteq \cB_t$ for $t \in [t_0,t_1)$ is immediate since $\iota|_{\cA_{t_1}} = \psi|_{\cA_{t_1}}$.  Similarly, for $t \in [t_1,T]$, we have $\widetilde{\iota}_k(\cB_t^k) \subseteq \cB_t$ by construction.
			
			Define a $d$-variable free Brownian motion $(S_t)_{t \in [t_0,T]}$ as follows:
			\begin{itemize}
				\item Set $S_t = \psi(S_t^0)$ for $t \in [t_0,t_1]$, which is well-defined since $S_t \in \widetilde{\cA}_{t_1}$.
				\item Set $S_t = S_{t_0} + \psi(Z_t)$ for $t \in [t_1,T]$.  Note that $S_t = S_{t_0} + \widetilde{\iota}_k(S_t^k)$ as well since $\psi = \widetilde{\iota}_k \circ \varphi_k$.
			\end{itemize}
			To check that this is actually a free Brownian motion, note that $(S_t: t \in [t_0,t_1])$ and $(S_t - S_{t_1}: t \in [t_1,T])$ are free Brownian motions, and by construction, $(S_t - S_{t_1}: t \in [t_1,T]) = (\psi(Z_t): t \in [t_1,T])$ is freely independent of $\psi(\cA_{t_0})$ which contains $(S_t: t \in [t_0,t_1])$.  From here it is easy to check the free independence of increments in $[t_0,T]$ by splitting into cases based on the intervals $[t_0,t_1]$ and $[t_1,T]$.
			
			Finally, we need to check that the free Brownian motion $(S_t)_{t \in [t_0,T]}$ is compatible with the filtration $(\cB_t)_{t \in [t_0,T]}$.  To see adaptedness, note that by construction, we have
			\[
			S_t = \psi(S_t^0) \in \psi(\cA_t) = \cB_t \text{ for } t \in [t_0,t_1]
			\]
			and
			\[
			S_t = \widetilde{\iota}_k(S_t^k) + \psi(S_{t_1}^0) \in \widetilde{\iota}_k(\cB_t^k) + \psi(\cA_{t_1}) = \widetilde{\iota}_k(\cB_t^k) \text{ for } t \in [t_1,T].
			\]
			Then we have to check that for $t \in [t_0,T]$, we have free independence of $\cB_t$ and $\mathrm{W}^*(S_s - S_t: s \in [t,T])$.
			
			First consider the case where $t \in [t_0,t_1]$.  Then $\cB_t = \psi(\cA_t)$ and $S_s - S_t$ for $s \geq t$ are both in the image of $\psi$.  Specifically, we have $S_s = \psi(\widetilde{S}_s^0)$, where
			\[
			\widetilde{S}_t^0 = \begin{cases} S_t^0, & t \in [t_0,t_1] \\ S_{t_1}^0 + Z_t, & t \in [t_1,T]. \end{cases}
			\]
			Since $\psi$ is trace-preserving, it suffices to show that $\cA_t$ and $\mathrm{W}^*(\widetilde{S}_s^0 - \widetilde{S}_t^0)$ are freely independent in $\widetilde{\cA}$.  For each of notation, for $t_1 \leq a \leq b \leq T$, let
			\[
			\mathcal{C}_{a,b} = \mathrm{W}^*(\widetilde{S}_s^0 - \widetilde{S}_a^0: s \in [a,b]).
			\]
			Since $S_t^0$ is compatible with the filtration on $\cA$, we know that $\cA_t$ and $\mathcal{C}_{t,t_1}$ are freely independent.  Furthermore, by the construction of $\widetilde{\cA}$ as a free product, we see that $\mathcal{C}_{t_1,T}$ is freely independent of $\cA_{t_1}$, and in particular freely independent of the $\mathrm{W}^*$-algebra generated by $\cA_t$ and $\mathcal{C}_{t,t_1}$.  By the associativity property of free products (Lemma \ref{lem: associativity}), it follows that $\cA_t$ and $\mathcal{C}_{t,t_1}$ and $\mathcal{C}_{t_1,T}$ are freely independent.  Then by associativity again, $\cA_t$ is freely independent of the $\mathrm{W}^*$-algebra generated by $\mathcal{C}_{t,t_1}$ and $\mathcal{C}_{t_1,T}$, which is $\mathcal{C}_{t,T}$.  This is the claim we wanted to prove.
			
			Next, consider the case where $t \in [t_1,T]$, which is the more difficult case.  Using the notation above, we need to show that $\psi(\mathcal{C}_{t,T})$ is freely independent of $\mathcal{B}_t$.  Let $\widetilde{\iota}_k(\cB_t^k) \vee \psi(\mathcal{C}_{t,T})$ be the von Neumann subalgebra of $\cB$ generated by $\widetilde{\iota}_k(\cB_t^k)$ and $\psi(\mathcal{C}_{t,T})$, which is canonically isomorphic to the free product $\widetilde{\iota}_k(\cB_t^k) * \psi(\mathcal{C}_{t,T})$.  By Lemma \ref{lem: free product embedding}, since
			\[
			\psi \circ \lambda_1(\widetilde{\cA}) \subseteq \widetilde{\iota}_k(\cB_t^k) \vee \psi(\mathcal{C}_{t,T}) \subseteq \widetilde{\iota}_k(\cB_k),
			\]
			we see that the algebras $(\widetilde{\iota}_k(\cB_t^k) \vee \psi(\mathcal{C}_{t,T}))_{k \in K}$ are freely independent with amalgamation over $\psi \circ \lambda_1(\widetilde{\cA})$.   Hence, by Corollary \ref{cor: independence and embeddings},
			\[
			\cB_t \vee \psi(\mathcal{C}_{t,T}) = \bigvee_{k \in K} (\widetilde{\iota}_k(\cB_{t,k}) \vee \psi(\mathcal{C}_{t,T}))
			\]
			is canonically isomorphic to the free product with amalgamation over $\psi \circ \lambda_1(\widetilde{\cA})$
			\[
			\median_{\widetilde{\cA}} (\cB_{t,k} \vee \varphi_k(\mathcal{C}_{t,T}))_{k \in K}.
			\]
			Furthermore, by the choice of filtration on $\cB_k$, we have that $\cB_t^k$ is freely independent of $\mathrm{W}^*(S_s^k - S_t^k: s \in [t,T]) = \varphi_k(\mathcal{C}_{t,T})$.  Hence, we obtain an isomorphism
			\[
			\sigma: \cB \supseteq \cB_t \vee \psi(\mathcal{C}_{t,T}) \to \median_{\widetilde{\cA}} (\cB_{t,k} * \mathcal{C}_{t,T})_{k \in K}
			\]
			such that, writing $\sigma_k$ for the canonical embedding of $\cB_{t,k} * \mathcal{C}_{t,T}$ into the free product, we have
			\[
			\sigma \circ \widetilde{\iota}_k|_{\cB_t^k} = \sigma_k|_{\cB_t^k}
			\]
			and
			\[
			\sigma \circ \psi|_{\mathcal{C}_{t,T}} = \sigma_k \circ \varphi_k|_{\mathcal{C}_{t,T}}.
			\]
			Finally, recall by construction that
			\[
			\widetilde{\cA} = \cA_{t_1} * \mathcal{C}_{t_1,T} \cong \mathcal{A}_{t_1} * \mathcal{C}_{t_1,t} * \mathcal{C}_{t,T} = \mathcal{D} * \mathcal{C}_{t,T},
			\]
			where we set $\mathcal{D} = \mathcal{A}_{t_1} * \mathcal{C}_{t_1,t}$.  Hence, we have
			\[
			\cB_t \vee \psi(\mathcal{C}_{t,T}) \cong \median_{\mathcal{D} * \mathcal{C}_{t,T}} (\cB_t^k * \mathcal{C}_{t,T})_{k \in K},
			\]
			where the isomorphism respects the inclusions of various subalgebras as described above.  We claim that there is a unique isomorphism
			\[
			\Phi: \median_{\mathcal{D} * \mathcal{C}_{t,T}} (\cB_t^k * \mathcal{C}_{t,T})_{k \in K} \to \left( \median_{\mathcal{D}} (\cB_t^k)_{k \in K} \right) \median \mathcal{C}_{t,T},
			\]
			which respects the inclusions of various algebras in the way one would expect, namely, the following diagrams commute for each $i \in K$:
			\[
			\begin{tikzcd}
				& \cB_t^i \arrow{dl}{\rho_i} \arrow{dr}{\lambda_i'} & \\
				\cB_t^i \median \cC_{t,T} \arrow{d}{\tilde{\lambda}_i'} & & \median_{\cD} (\cB_t^k)_{k \in K} \arrow{d}{\rho} \\
				\median_{\cD \median \cC_{t,T}} \left(\cB_t^k \median \cC_{t,T} \right)_{k \in K} \arrow{rr}{\Phi} & & \left( \median_{\cD} (\cB_t^k)_{k \in K} \right) \median \cC_{t,T},
			\end{tikzcd}
			\]
			and
			\[
			\begin{tikzcd}
				\cB_t^i \median \cC_{t,T} \arrow{d}{\tilde{\lambda}_i'} & \cC_{t,T} \arrow{l}{\sigma_i'} \arrow{d}{\sigma'} \\
				\median_{\cD \median \cC_{t,T}} \left(\cB_t^k \median \cC_{t,T} \right)_{j \in J} \arrow{r}{\Phi} & \left( \median_{\cD} (\cB_t^k)_{k \in K} \right) \median \cC_{t,T},
			\end{tikzcd}
			\]
			where the maps $\rho_i$, $\lambda_i'$, $\tilde{\lambda}_i'$, $\rho$, $\sigma_i'$, $\sigma'$ in these diagrams are the canonical inclusions associated to the respective free product constructions (see Lemma \ref{lem: uniqueness of free product}).  We give a proof of this claim in Lemma \ref{lem: AFP manipulation} in the appendix.  Overall, we have the following chain of isomorphisms:
			\[
			\cB \supseteq \cB_t \vee \psi(\mathcal{C}_{t,T}) \xrightarrow{\sigma} \median_{\tilde{\cA}} (\cB_{t,k} * \mathcal{C}_{t,T})_{k \in K} \to \median_{\mathcal{D} * \mathcal{C}_{t,T}} (\cB_t^k * \mathcal{C}_{t,T})_{k \in K} \xrightarrow{\Phi} \left( \median_{\mathcal{D}} (\cB_t^k)_{k \in K} \right) \median \mathcal{C}_{t,T},
			\]
			which all respect the inclusions of $\cB_t^k$ and $\mathcal{C}_{t,T}$ in the way we expect.  In the right-most expression $\left( \median_{\mathcal{D}} (\cB_t^k)_{k \in K} \right) * \mathcal{C}_{t,T}$, the subalgebra generated by the images of $(\cB_t^k)_{k \in K}$ is freely independent of $\mathcal{C}_{t,T}$; hence, mapping these subalgebras back into $\cB$ by the chain of isomorphisms, we see that $\bigvee_{k \in K} \widetilde{\iota}_k(\cB_{t,k})$ is freely independent of $\psi(\mathcal{C}_{t,T})$ in $\cB$, which is what we wanted to prove.
		\end{proof}
		
		We continue with the proof of Lemma \ref{lem: arranging initial filtration} which we restate here.
		
		\begin{lemma}[{Lemma \ref{lem: arranging initial filtration}}] \label{lem: arranging initial filtration appendix}
			Suppose that Assumption \textbf{A} holds.  Let $\cA$ be a tracial $\mathrm{W}^*$-algebra, let $x_0 \in L^2(\cA)_{\sa}^d$, and let $t_0 \in [0,T]$.  Then for every $\epsilon > 0$, there exists a tracial $\mathrm{W}^*$-algebra $\cB$, a tracial $\mathrm{W}^*$-embedding $\iota: \cA \to \cB$, and a control policy $\alpha \in \mathbb{A}_{\mathcal{B},\iota x_0}^{t_0,T}$ associated with a filtration $(\cB_t)_{t \in [t_0,T]}$ and compatible $d$-variable free Brownian motion $(S_t)_{t \in [t_0,T]}$, such that
			\begin{enumerate}
				\item $\iota(\cA) \subseteq \cB_{t_0}$
				\item $\mathbb{E}\left[ \int_{t_0}^T L_{\mathcal{B}}(X_t[\widetilde \alpha],\alpha_t)\,dt + g_{\mathcal{B}}(X_T[\alpha]) \right] \leq 
				\overline{V}_{\mathcal{A}}(t_0,x_0) + \epsilon$.
			\end{enumerate}
		\end{lemma}
		
		\begin{proof}
			First, by definition of $\overline{V}_{\mathcal{A}}(t_0,x_0)$, there exists some tracial $\mathrm{W}^*$-embedding $\iota_0: \cA \rightarrow  \mathcal{B}^0$ and some control policy $\widetilde \alpha^0$ in $\mathbb{A}_{\mathcal{B}^0,{\iota_0}(x_0)}^{t_0,T}$ such that
			\[
			\mathbb{E}\left[ \int_{t_0}^T L_{\mathcal{B}^0}(X_t[\widetilde \alpha^0],\alpha_t^0)\,dt + g_{\mathcal{B}^0}(X_T[\alpha^0]) \right] \leq 
			\overline{V}_{\mathcal{A}}(t_0,x_0) + \epsilon.
			\]
			Let $(\mathcal B_t^0)_{t \in [t_0,T]}$ be the corresponding filtration and $(S_t^0)_{t \in [t_0,T]}$ the corresponding  free Brownian motion.
			
			Next, we define a tracial von Neumann algebra $\mathcal{B}^1$ as follows.  By Voiculescu's free Gaussian functor construction, there exists a tracial von Neumann algebra $\mathcal{C}$ generated by a $d$-variable free Brownian motion $(Z_t)_{t \in [t_0,T]}$ with initial condition $Z_{t_0} = 0$.  For $t_0 \leq a \leq b \leq T$, let $\mathcal{C}_{a,b}$ be the von Neumann subalgebra of $\mathcal{C}$ generated by $(Z_s - Z_a)_{s \in [a,b]}$.  Let $\mathcal{B}^1$ be the free product $\mathcal{A} * \mathcal{C}$.  Let
			\[
			\iota_1: \mathcal{A} \to \mathcal{B}^1, \qquad \psi: \mathcal{C} \to \mathcal{B}^1
			\]
			be the canonical inclusions from the free product construction.  Define $S_t^1 = \psi(Z_t)$ for $t \in [t_0,T]$.  Define the filtration
			\[
			\mathcal{B}_t^1 = \iota_1(\mathcal{A}) \vee \psi(\mathcal{C}_{t_0,t}),\qquad t\in [t_0,T].
			\]
			We claim that $(S_t^1)_{t \in [t_0,T]}$ is compatible with the filtration $(\mathcal{B}_t^1)_{t \in [t_0,T]}$.  To prove this, first note that it is adapted because $S_t^1 = \psi(Z_t) \in \psi(\mathcal{C}_{t_0,t}) \subseteq \mathcal{B}_t^1$.  Next, we show that for each $t \in [t_0,T]$, the algebra $\mathrm{W}^*(S_s^1 - S_t^1: s \in [t,T])$ is freely independent of $\mathcal{B}_t^1$.  Note that $\mathrm{W}^*(S_s^1 - S_t^1: s \in [t,T]) = \psi(\mathcal{C}_{t,T})$.  Recall that $\mathcal{C}_{t_0,t}$ and $\mathcal{C}_{t,T}$ are freely independent.  Hence, $\psi(\mathcal{C}_{t_0,t})$ and $\psi(\mathcal{C}_{t,T})$ are freely independent of each other, and $\iota_1(\mathcal{A})$ is freely independent of $\psi(\mathcal{C}_{t_0,t}) \vee \psi(\mathcal{C}_{t,T})$.  Therefore, by associativity of free products (Lemma \ref{lem: associativity}), $\iota_1(\mathcal{A})$ and $\psi(\mathcal{C}_{t_0,t})$ and $\psi(\mathcal{C}_{t,T})$ are freely independent.  Using associativity again, it follows that $\psi(\mathcal{C}_{t,T})$ is freely independent of $\iota_1(\mathcal{A}) \vee \psi(\mathcal{C}_{t_0,t}) = \mathcal{B}_t^1$, which is what we wanted to prove.
			
			Now we apply Lemma \ref{lem: amalgamation of Brownian motions} using the index set $K = \{0,1\}$ and using the algebras $\mathcal{B}^0$ and $\mathcal{B}^1$ with the associated filtrations and Brownian motions above.  Set $t_1 = t_0$, and for the filtration in $\cA$ on the degenerate time interval $[t_0,t_0]$, which reduces to just a single subalgebra, use $\cA_{t_0} = \mathrm{W}^*(x_0)$.  Then $\iota_0(\cA_{t_0}) \subseteq \cB_{t_0}$ because the definition of admissible control policies requires that $x_0 \in L^2(\cB_{t_0})_{\sa}^d$.  Also, $\iota_1(\cA_{t_0}) \subseteq \iota_1(\cA) = \cB_{t_0}^1$ by construction.  Let $\cB$, $\widetilde \iota_0$, $\widetilde \iota_1$, $(\cB_t)_{t \in [t_0,T]}$, and $(S_t)_{t \in [t_0,T]}$ be as in the previous lemma.  By that lemma, we may take the embedding $\iota: \cA \to \cB$ to be given by $\iota = \widetilde \iota_1 \circ \iota_1$.  Thus, by (5) of the previous lemma,
			\[
			\iota(\cA) = \widetilde \iota_1 \circ \iota_1(\cA) = \widetilde \iota_1(\cB_{t_0}^1) \subseteq \cB_{t_0},
			\]
			which verifies the first condition we wanted to prove.  Let $\alpha := \widetilde \iota_0 \circ \alpha^0$.  By construction,
			\[
			\alpha_t = \widetilde \iota_0(\alpha_t^0) \in \widetilde \iota_0(\cB_t^0) \subseteq \cB_t,
			\]
			so that $((\alpha_t)_{t \in [t_0,T]}, (\cB_t)_{t \in [t_0,T]},(S_t)_{t \in [t_0,T]})$ is an admissible control policy in $\mathcal{B}$. Note that $X_t[\alpha] = \widetilde \iota_0(X_t[\alpha^0])$.  Since $L$ and $g$ are tracial $\mathrm{W}^*$-functions, we have
			\begin{align*}
				\mathbb{E}\left[ \int_{t_0}^T L_{\mathcal{B}}(X_t[\alpha],\alpha_t)\,dt + g_{\mathcal{B}}(X_T[\alpha]) \right] &= \mathbb{E}\left[ \int_{t_0}^T L_{\mathcal{B}^0}(X_t[\alpha^0],\alpha_t^0)\,dt + g_{\mathcal{B}^0}(X_T[\alpha^0]) \right] \\
				&\leq V_{\cA}(t_0,x_0) + \epsilon,
			\end{align*}
			which verifies the second condition that we wanted to prove.
		\end{proof}

		\bibliography{FreeViscBibJan21}

\begin{thebibliography}{10}

\bibitem{AmbrosioF2014}
Luigi Ambrosio and Jin Feng.
\newblock On a class of first-order hamilton-jacobi equations in metric spaces.
\newblock {\em J. Diff. Eq.}, 256(7):2194–2245, 2014.

\bibitem{anderson2010introduction}
Greg~W Anderson, Alice Guionnet, and Ofer Zeitouni.
\newblock {\em An introduction to random matrices}.
\newblock Number 118 in Cambridge series in advanced mathematics. Cambridge
  university press, 2010.

\bibitem{freesde2}
Michael Anshelevich.
\newblock It{\^o} formula for free stochastic integrals.
\newblock {\em Journal of Functional Analysis}, 188(1):292--315, 2002.

\bibitem{BansilMeszM}
Mohit Bansil, Alp\'ar~R M\'esz\'aros, and Chenchen Mou.
\newblock Global well-posedness of displacement monotone degenerate mean field
  games master equations.
\newblock {\em arXiv:2308.16167v2}, 2024.

\bibitem{BayraktarEkrenZhang}
Erhan Bayraktar, Ibrahim Ekren, and Xin Zhang.
\newblock Comparison of viscosity solutions for a class of second order pdes on
  the wasserstein space.
\newblock {\em ArXiv:2309.05040}, 2023.

\bibitem{bertucci2022spectral}
Charles Bertucci, M{\'e}rouane Debbah, Jean-Michel Lasry, and Pierre-Louis
  Lions.
\newblock A spectral dominance approach to large random matrices.
\newblock {\em Journal de Math{\'e}matiques Pures et Appliqu{\'e}es},
  164:27--56, 2022.

\bibitem{biane1997freebrownian}
Philippe Biane.
\newblock {\em Free Brownian Motion, Free Stochastic Calculus and Random
  Matrices}, volume~12 of {\em Fields Institute Communications}, pages 1--19.
\newblock American Mathematical Society, Providence, RI, 1997.

\bibitem{biane2003large}
Philippe Biane, Mireille Capitaine, and Alice Guionnet.
\newblock Large deviation bounds for matrix brownian motion.
\newblock {\em Inventiones mathematicae}, 152:433--459, 2003.

\bibitem{biane1998stochastic}
Philippe Biane and Roland Speicher.
\newblock Stochastic calculus with respect to free brownian motion and analysis
  on wigner space.
\newblock {\em Probability theory and related fields}, 112:373--409, 1998.

\bibitem{biane2001free}
Philippe Biane and Dan Voiculescu.
\newblock A free probability analogue of the wasserstein metric on the
  trace-state space.
\newblock {\em Geometric \& Functional Analysis GAFA}, 11(6):1125--1138, 2001.

\bibitem{brown2008c*algebras}
Nathaniel~P. Brown and Narutaka Ozawa.
\newblock {\em $\mathrm{C}^*$-algebras and finite-dimensional approximations},
  volume~88 of {\em Graduate Studies in Mathematics}.
\newblock American Mathematical Society, Providence, 2008.

\bibitem{cardaliaguet2019master}
Pierre Cardaliaguet, Fran{\c{c}}ois Delarue, Jean-Michel Lasry, and
  Pierre-Louis Lions.
\newblock {\em The master equation and the convergence problem in mean field
  games:(ams-201)}.
\newblock Princeton University Press, 2019.

\bibitem{carlen2017gradient}
Eric~A. Carlen and Jan Maas.
\newblock Gradient flow and entropy inequalities for quantum markov semigroups
  with detailed balance.
\newblock {\em Journal of Functional Analysis}, 273(5):1810--1869, 2017.

\bibitem{carmona2020dyson}
Ren{\'e} Carmona, Mark Cerenzia, and Aaron~Zeff Palmer.
\newblock The dyson and coulomb games.
\newblock In {\em Annales Henri Poincar{\'e}}, volume~21, pages 2897--2949.
  Springer, 2020.

\bibitem{carmona2018probabilistic}
Ren{\'e} Carmona, Fran{\c{c}}ois Delarue, et~al.
\newblock {\em Probabilistic theory of mean field games with applications I}.
\newblock Springer, 2018.

\bibitem{chen2020matrix}
Yongxin Chen, Wilfrid Gangbo, Tryphon~T Georgiou, and Allen Tannenbaum.
\newblock On the matrix monge--kantorovich problem.
\newblock {\em European Journal of Applied Mathematics}, 31(4):574--600, 2020.

\bibitem{chen2017matrix}
Yongxin Chen, Tryphon~T Georgiou, and Allen Tannenbaum.
\newblock Matrix optimal mass transport: a quantum mechanical approach.
\newblock {\em IEEE Transactions on Automatic Control}, 63(8):2612--2619, 2017.

\bibitem{cosso2024master}
Andrea Cosso, Fausto Gozzi, Idris Kharroubi, Huy{\^e}n Pham, and Mauro
  Rosestolato.
\newblock Master bellman equation in the wasserstein space: Uniqueness of
  viscosity solutions.
\newblock {\em Transactions of the American Mathematical Society},
  377(01):31--83, 2024.

\bibitem{vis3}
Michael~G Crandall, Lawrence~C Evans, and P-L Lions.
\newblock Some properties of viscosity solutions of hamilton-jacobi equations.
\newblock {\em Transactions of the American Mathematical Society},
  282(2):487--502, 1984.

\bibitem{vis2}
Michael~G Crandall and Pierre-Louis Lions.
\newblock Viscosity solutions of hamilton-jacobi equations.
\newblock {\em Transactions of the American mathematical society},
  277(1):1--42, 1983.

\bibitem{crandall1985hamilton}
Michael~G Crandall and Pierre-Louis Lions.
\newblock Hamilton-jacobi equations in infinite dimensions i. uniqueness of
  viscosity solutions.
\newblock {\em Journal of functional analysis}, 62(3):379--396, 1985.

\bibitem{Crandall-Lions85}
Michael~G Crandall and Pierre-Louis Lions.
\newblock Viscosity solutions of hamilton-jacobi equations in infinite
  dimensions. i. uniqueness of viscosity solutions.
\newblock {\em J. Funct. Anal.}, 62, 1985.

\bibitem{crandall1986hamilton}
Michael~G Crandall and Pierre-Louis Lions.
\newblock Hamilton-jacobi equations in infinite dimensions. ii. existence of
  viscosity solutions.
\newblock {\em Journal of Functional Analysis}, 65(3):368--405, 1986.

\bibitem{visinf1}
Michael~G Crandall and Pierre-Louis Lions.
\newblock Hamilton-jacobi equations in infinite dimensions, iii.
\newblock {\em Journal of functional analysis}, 68(2):214--247, 1986.

\bibitem{Crandall-Lions86}
Michael~G Crandall and Pierre-Louis Lions.
\newblock Viscosity solutions of hamilton-jacobi equations in infinite
  dimensions. ii. existence of viscosity solutions.
\newblock {\em J. Funct. Anal.}, 65, 1986.

\bibitem{Crandall-Lions86b}
Michael~G Crandall and Pierre-Louis Lions.
\newblock Viscosity solutions of hamilton-jacobi equations in infinite
  dimensions. iii. existence of viscosity solutions.
\newblock {\em J. Funct. Anal.}, 65, 1986.

\bibitem{crandall1990viscosity}
Michael~G Crandall and PL~Lions.
\newblock Viscosity solutions of hamilton-jacobi equations in infinite
  dimensions. iv. hamiltonians with unbounded linear terms.
\newblock {\em Journal of Functional Analysis}, 90(2):237--283, 1990.

\bibitem{crandall1991viscosity}
Michael~G Crandall and PL~Lions.
\newblock Viscosity solutions of hamilton-jacobi equations in infinite
  dimensions. v. unbounded linear terms and b-continuous solutions.
\newblock {\em Journal of functional analysis}, 97(2):417--465, 1991.

\bibitem{Dabrowski2017Laplace}
Yoann Dabrowksi.
\newblock A {L}aplace principle for {H}ermitian {B}rownian motion and free
  entropy {I}: the convex functional case.
\newblock arXiv:1604.06420, 2017.

\bibitem{dabrowski2010path}
Yoann Dabrowski.
\newblock A non-commutative path space approach to stationary free stochastic
  differential equations.
\newblock arxiv:1006.4351, 2010.

\bibitem{depalma2021quantumchannels}
Giacomo {De Palma} and Dario Trevisan.
\newblock Quantum optimal transport with quantum channels.
\newblock {\em Ann. Henri Poincar{\'e}}, 22:3199--3234, 2021.

\bibitem{driver2013large}
Bruce~K. Driver, Brian~C. Hall, and Todd Kemp.
\newblock The large-$n$ limit of the {S}egal-{B}argmann transform on $u_n$.
\newblock {\em Journal of Functional Analysis}, 265(11):2585 -- 2644, 2013.

\bibitem{duvenhage2022quadratic}
Rocco Duvenhage.
\newblock Quadratic {W}asserstein metrics for von {N}eumann algebras via
  transport plans.
\newblock {\em Journal of Operator Theory}, 88(2):289--308, 2022.

\bibitem{dyson1962brownian}
Freeman~J Dyson.
\newblock A brownian-motion model for the eigenvalues of a random matrix.
\newblock {\em Journal of Mathematical Physics}, 3(6):1191--1198, 1962.

\bibitem{vis1}
Lawrence~C Evans.
\newblock On solving certain nonlinear partial differential equations by
  accretive operator methods.
\newblock {\em Israel Journal of Mathematics}, 36:225--247, 1980.

\bibitem{fabbri2017stochastic}
Giorgio Fabbri, Fausto Gozzi, and Andrzej Swiech.
\newblock Stochastic optimal control in infinite dimension.
\newblock {\em Probability and Stochastic Modelling. Springer}, 2017.

\bibitem{feliciangeli2023non}
Dario Feliciangeli, Augusto Gerolin, and Lorenzo Portinale.
\newblock A non-commutative entropic optimal transport approach to quantum
  composite systems at positive temperature.
\newblock {\em Journal of Functional Analysis}, page 109963, 2023.

\bibitem{gangbo2022duality}
Wilfrid Gangbo, David Jekel, Kyeongsik Nam, and Dimitri Shlyakhtenko.
\newblock Duality for optimal couplings in free probability.
\newblock {\em Communications in Mathematical Physics}, 396(3):903--981, 2022.

\bibitem{GangboMayorgaS}
Wilfrid Gangbo, Sergio Mayorga, and Andrzej Swiech.
\newblock Finite dimensional approximations of hamilton-jacobi-bellman
  equations in spaces of probability measures.
\newblock {\em SIAM J. Math. Anal.}, 53(2):1320--1356, 2021.

\bibitem{gangboMesz}
Wilfrid Gangbo and Alp\'ar~R. M\'esz\'aros.
\newblock Global well-posedness of master equations for deterministic
  displacement convex potential mean field games.
\newblock {\em Communications on Pure and Applied Mathematics},
  75(12):2685--2801, 2022.

\bibitem{gangbo2022mean}
Wilfrid Gangbo, Alp{\'a}r~R M{\'e}sz{\'a}ros, Chenchen Mou, and Jianfeng Zhang.
\newblock Mean field games master equations with nonseparable hamiltonians and
  displacement monotonicity.
\newblock {\em The Annals of Probability}, 50(6):2178--2217, 2022.

\bibitem{GangboSwiech2015}
Wilfrid Gangbo and Andrzej Swiech.
\newblock Existence of a solution to an equation arising from mean field games.
\newblock {\em J. Differential Equations}, 259(11):6573--6643, 2015.

\bibitem{GangboSwiech}
Wilfrid Gangbo and Andrzej Swiech.
\newblock Metric viscosity solution to hamilton-jacobi equations depending on
  local slopes.
\newblock {\em Calculus of Variations and Partial Differential Equations},
  54(1):1183 -- 1218, 2015.

\bibitem{gangboNguyenT}
Wilfrid Gangbo, Nguyen Truyen, and Adrian Tudorascu.
\newblock Hamilton-jacobi equations in the wasserstein space.
\newblock {\em Meth. Appl. Anal.}, 15(2):155–184, 2008.

\bibitem{gangbo2019}
Wilfrid Gangbo and Adrian Tudorascu.
\newblock On differentiability in the wasserstein space and well-posedness for
  hamilton--jacobi equations.
\newblock {\em Journal de Math{\'e}matiques Pures et Appliqu{\'e}es},
  125:119--174, 2019.

\bibitem{gomes2014meanfield}
Diogo~A. Gomes and Jo{\~a}o Sa{\'u}de.
\newblock Mean field games models--a brief survey.
\newblock {\em Dyn Games Appl}, 4:110--154, 2014.

\bibitem{guionnet2009free}
Alice Guionnet and Dimitri Shlyakhtenko.
\newblock Free diffusions and matrix models with strictly convex interaction.
\newblock {\em Geometric and Functional Analysis}, 18(6):1875--1916, 2009.

\bibitem{guionnet2014freemonotone}
Alice Guionnet and Dimitri Shlyakhtenko.
\newblock Free monotone transport.
\newblock {\em Inventiones Mathematicae}, 197(3):613--661, 09 2014.

\bibitem{guionnet2000concentration}
Alice Guionnet and Ofer Zeitouni.
\newblock {Concentration of the Spectral Measure for Large Matrices}.
\newblock {\em Electronic Communications in Probability}, 5:119--136, 2000.

\bibitem{hiai2006free}
F.~Hiai and Y.~Ueda.
\newblock Free transportation cost inequalities for noncommutative
  multi-variables.
\newblock {\em Infinite Dimensional Analysis, Quantum Probability, and Related
  Topics}, 9:391--412, 2006.

\bibitem{hiai2004free}
Fumio Hiai, D{\'e}nes Petz, and Yoshimichi Ueda.
\newblock Free transportation cost inequalities via random matrix
  approximation.
\newblock {\em Probab. Theory Related Fields}, 130(2):199--221, 2004.

\bibitem{houdayer2007freeproducts}
Cyril Houdayer.
\newblock On some free products of von neumann algebras which are free
  araki-woods factors.
\newblock {\em Int. Math. Res. Not. IMRN}, 2007(23):21, 2007.

\bibitem{Menon.et.al}
Ching-Peng Huang, Dominik Inauen, and Govind Menon.
\newblock Motion by mean curvature and dyson brownian motion.
\newblock {\em Electron. Commun. Probab.}, 28(34):1--10, 2023.

\bibitem{huang2006large}
Minyi Huang, Roland~P Malham{\'e}, and Peter~E Caines.
\newblock Large population stochastic dynamic games: closed-loop
  {M}c{K}ean-{V}lasov systems and the {N}ash certainty equivalence principle.
\newblock {\em Communications in information and systems}, 2006.

\bibitem{ishii1993viscosity}
Hitoshi Ishii.
\newblock Viscosity solutions of nonlinear second-order partial differential
  equations in hilbert spaces.
\newblock {\em Communications in partial differential equations},
  18(3-4):601--650, 1993.

\bibitem{JakobsenRutkowski}
Espen~Robstad Jakobsen, Artur Rutkowski, and Rocco Duvenhage.
\newblock The master equation for mean field game systems with fractional and
  nonlocal diffusions.
\newblock {\em arXiv:2305.18867}, 88(2):289--308, 2022.

\bibitem{jekel2020elementary}
David Jekel.
\newblock An elementary approach to free entropy theory for convex potentials.
\newblock {\em Analysis \& PDE}, 13(8):2289--2374, 2020.

\bibitem{jekel2024optimal}
David Jekel.
\newblock Optimal transport for types and convex analysis for definable
  predicates in tracial w⁎-algebras.
\newblock {\em Journal of Functional Analysis}, 287(9):110583, 2024.

\bibitem{jekel2022tracial}
David Jekel, Wuchen Li, and Dimitri Shlyakhtenko.
\newblock Tracial smooth functions of non-commuting variables and the free
  wasserstein manifold.
\newblock {\em Dissertationes Mathematicae}, 580:1--150, 2022.

\bibitem{jekel2020operad}
David Jekel and Weihua Liu.
\newblock An operad of non-commuative independences defined by trees.
\newblock {\em Dissertationes Mathematicae}, 553:1--100, 2020.

\bibitem{jekel2023martingale}
David~A Jekel, Todd~A Kemp, and Evangelos~A Nikitopoulos.
\newblock A martingale approach to noncommutative stochastic calculus.
\newblock {\em arXiv preprint arXiv:2308.09856}, 2023.

\bibitem{jekel2020evolution}
David~Andrew Jekel.
\newblock {\em Evolution equations in non-commutative probability}.
\newblock University of California, Los Angeles, 2020.

\bibitem{lasry2007mean}
Jean-Michel Lasry and Pierre-Louis Lions.
\newblock Mean field games.
\newblock {\em Japanese journal of mathematics}, 2(1):229--260, 2007.

\bibitem{lions1988viscosity}
Pierre-Louis Lions.
\newblock Viscosity solutions of fully nonlinear second-order equations and
  optimal stochastic control in infinite dimensions. part i: The case of
  bounded stochastic evolutions.
\newblock 1988.

\bibitem{meckes2013}
Elizabeth Meckes and Mark Meckes.
\newblock Spectral measures of powers of random matrices.
\newblock {\em Electron. Commun. Probab.}, 18:1--13, 2013.

\bibitem{menon2017complex}
Govind Menon.
\newblock The complex burgers’ equation, the hciz integral and the
  calogero-moser system.
\newblock {\em preprint}, 2017.

\bibitem{nica2006lectures}
Alexandru Nica and Roland Speicher.
\newblock {\em Lectures on the Combinatorics of Free Probability}.
\newblock London Mathematical Society Lecture Note Series. Cambridge University
  Press, 2006.

\bibitem{popa1993markov}
Sorin Popa.
\newblock Markov traces on universal jones algebras and subfactors of finite
  index.
\newblock {\em Invent Math}, 111:375--405, 1993.

\bibitem{segal1947irreducible}
Irving~E Segal.
\newblock Irreducible representations of operator algebras.
\newblock 1947.

\bibitem{shlyakhtenko2000entropy}
Dimitri Shlyakhtenko.
\newblock Free entropy with respect to a completely positive map.
\newblock {\em American Journal of Mathematics}, 122(1):45--81, 02 2000.

\bibitem{freesde1}
Roland Speicher.
\newblock A new example of ‘independence’and ‘white noise’.
\newblock {\em Probability theory and related fields}, 84(2):141--159, 1990.

\bibitem{speicher1998combinatorial}
Roland Speicher.
\newblock Combinatorial theory of the free product with amalgamation and
  operator-valued free probability theory.
\newblock {\em Mem. Amer. Math. Soc.}, 132(627), 1998.

\bibitem{voiculescu1991limit}
Dan Voiculescu.
\newblock Limit laws for random matrices and free products.
\newblock {\em Inventiones mathematicae}, 104(1):201--220, 1991.

\bibitem{voiculescu1985symmetries}
Dan-Virgil Voiculescu.
\newblock Symmetries of some reduced free product ${C}^*$-algebras.
\newblock In Huzihiro Araki, Calvin~C. Moore, {\c{S}}erban-Valentin Stratila,
  and Dan-Virgil Voiculescu, editors, {\em Operator Algebras and their
  Connections with Topology and Ergodic Theory}, pages 556--588. Springer
  Berlin Heidelberg, Berlin, Heidelberg, 1985.

\bibitem{voiculescu1993analogues1}
Dan-Virgil Voiculescu.
\newblock The analogues of entropy and {F}isher's information in free
  probability, {I}.
\newblock {\em Communications in Mathematical Physics}, 155(1):71--92, 1993.

\bibitem{voiculescu1994analogues2}
Dan-Virgil Voiculescu.
\newblock The analogues of entropy and of {F}isher's information in free
  probability, {II}.
\newblock {\em Inventiones Mathematicae}, 118:411--440, 1994.

\bibitem{voiculescu1995operations}
Dan-Virgil Voiculescu.
\newblock Operations on certain non-commutative operator-valued random
  variables.
\newblock In {\em Recent advances in operator algebras}, number 232 in
  Ast{\'e}risque, pages 243--275. Societe mathematique de {F}rance, 1995.

\bibitem{Voiculescu1998analogues5}
Dan-Virgil Voiculescu.
\newblock The analogues of entropy and of {F}isher's information in free
  probability {V}.
\newblock {\em Inventiones Mathematicae}, 132:189--227, 1998.

\bibitem{voiculescu1998strengthened}
Dan-Virgil Voiculescu.
\newblock A strengthened asymptotic freeness result for random matrices with
  applications to free entropy.
\newblock {\em Internat. Math. Res. Not. IMRN}, (1):41--63, 1998.

\bibitem{voiculescu1999analogues6}
Dan-Virgil Voiculescu.
\newblock The analogues of entropy and of fisher's information measure in free
  probability theory, {VI}: Liberation and mutual free information.
\newblock {\em Advances in Mathematics}, 146:101--166, 1999.

\bibitem{voiculescu2000coalgebra}
Dan-Virgil Voiculescu.
\newblock The coalgebra of the difference quotient and free probability.
\newblock {\em International Mathematics Research Notices}, 2000(2):79--106,
  2000.

\bibitem{voiculescu1992freerandom}
Dan-Virgil Voiculescu, Kenneth~J. Dykema, and Alexandru Nica.
\newblock {\em Free Random Variables}, volume~1 of {\em CRM Monograph Series}.
\newblock American Mathematical Society, Providence, 1992.

\bibitem{wigner1967random}
Eugene~P Wigner.
\newblock Random matrices in physics.
\newblock {\em SIAM review}, 9(1):1--23, 1967.

\bibitem{wishart1928generalised}
John Wishart.
\newblock The generalised product moment distribution in samples from a normal
  multivariate population.
\newblock {\em Biometrika}, pages 32--52, 1928.

\bibitem{yong1999stochastic}
Jiongmin Yong and Xun~Yu Zhou.
\newblock {\em Stochastic controls: Hamiltonian systems and HJB equations},
  volume~43.
\newblock Springer Science \& Business Media, 1999.

\end{thebibliography}
		\bibliographystyle{plain}

	\end{document}